\documentclass[11pt]{amsart} 

\usepackage{fullpage}

\usepackage{url}

\usepackage{amssymb}

\usepackage{cite}

\usepackage{graphicx}

\usepackage[usenames]{color}

\usepackage{accents}

\renewcommand{\a}{\alpha}
\renewcommand{\b}{\beta}
\newcommand{\g}{\gamma}
\renewcommand{\d}{\delta}
\newcommand{\D}{\Delta}
\newcommand{\e}{\varepsilon}
\newcommand{\f}{\varphi}
\newcommand{\s}{\sigma}
\newcommand{\Si}{\Sigma}
\renewcommand{\k}{\kappa}
\renewcommand{\l}{\lambda}
\renewcommand{\t}{\theta}
\renewcommand{\O}{\Omega}

\newcommand{\cF}{{\mathcal F}}
\newcommand{\cC}{{\mathcal C}}
\newcommand{\cM}{{\mathcal M}}
\newcommand{\cT}{{\mathcal T}}

\newcommand{\cB}{{\mathcal B}}
\newcommand{\cL}{{\mathcal L}}
\newcommand{\cE}{{\mathcal E}}
\newcommand{\cU}{{\mathcal U}}
\newcommand{\cV}{{\mathcal V}}

\newcommand{\cN}{{\mathcal N}}

\newcommand{\cH}{{\mathcal H}}
\newcommand{\cD}{{\mathcal D}}

\newcommand{\cI}{\mathcal I} 
\newcommand{\cJ}{{\mathcal J}}

\newcommand{\bR}{\mathbb R}
\newcommand{\bB}{\mathbb B}

\newcommand{\bH}{\mathbb H}

\newcommand{\bN}{\mathbb N}

\newcommand{\bI}{\mathbb I} 

\newcommand{\Ric}{\mathrm{Ric}}

\newcommand{\be}{\begin{equation}}
\newcommand{\ee}{\end{equation}}
\newcommand{\bes}{\begin{equation*}}
\newcommand{\ees}{\end{equation*}}

\newcommand{\tr}{\mathrm{tr}}

\newcommand{\beaa}{\begin{eqnarray*}}
\newcommand{\bea}{\begin{eqnarray}}
\newcommand{\beal}[1]{\begin{eqnarray}\label{#1}}
\newcommand{\bean}{\begin{eqnarray}\nonumber}
\newcommand{\beadl}[1]{\begin{deqarr}\label{#1}}
\newcommand{\eeadl}[1]{\arrlabel{#1}\end{deqarr}}
\newcommand{\eeal}[1]{\label{#1}\end{eqnarray}}
\newcommand{\eead}[1]{\end{deqarr}}
\newcommand{\eea}{\end{eqnarray}}
\newcommand{\eeaa}{\end{eqnarray*}}

\newcommand{\p}{\partial}

\renewcommand{\to}{\rightarrow}

\renewcommand{\phi}{\varphi}
\renewcommand{\epsilon}{\varepsilon}
\renewcommand{\hat}{\widehat}

\renewcommand{\>}{\rangle}

\newcommand{\dm}{{\partial M}}

\newcommand{\w}{\widetilde}

\theoremstyle{plain}
\newtheorem{lemma}{Lemma}[section]
\newtheorem{proposition}[lemma]{Proposition}
\newtheorem{theorem}[lemma]{Theorem}
\newtheorem{Theorem}{Theorem}

\newtheorem{corollary}[lemma]{Corollary}

\theoremstyle{remark}
\theoremstyle{definition}
\newtheorem{remark}[lemma]{Remark}

\newtheorem{definition}[lemma]{Definition}
\newtheorem{Definition}[Theorem]{Definition}

\def\blacksquare{\hbox to .60em {\vrule width .60em height .60em}}

\numberwithin{equation}{section}

\begin{document}

\title[ ]{On the initial boundary value problem for the vacuum Einstein equations and geometric uniqueness}

\author{Zhongshan An and Michael T. Anderson}

\address{Department of Mathematics, 
University of Connecticut,
Storrs, CT 06269}
\email{zhongshan.an@uconn.edu}

\address{Department of Mathematics, 
Stony Brook University,
Stony Brook, NY 11794}
\email{michael.anderson@stonybrook.edu}


\begin{abstract}
We formulate an initial boundary value problem (IBVP) for the vacuum Einstein equations by describing the boundary conditions of a spacetime metric in its associated gauge. This gauge is determined, equivariantly with respect to diffeomorphisms, by the spacetime metric. The vacuum spacetime metric $g$ and its associated gauge $\phi_g$ are solved simultaneously in local harmonic coordinates. Further we show that vacuum spacetimes satisfying fixed initial-boundary conditions and corner conditions are geometrically unique near the initial surface. Finally, in analogy to the solution of the Cauchy problem, we also construct a unique maximal globally hyperbolic solution of the IBVP.
\end{abstract}

\maketitle

\section{Introduction}

This article is concerned with the initial boundary value problem (IBVP) for the vacuum Einstein equations on a spacetime $M$ with boundary of the form $M = I\times S$, where $I=[0,1)$ and $S$ is a compact 3-manifold with non-empty boundary $\p S = \Si$. 
 The boundary $\p M$ of $M$ consists of two parts: the initial surface $S\cong \{0\}\times S$ and hypersurface $\cC \cong I\times \Si$.  These 3-manifolds are glued along their common boundary $\Si$ giving $M$ the structure of a manifold with corner.  We are interested in Lorentz metrics $g$ on $M$ such that the initial surface $S$ is spacelike and the hypersurface $\cC$ is timelike; such spacetimes $(M,g)$ are called \textit{ST-spacetimes} in \cite{Friedrich:2009}. 
The IBVP is the problem of finding ST-spacetimes $(M,g)$ satisfying the vacuum Einstein equations 
\be \label{vac}
\Ric_g = 0,
\ee
together with prescribed initial conditions along $S$ and boundary conditions along $\cC$. Throughout the paper we will use $\cT$ to denote a connected open subset in $M$, with $\{x \in M: t(x) \leq \tau\}\subset\cT$ for {some time function $t$ on $M$ and some $\tau>0$}. So the boundary $\p \cT$ contains the entire initial surface $S$ and  a portion $\cC\cap \cT$ of $\cC$. We will refer to such a subset $\cT$ as \textit{an ST-neighborhood (of the initial surface)}.

\subsection{Well-posed IBVP}
To place the problem in perspective, recall that the Cauchy problem for the equation \eqref{vac} has been well-understood since the 
fundamental work of Choquet-Bruhat \cite{Choquet-Bruhat:1952} and has been extensively studied in the literature, cf.~\cite{Choquet-Bruhat:2009}, \cite{Friedrich-Rendall:2000}, \cite{Hawking-Ellis:1973} and 
\cite{Ringstrom:2009} for example. 
The initial data $(\g, \k)$ on $S$ consists of a Riemannian metric $\g$ and symmetric bilinear form $\k$ satisfying 
the vacuum Einstein constraint equations, i.e. 
the Hamiltonian and momentum (or Gauss and Gauss-Codazzi) constraint equations: 
\be \label{constraint}
|\k|^2 - (\tr \k)^2 - R_{\g} = 0,\quad
{\rm div}[\k - (\tr \k)\g] = 0\ \ \mbox{ on }S.
\ee
Here $R_{\g}$ denotes the scalar curvature of the Riemannian metric $\g$ on $S$, while the norm $|.|$, trace $\tr$ and divergence {${\rm div}$} are all with respect to $\g$. It is well-known that given such smooth data $(\g,\k)$ on $S$, there exists a smooth globally hyperbolic vacuum spacetime $(V,g)$ such that the Riemannian metric $g_{S}$ and the second fundamental form $K_{g|S}$, both induced by $g$ on the initial surface $S\subset (V,g)$, satisfy $g_{S}=\g,~K_{g|S}=\k$ on $S$. Furthermore, if there is another vacuum spacetime $(V',g')$ inducing the same initial data $(g'_S,K_{g'|S})=(\g,\k)$ on $S$, then $(V,g)$ and $(V',g')$ must be isometric in a neighborhood of $S$.
Existence and Cauchy stability of solutions of \eqref{vac} with given initial data are proved by means of suitable choices of gauge (e.g.~suitable choices of local coordinates or space-time foliations of the spacetime). Such gauge choices are necessary to reduce the Einstein equations to be strictly hyperbolic in view of the invariance of solutions of \eqref{vac} under diffeomorphisms. Uniqueness of solutions (up to local isometry) is then proved by exploiting the geometric nature of the initial data. This allows one to patch together locally defined solutions (not necessarily given apriori in the same gauge) to obtain more global solutions. This leads to the existence of a unique (up to isometry) maximal globally hyperbolic 
solution to the Cauchy problem proved by Choquet-Bruhat and Geroch \cite{Choquet-Geroch:1969}.

In analogy to the Cauchy problem above, throughout the paper, we use the following definition of a well-posed IBVP for the vacuum Einstein equations:
\begin{Definition}\label{well-posed}
An initial boundary value problem for the vacuum Einstein equations on $M$ is a system for the spacetime metric $g$ given by 
\be\label{ibvp}
\Ric_g=0 \ \mbox{ on }M, \ \ g_{S}=\g,~K_{g|S}=\k \ \mbox{ on }S, \ \ \mathcal B(g)=b \ \mbox{ on }\cC,
\ee
where $\mathcal B(g)$ denotes a collection (to be chosen) of geometric quantities of $g$ evaluated along $\cC$. Such an IBVP is called (smoothly) \textit{well-posed} if for any smooth initial data $(\g,\k)$ (satisfying the vacuum constraints) and smooth boundary data $b$, satisfying smooth compatibility conditions at the corner $\Sigma$, there exists a smooth Lorentz metric $g$ solving the system above.  Moreover, if $g'$ is another solution, then $(M,g)$ and $(M,g')$ are isometric in some ST-neighborhood $\cT$. 
In addition, up to such local isometries, solutions $g$ depend continuously on the initial and boundary data.  
\end{Definition}

It needs to be pointed out that the uniqueness of solutions up to isometry above does not mean geometric uniqueness in the most ideal case, mainly due to the complexity in describing the boundary geometry $\mathcal B(g)$. We refer to the next two subsections for a detailed discussion.

\medskip 

The IBVP for \eqref{vac} has been well understood in the case of asymptotically locally anti-de-Sitter spacetimes, where the boundary is at infinity. Namely, Friedrich \cite{Friedrich:1995} has shown that geometric boundary data consisting of the conformal class of the metric at conformal infinity admits a well-posed IBVP, cf.~also \cite{Enciso-Kamran:2019} for similar results in higher dimensions. However, it is well-known that the situation where the boundary is at finite distance is much more difficult to understand than that of a boundary at conformal infinity. This is already apparent in the simpler, but formally analogous, situation of Riemannian Einstein metrics, cf.~\cite{Anderson:2008}; in particular, in the Riemannian setting boundary data consisting of the boundary metric $g_\cC$ alone cannot lead to a well-posed IBVP at a finite boundary in general. Henceforth, we will only be concerned with the finite IBVP, where the Cauchy surface $S$ is compact, with non-degenerate induced metric on $\Sigma$. 
\medskip

A number of distinct approaches to the finite IBVP for the equation \eqref{vac} have been developed. The IBVP was first seriously investigated and proved to have a well-posed gauge-dependent solution by Friedrich-Nagy in \cite{Friedrich-Nagy:1999}. This approach takes the basic unknowns as an orthonormal tetrad of the metric, associated connection coefficients and Weyl curvature components. 
They extract a symmetric hyperbolic system for these unknowns from the Einstein equations by assuming the so-called \textit{adapted gauge}. To fix an adapted gauge, one needs to choose apriori gauge source fields $f, F^A (A=1,2)$ in a neighborhood of $\cC$. 
One main step of their work is to prove that a solution of the reduced hyperbolic system must solve the original Einstein equations and satisfy the adapted gauge condition simultaneously. The boundary geometry $\mathcal B(g)$ of the spacetime
consists of the mean curvature $H_{\mathcal C}$ of $\cC$, and a tensor field built upon the coordinate components of the Weyl tensor expressed in the boundary frame determined by $F^A$.  
By solving the reduced system, they obtain a solution $g$ to \eqref{ibvp} in the adapted gauge. Furthermore, if $g'$ is another solution (not necessarily in the adapted gauge), then $g'$ is isometric to $g$ in an ST-neighborhood. Such a uniqueness result was first proved in the case $H_{\mathcal C} = f=const, F^A=0$, and later shown in the general case in \cite{Friedrich:2021}. We note that here the function $f$ is both treated as a gauge source function and as geometric data (mean curvature) describing the timelike boundary. 

In \cite{KRSW:2009}, cf.~also \cite{KRSW:2007}, Kreiss-Reula-Sarbach-Winicour treat the IBVP for \eqref{vac} in a harmonic gauge, analogous to the gauge most often used for the Cauchy problem.  Under the harmonic gauge, the vacuum equation \eqref{vac} is reduced to a hyperbolic system of wave equations for $g$. By imposing appropriate gauge conditions on the boundary $\cC$, it can be shown that solutions to the reduced system must solve both the original Einstein equation and satisfy the harmonic gauge condition. The boundary geometry $\mathcal B(g)$ of the spacetime is described by a collection of Sommerfeld-type boundary equations for the metric. Again, existence of a solution $g$ to \eqref{ibvp} in the harmonic gauge is proved by solving the reduced hyperbolic system. However, the uniqueness part is not clear --  it is not known whether other possible solutions must be isometric to such constructed $g$. The main reason for the failure of uniqueness is that the boundary equations in $\mathcal B(g)$ are expressed in a non-canonical frame of the spacetime on $\cC$.

In this paper, we develop a new approach to construct a well-posed IBVP \eqref{ibvp}, where the existence of solutions is proved in a similar way as in \cite{KRSW:2009} using (local) harmonic charts; the uniqueness result is achieved by describing the boundary geometry $\mathcal B(g)$ in a preferred frame -- referred to as the \textit{associated gauge} $\phi_g$ near $\cC$, which is determined by the spacetime $(M,g)$ and a gauge source field $\Theta_\cC$. {The way the uniqueness issue is resolved is, to a certain  extent, similar to that in \cite{Friedrich-Nagy:1999}; the key is the method of describing the boundary geometry. However, while the boundary frame in \cite{Friedrich-Nagy:1999} is deeply involved with the adapted gauge in the process of reducing the Einstein equations, in our work the associated gauge $\phi_g$ is geometric and independent of the local harmonic gauge imposed on the Einstein equations. }

The approach is to associate to each spacetime $(M,g)$ a natural wave map {(preferred frame)} $\f_g$ which, with {an appropriate} choice of initial-boundary data for $\f_g$, is uniquely determined by $g$. Moreover, this association will satisfy a natural and crucial equivariance property \eqref{equif} below. We then use natural boundary data for the metric $g$ in the {frame} $\f_g$ to describe the boundary conditions for $g$ in the IBVP. 

We first describe the construction of the associated gauge in detail. Given a compact, connected 3-manifold $S$ with nonempty boundary $\p S=\Si$ as above, let $M_0 = \bR_0^+ \times S$ be the standard product space equipped with a fixed global time function $t_0$ on $\bR_0^+ = [0,+\infty)$. Let $S_0=\{0\}\times S$ be the initial surface naturally identified with $S$, and $\cC_0 = \bR_0^+\times \Si$ be the associated portion of $\dm_0$, i.e.~$\dm_0 = S_0 \cup \cC_0$ with corner $\Si_0=\{0\}\times\Si$.  
Next, choose a fixed background smooth complete Riemannian metric $g_R$ on an open extension of $M_0$. As an example, when $\Si_0 = \p S_0$ is embedded in $\bR^3$ as the boundary of a handlebody $S_0$, one may choose $g_R$ to be the flat Euclidean metric on $\bR^4$. Let $r_0$ denote the distance function to $\cC_0$ in $M_0$ with respect to the Riemannian metric $g_R$. In the following we will consider wave maps from (subdomains of) an ST-spacetime $(M,g)$ to the target space $(M_0,g_R)$, equipped with the standard functions $t_0,r_0$.  The main results described below do not depend on these choices of $g_R, t_0, r_0$.

Suppose $g$ is a spacetime metric defined on a ST-neighborhood $\cT$ in $M$, with timelike boundary $\cC\cap\cT$. Consider a wave map 
$$\phi_g:(\cU,g) \to (M_0,g_R).$$
Here $\cU$ is an open neighborhood of the corner $\Sigma$ in $\cT$, which contains the entire corner $\Si$ and admits spacelike initial surface $S\cap\cU$ and timelike boundary $\cC\cap\cU$. {We will refer to such neighborhood $\cU$ as an \textit{ST-corner neighborhood}}. To describe the boundary geometry, it is sufficient to construct a preferred frame in an ST-corner neighborhood. The wave map equation for $\phi_g$ is given by 
\be \label{wave_map}
\Box_g \phi_g + \Gamma_{g_R}(\phi_g)g(\nabla \phi_g, \nabla \phi_g)= 0 \ \mbox{ in }\cU;
\ee
we refer to \S 2.1 for a detailed discussion of \eqref{wave_map}. Given a fixed Lorentz metric $g$, \eqref{wave_map} is a system of semi-linear hyperbolic wave equations for $\phi_g$. To obtain a unique wave map $\f_g$ associated to $g$, we impose natural initial conditions on $S\cap\cU$ and boundary conditions on $\cC\cap\cU$ for the hyperbolic {system} \eqref{wave_map}. 

We first discuss the initial conditions, which play the main role in establishing the equivariance property \eqref{equif} below. The initial data are chosen using a slice to the action of diffeomorphisms on the space of metrics on the surface $\Si$ given by the uniformization theorem. 
For $\Si = S^2$, fix three distinct points $p_i \in S^2$ (to break the action of the conformal group of $S^2$). For $\Si = T^2$, fix one point $p \in T^2$ (to break the action of translations). For $\Si$ of higher genus, such base points are not needed. Let $\mathrm{Diff}'(\Si)$ be the group of diffeomorphisms of $\Si$, isotopic to the identity, fixing $\{p_i\}$ in the case of $S^2$ and fixing $p$ in the case of $T^2$. 
By the uniformization theorem for surfaces, for any metric $\s$ on $\Si$, there is a unique diffeomorphism $\Phi_{\s} \in \mathrm{Diff}'(\Si)$ such that the pullback metric $(\Phi_{\s}^{-1})^*\s$ is pointwise conformal to a space-form metric $\s_0$ on $\Si$, i.e. $\s = \Phi_{\s}^*(\l^2 \s_0)$ for some function $\l$ on $\Si$. Here $\s_0$ is the round metric on the unit sphere for $\Si = S^2$, while $\s_0$ is a quotient of the Euclidean plane $\bR^2$ or hyperbolic plane $\bH^2$ by a lattice in case $\Si$ has positive genus. In suitable function space topologies, the transverse slice mapping 
$\s \to \Phi_{\s}$
is smooth. 

For an arbitrary Riemannian metric $\g$ defined on $S$, let $\g_\Si$ denote the induced metric on the boundary $\Si\subset (S,\g)$. By the analysis above, there is a unique diffeomorphism $\Phi_{\g_\Si}\in {\rm Diff}'(\Si)$ associated to $\g_\Si$. Since $\Si$ is naturally identified with the corner $\Si_0$ of $M_0$, we regard $\Phi_{\g_\Si}$ as a diffeomorphism from $\Si$ to $\Si_0$. Let $n_{\g}$ be the field of inward unit normal vectors to the equidistant foliations from $\Si$ in $(S,\g)$; $n_{\g}$ is well-defined in a collar neighborhood $W$ of $\Si$ in $S$. In the target spacetime $(M_0,g_R)$, let $\g_R$ denote the induced metric on $S_0$, then we can define the vector field $n_{\g_R}$ similarly on $S_0$. To each Riemannian metric $\g$ on $S$ we then assign a map $E_{\g}: W \to S_0$, defined by
\be \label{c1}
E_{\g} = \Phi_{\g_\Si} \ \ \mbox{ on }\Si, \quad\quad (E_{\g})_*(n_{\g}) = n_{\g_R}\ \ \mbox{ on }W.
\ee
Geometrically, $E_{\g}$ maps $\Si$ to $\Si_0$ and maps equidistance surfaces from $\Si$ to equidistance surfaces from $\Si_0$. Clearly $E_{\g}$ is a diffeomorphism onto its image in $S_0$ and note that $E_{\g}$ is uniquely determined and depends smoothly on the Riemannian metric $\g$ on $S$.
The choice of the defining equation \eqref{c1} is not unique. The main reason for the choice \eqref{c1} is that it satisfies the following equivariance property: Let ${\rm Diff}'(S)$ denote the diffeomorphism group
\be\label{diffS}
\mathrm{Diff}'(S):= \{\psi \in \mathrm{Diff}(S): \psi |_{\Si} \in \mathrm{Diff}'(\Si)\},
\ee
then for any Riemannian metric $\g$ and diffeomorphism $\psi \in \mathrm{Diff}'(S)$, the map $E_{\psi^*\g}$ generated by the pullback metric $\psi^*\g$ is naturally related to $E_{\g}$ by
\be \label{equi}
E_{\psi^*\g} 
= E_{\g} \circ \psi.
\ee
The results below hold for any choice of $E_{\g}$ satisfying \eqref{equi}.

Given an ST neighborhood $(\cT,g)$, let $g_S$ denote the induced Riemannian metric on the initial surface $S\subset (\cT,g)$. {So $g_S$ generates a map $E_{g_S}$ as in the above. We will choose an ST-corner neighborhood $\cU$ in $\cT$ so that} $E_{g_S}$ is well defined on $S\cap\cU$ and set it as the initial value for $\phi_g$, i.e. $\phi_g|_S=E_{g_S}$, and in addition we prescribe the 1-jet of $\phi_g$ by requiring
\be\label{c2}
(\phi_g)_*(N_g) = T_{g_R}\ \ \mbox{ on }S\cap\cU.
\ee
Recall that the boundary $\cC\cap\cT$ is timelike in $(\cT, g)$, so the induced metric $g_\cC$ is Lorentzian. Let $N_g$ be the future pointing unit timelike normal to the corner $\Si$ with respect to the ambient Lorentzian space $(\cC\cap\cT, g_\cC)$. Then we extend $N_g$ to be defined on $S\cap \cU$ by parallel translation along the flow of $n_{g_S}$. 
In the target spacetime, let $T_{g_R}$ denote the future pointing unit normal vector to the initial surface $S_0\subset (M_0,g_R)$. For convenience, here and in the following, we assume that $S_0$ is perpendicular to $\cC_0$ in $(M_0,g_R)$ along the corner $\Si_0$, so $T_{g_R}$ is tangent to the boundary surface $\cC_0$ along $\Si_0$. These provide the initial data for $\phi_g$.

Next we turn to the boundary conditions for \eqref{wave_map}. We first require that $\phi_g$ maps the boundary $\cC\cap\cU$ into $\cC_0$, i.e.~$r_0\circ\phi_g=0$ on $\cC\cap\cU$. Let $\Si_\tau$ be the level set $\Si_\tau = \{p\in\cC_0: t_0(p) = \tau\}$ in $\cC_0$. It naturally gives rise to the level set $\phi_g^{-1}(\Si_\tau)$ in $\cC\cap \cU$. Define $T_g^c$ on $\cC\cap\cU$ as the future pointing timelike unit normal vector to every level set $\phi_g^{-1}(\Si_\tau)$ in $\cC\cap\cU$ with respect to the induced metric $g_\cC$, and let $\nu_g$ be the outward unit normal vector to the boundary $(\cC\cap\cU)\subset (\cU,g)$. Then we impose the condition 
\begin{equation*}
[(\f_g)_*(T^c_g+\nu_g)]^T=\Theta_{\cC} \ \  \mbox{ on } \cC\cap\cU,
\end{equation*}
where $\Theta_{\cC}$ is a fixed, arbitrarily prescribed smooth vector field on $\cC_0$. Here the notation $[\cdot]^T$ denotes orthogonal projection of a vector field onto the boundary $\cC_0$ with respect to $g_R$ (cf. equation \eqref{prj}). This provides three Sommerfeld-type boundary conditions for $\phi_g$ (cf.~ \S 6.1).

By means of the discussion above, the associated gauge $\f_g$ to a spacetime metric $g$ is defined as follows.
\begin{Definition}\label{asso_wave}
Suppose $(\cT,g)$ is an ST- neighborhood 
and let $\cU$ be an ST-corner neighborhood in $ \cT$. Let $\Theta_\cC$ be a non-vanishing vector field on $\cC_0$. 
{The \textit{associated gauge} $\phi_g$ to $g$ in $\cU$, with respect to the gauge source $\Theta_\cC$, is defined as the wave map $\f_g:\cU\to M_0$
which is the unique solution to the system: }
\be \label{mainf3}
\begin{split}
\Box_g \f_g + \Gamma_{g_R}(\f_g)g(\nabla \f_g, \nabla \f_g)= 0
\quad\text{on }\cU
\end{split}
\end{equation}
\be \label{if3}
\begin{split}
\phi = E_{g_S},~\phi_*(N_g) = T_{g_R}
\quad\text{on }S\cap\cU
\end{split}
\end{equation}
\be \label{bf3}
\begin{split}
r_0 \circ \f_g = 0,\ \ 
[(\f_g)_*(T^c_g + \nu_g)]^T = \Theta_{\cC}
\quad\text{on }\cC \cap \cU.
\end{split}
\end{equation}

\end{Definition}

{Basic analysis from the theory of hyperbolic systems shows that for a fixed choice of $\Theta_{\cC}$, any spacetime $(\cT,g)$ satisfying certain compatibility condition along the corner $\Si$ (cf.~Proposition \ref{coner_cmpt}) admits a unique associated gauge $\f_g$ defined on some open neighborhood $\cU$.} Further the mapping $g \to \f_g$ is smooth. Equally importantly, the equivariance property \eqref{equi} implies that $\f_g$ also transforms equivariantly, i.e. the associated gauge $\phi_{\Psi^* g}$ for the pull back metric $\Psi^*g$ is related to $\phi_g$ by
\be \label{equif}
\f_{\Psi^*g} = \f_g \circ \Psi,
\ee
for any diffeomorphism $\Psi \in \mathrm{Diff}'(\cU) = \{\Psi \in \mathrm{Diff}(\cU)\big| \Psi:S\cap\cU\to S\cap\cU,\Psi:\cC\cap \cU\to\cC\cap \cU\mbox{ and }\Psi|_{\Si} \in \mathrm{Diff}'(\Si)\}$. 

\medskip

With the associated gauge $\phi_g$ well-established, we are ready to describe 
geometric boundary conditions for the metric $g$. We first describe the boundary data for Lorentz metrics $\mathring{g}$ on the target space $M_0$; the main case of interest will be $\mathring{g}$ of the form $\mathring{g} = (\f_g^{-1})^*g$ determined by an ST-corner neighborhood $(\cU,g)$. 
Recall that $\Si_{\tau}$ is the $\tau$-level set of the time function $t_0$ on $\cC_0$. For any $p\in\cC_0$, there is a surface $\Si_\tau$ containing $p$ and a {2-dimensional} metric $\mathring{g}_{\Si_\tau}$ induced by $\mathring{g}$ {on $\Si_\tau$}. Let $[\mathring{g}^\intercal]$ be the field on $\cC_0$ which, when evaluated at $p\in\cC_0$, equals to the pointwise conformal class of $\mathring{g}_{\Si_\tau}$ {at $p$}. In other words, we can understand $[\mathring{g}^\intercal]$ as the conformal invariant tensor field 
\bes
[\mathring{g}^\intercal]=\det (\mathring{g}_{\Si_\tau})^{-1/2} \mathring{g}_{\Si_\tau} \ \ \mbox{ on }\cC_0.
\ees
Next, let $S_\tau$ be the level set $\{x\in M_0: t_0(x)=\tau\}$ in $M_0$. Also let $K_{\mathring{g}|S_\tau}$ denote the second fundamental form of $S_\tau\subset (M_0,\mathring{g})$, and $K_{\mathring{g}|\cC_0}$ denote that of $\cC_0\subset(M_0,\mathring{g})$. Define the function $H_{\mathring{g}}$ on $\cC_0$ as
\be \label{Hintro}
H_{\mathring{g}} = a \tr_{\mathring{g}_{\Si_\tau}}K_{\mathring{g}|S_\tau} + b \tr_{\mathring{g}_{\cC_0}}K_{\mathring{g}|\cC_0} + c \tr_{\mathring{g}_{\Si_\tau}}K_{\mathring{g}|\cC_0}.
\ee
Here $\tr_{\mathring{g}_{\cC_0}}$ denotes the trace with respect to the induced metric $\mathring{g}_{\cC_0}$ on $\cC_0\subset(M_0,\mathring{g})$; and $\tr_{\mathring{g}_{\Si_\tau}}$ denotes the trace with respect to the induced metric $\mathring{g}_{\Si_\tau}$ on $\Si_\tau=S_\tau\cap\cC_0$; the coefficients $a, b, c$ are freely specifiable with a certain range, cf.~\eqref{h2n} and Proposition \ref{energy2} for exact details. 
Given an ST-corner neighborhood $(\cU,g)$ with the associated gauge $\phi_g$, we will consider the pull-back metric $(\f_g^{-1})^*g$ and the boundary geometry described as above when $\mathring{g} = (\f_g^{-1})^*g$, i.e.
\be\label{bg0}
\cB(g)=\big(~[ \big((\phi_g^{-1})^*g\big)^\intercal],~ H_{(\phi_g^{-1})^*g}~\big)~\mbox{ on }\cC_0\cap \phi_g(\cU).
\ee
Note that $\f_g$ may be a diffeomorphism only well-defined in an ST-corner neighborhood $\cU\subset\cT$. This is sufficient to provide a foliation of $\cT$ near $\cC$ in which to describe the boundary geometry of $g$. Thus $\f_g$ should be understood as a preferred frame (depending on the choice of $\Theta_{\cC}$) near the boundary assigned in a unique way to each metric $g$ in which one computes the boundary quantities \eqref{bg0}. 

With this understood, the first main result of the paper is that the IBVP \eqref{ibvp} for the vacuum Einstein equations with $\cB(g)$ given by \eqref{bg0} is well-posed. 

\begin{Theorem}\label{exist}(Well-posed IBVP)
For {any smooth vector field $\Theta_\cC$ on $\cC_0$}, smooth initial data $(\g,\k)$ satisfying the vacuum constraint equations \eqref{constraint} on $S$, and smooth boundary data $([\s], H)$ on $\cC_0$, all satisfying smooth compatibility conditions at the corner $\Si$, there exists an ST-neighborhood $\cT\subset M$ 
and a smooth spacetime metric $g$  on $\cT$ such that 
\be \label{maing3}
\begin{split}
\Ric_g = 0 
\quad\text{in }\cT
\end{split}
\end{equation}
\be \label{ig3}
\begin{split}
g_S = \g,~K_{g|_S} = \k
\quad\text{on }S
\end{split}
\end{equation}
\be \label{bg3}
\begin{split}
[\big((\phi_g^{-1})^*g\big)^\intercal]= [\s],\ \ H_{(\phi_g^{-1})^*g}=H
\quad\text{on }\cC_0 \cap \phi_g(\cU).
\end{split}
\end{equation}
In the above $\f_g$ is the unique associated gauge to $g$ on some ST-corner neighborhood $\cU$ 
in $\cT$ (with respect to the gauge source $\Theta_\cC$).
Further, $\cT$ equals to the domain of dependence of the boundary $\p\cT=S\cup (\cC\cap\cT)$ in $(\cT,g)$.

In addition, if $(\cT',g')$ is another solution of \eqref{maing3}-\eqref{bg3}, then it must be isometric to $(\cT,g)$ in some ST-neighborhood. {Finally, up to such isometries, solutions $g$ depend continuously on the data $\Theta_{\cC}$, $(\g, \k)$ and $([\s], H)$. }
\end{Theorem}
{Here smooth boundary data $([\s],H)$ on $\cC_0$ consists of a smooth function $H$ on $\cC_0$ and a field of conformal metrics $[\s]$ determined by $\s$, where $\s=\s(\tau)$ is a smooth family of 2-dim 
metrics with $\s(\tau)$ defining a metric on the level set $\Si_\tau$ of $\cC_0$ for each $\tau$. So the first equation in \eqref{bg3} means that the metric on $\Si_\tau$ induced from $(\phi_g^{-1})^*g$ is pointwisely conformal to $\s(\tau)$ for every $\tau$ that is contained in $\phi_g(\cU)\cap \cC_0$. }
We refer to \S 5 for the detailed compatibility conditions for the choice of $\Theta_\cC$ and initial-boundary data.

\subsection{Geometric uniqueness} 
As mentioned previously, the uniqueness in Definition \ref{well-posed} for the IBVP is different from the most general geometric uniqueness result in the Cauchy problem. Let $(V,g)$ denote a {Cauchy development} of some initial data $(\g,\k)$ on $S$, i.e.~a globally hyperbolic vacuum spacetime $V$ containing the initial data set $(S, \g,\k)$. Two Cauchy developments $(V_1, g_2)$, $(V_2, g_2)$ are called equivalent if they contain a common subdevelopment, i.e.~there exists a Cauchy development $(V, g)$ and isometric embeddings $\Psi_i: V \to V_i$ such that $\Psi_i^*g_i = g~(i=1,2)$. Let $\cV$ be the space of equivalence classes of Cauchy developments; in particular, isometric solutions belong to the same equivalence class. 
Similarly, we can define the space $\cI$ of equivalence classes of initial data $(\g,\k)$, where two initial data $(\g_1,\k_1),(\g_2,\k_2)$ are equivalent if there is a diffeomorphism $\psi$ of $S$ such that $\g_1=\psi^*\g_2,\k_1=\psi^*\k_2$. The geometric uniqueness result of the Cauchy problem implies that if two initial data sets are equivalent, then their Cauchy developments are also equivalent, i.e. there is a bijective correspondence  {(in fact a homeomorphism)}
\be \label{mod1}
D: \cV \to \cI.
\ee 
The map $D$ in \eqref{mod1} is canonical (there is only one natural choice). However, to construct an explicit inverse to $D$, i.e.~to construct a representative $(V, g)$ of the equivalence class $\{(V,g)\}$ over $(\g,\k) \in \{(\g, \k)\}$ in $\cI$, requires solving the equation \eqref{vac} which in turn requires a choice of gauge. Thus a map $D^{-1}$ at the level of representatives of an equivalence class is gauge-dependent. The geometric uniqueness of solutions shows this dependence disappears when passing to equivalence (or local isometry) classes. 

The bijection \eqref{mod1} gives an effective parametrization of the space of solutions of \eqref{vac} by their initial data. Of course the most natural abstract representative for an equivalence class is the unique (up to isometry) maximal globally hyperbolic solution to the Cauchy problem. 

Ideally, one would like to obtain similar results and a similar understanding for the IBVP for the equation \eqref{vac}, and so in particular obtain a bijective correspondence 
\be \label{mod2}
D: \cM \to \cI \times_{c} \cB,
\ee 
where $\cM$ denotes the space of equivalence classes of ST-spacetimes. 
As above, $\cI$ is the moduli space of initial data on $S$ and $\cB$ is a space of boundary data on $\cC$ modulo some (to be determined) action of $\mathrm{Diff}(\cC)$. The subscript $c$ denotes compatibility conditions between the initial and boundary data at the corner $\Si$. This suggests that one starts with a well-posed IBVP \eqref{ibvp} and try to define equivalence  relations for the boundary geometry $\mathcal B(g)$.  Again ideally, the boundary term $\mathcal B(g)$ would also be determined from the ${\rm Diff}_0(M)$-invariant Cauchy data at $\cC$, i.e.~the induced metric $g_\cC$ and second fundamental form $K_{g|\cC}$ of $\cC$ in $(M, g)$. It remains a basic open problem (not answered here) of whether there is a choice of gauge for which the IBVP is well-posed for some choice of such geometric boundary data and for which \eqref{mod2} holds. 

The next main result is that the well-posed IBVP for the vacuum Einstein equations \eqref{maing3}-\eqref{bg3} in Theorem \ref{exist} possesses certain geometric uniqueness property,  relevant to establishing a parametrization of $\cM$ as in \eqref{mod2}.

In the following and throughout the paper, ${\rm Diff}(M)$ denotes the group of diffeomorphisms on $M$ which induce diffeomorphisms $S \to S$ and $\cC \to \cC$; the restricted gauge group ${\rm Diff}_0(M)$ consists of diffeomorphisms which restrict to the identity on the boundary $\p M=S \cup \cC$. We also recall ${\rm Diff}'(S)$ defined in \eqref{diffS}.
We will use $({\mathbb I},\bB)$ to denote the free initial-boundary data in 
\eqref{ig3}-\eqref{bg3}:
\be\label{ibg3}
{\mathbb I}=(\g,\k),~\bB=([\s],H).
\ee
\begin{Definition}\label{equiv_IB}
Two sets $(\mathbb{I}_1,\mathbb{B}_1)=(\g_1,\k_1,[\s_1],H_1)$ and $(\mathbb{I}_2,\mathbb{B}_2)=(\g_2,\k_2,[\s_2],H_2)$ are called \textit{equivalent} if there exists a diffeomorphism $\psi\in{\rm Diff}'(S)$ such that 
\begin{equation}\label{equivI}
(\g_1,\k_1)=(\psi^*\g_2,\psi^*\k_2)\ \ \mbox{on }S,
\end{equation}
and a subdomain $\cC_{0\tau}=\{x\in\cC_0:t_0(x)<\tau\}~(\tau>0)$ in $\cC_0$ such that 
\begin{equation}\label{equivB}
([\s_1],H_1)=([\s_2],H_2)\ \ \mbox{on }\cC_{0\tau}.
\end{equation}
\end{Definition}
Based on the construction, it is obvious that if $g_1,g_2$ are two isometric vacuum spacetime metrics on $M$, i.e. $\Psi^*g_2=g_1$ for some $\Psi\in {\rm Diff}'(M)=\{\Psi\in {\rm Diff}(M):\Psi|_S\in{\rm Diff}'(S)\}$, then they must satisfy the system \eqref{maing3}-\eqref{bg3} with equivalent initial-boundary data, where their associated gauge are with respect to a fixed source $\Theta_\cC$. 

Note that the full diffeomorphism group ${\rm Diff}(M)$ acts trivially on the boundary data space $\bB$; thus, diffeomorphisms act simultaneously (and inversely) on the metric $g$ and the associated gauge $\f_g$. Among the 6 degrees of freedom one would expect to prescribe for boundary geometry of $g$, 3 correspond to the choice of $\Theta_{\cC}$ (a gauge choice) and 3 correspond to data in $\bB$. The latter accounts for the 2 degrees of freedom of the gravitational field while 1 (for example the second condition in \eqref{bg3}) accounts for the evolution of the boundary $\cC$ off the corner $\Si$.

Conversely, we prove the equivalence class of ST-spacetimes solving the IBVP \eqref{maing3}-\eqref{bg3} is determined by the initial-boundary data:

\begin{Theorem}\label{geom_unique}(Geometric Uniqueness)
Let $(\cT_1, g_1)$ and $(\cT_2, g_2)$ be two solutions of the system \eqref{maing3}-\eqref{bg3} with respect to the initial-boundary data $(\mathbb{I}_1,\mathbb{B}_1)$ and $(\mathbb{I}_2,\mathbb{B}_2)$, {where $\phi_{g_1},\phi_{g_2}$ are the associated gauge to $g_1, g_2$ (with respect to a fixed gauge source $\Theta_\cC$).} If $(\mathbb{I}_1,\mathbb{B}_1)$ and $(\mathbb{I}_2,\mathbb{B}_2)$ are equivalent and related via $\psi$ {as in \eqref{equivI}-\eqref{equivB},} then there are ST-neighborhoods $\cT'_i \subset \cT_i$ ($i=1,2$), such that 
$$\Psi^*g_2=g_1$$
for some diffeomorphism $\Psi: \cT'_1 \to \cT'_2$. In addition, $\Psi|_{S}=\psi$ and $\Psi|_{\cU} = \f_{g_2}^{-1}\circ\f_{g_1}|_{\cU}$ where $\cU$ is an ST-corner neighborhood in $\cT'_1$ on which $\phi_{g_1}$ is well defined.
\end{Theorem}

In the above the diffeomorphism $\psi$ must belong to ${\rm Diff}'(S)$. It remains open if this can be generalized to $\psi \in {\rm Diff}(S)$, since for general $\psi\in{\rm Diff}(S)$, one may lose track of the boundary data when transforming from $g$ to $\psi^*g$. We refer to Theorem \ref{exist_unique} for more detail. 

\medskip

Next we turn to the existence of a correspondence as in \eqref{mod2}. In analogy to prior discussion, we call a solution of the system \eqref{maing3}-\eqref{bg3} a vacuum development of the initial-boundary data $(\bI,\bB)$. Let $\cM$ be the moduli space of vacuum developments, where two solutions $(\cT_1, g_1)$, $(\cT_2, g_2)$ are equivalent if they are isometric for a short time starting from the initial surface $S$, in other words, there exists a vacuum development $(\cT, g)$ and embeddings $\Psi_i: \cT \to \cT_i$, with $\Psi_i:S\to S,~\Psi_i:\cC\cap\cT\to\cC\cap\cT_i$ and $\Psi_i|_S\in{\rm Diff}'(S)$, such that $\Psi_i^*g_i = g$ $(i=1,2)$. 

Regarding the right side of \eqref{mod2}, there are compatibility or corner conditions between the initial data $\mathbb{I}$ and boundary data $\mathbb{B}$; these are described in more detail in \S 2. Let then $\cI\times_c \cB$ denote the space of compatible initial data $\mathbb{I}$ and boundary data $\mathbb{B}$, modulo the equivalence relation given in Definition \ref{equiv_IB}. Next, regarding the boundary gauge $\Theta_{\cC}$, let $\chi(\cC_0)$ denote the space of smooth vector fields on $\cC_0$. While $\Theta_{\cC} \in \chi(\cC_0)$ may be chosen arbitrarily away from the corner $\Si_0$, there are also compatibility conditions that $\Theta_{\cC}$ must satisfy at the corner $\Si_0$; for example, at lowest order (cf. equations \eqref{cibc3} and \eqref{corner1})
$$\Theta_{\cC} = \ell T_{g_R}$$
at $\Si_0$, for some function $\ell$, $0<\ell<\tfrac{1}{\sqrt{2}}$; the value of $\ell$ is determined by the intersection angle of $S_0$ and $\cC_0$ at $\Si_0$ with respect to the pull-back metric $(\f_{g}^{-1})^*g$. Let $\cJ(\Si_0)$ denote the space of $C^{\infty}$ jets of vector fields on $\cC_0$ at $\Si_0$, satisfying the corner conditions, cf.~\S 2 and Proposition \ref{coner_cmpt} for details. A jet $J\in \cJ(\Si_0)$
 is given by
 \bes
J= \big(\Theta_\cC,~\cL_{T_0}\Theta_\cC,~\cL^2_{T_0}\Theta_\cC,...\big) ~\mbox{ on }~\Si_0
 \ees
for some vector field $\Theta_\cC\in\chi(\cC_0)$. Here $T_0$ is the field of
the unit normal vectors to the level sets $\Si_\tau$ in the ambient manifold $(\cC_0,(g_R)|_{\cC_0})$; and $\cL^n_{T_0}$ is the $n$ times Lie derivative with respect to $T_0$. One has a natural fibration $\pi: \chi(\cC_0) \to \cJ(\Si_0)$; each fiber is diffeomorphic to the space of smooth vector fields on $\cC_0$ vanishing to infinite order at $\Si_0$. A smooth section $\Lambda:\cJ(\Si_0)\to \chi(\cC_0)$ assigns to each jet $J\in\cJ(\Si_0)$ a smooth vector field $\Lambda(J)\in\chi(\cC_0)$.

\begin{Theorem}\label{phase_space} With a fixed smooth section $\Lambda$ of the fibration $\pi: \chi(\cC_0) \to \cJ(\Si_0)$, there is a bijective correspondence
\be \label{mod}
D_{\Lambda}: \cM \to (\cI \times_c \cB)\times  \cJ(\Si_0).
\ee
In addition, given a representative element $(\bI,\bB,J)\in (\cI \times_c \cB)\times  \cJ(\Si_0)$ there is a unique (up to isometry) maximal vacuum development corresponding to $\big(\bI,\bB,\Lambda(J)\big)$.
\end{Theorem} 
We refer to Proposition \ref{coner_cmpt} for the precise construction of the map $D_\Lambda$.
This theorem implies that near the initial surface, the isometry class of a vacuum ST-spacetime is uniquely determined by its geometry at the corner, the initial data, and the boundary data expressed in its associated gauge. In the above, the smooth section $\Lambda$ gives rise to the gauge source $\Theta_\cC$ that is used in the construction of the associated gauge. 
Moreover, different choices of the background data $(t_0, g_R)$ are incorporated in different choices of $\Theta_{\cC}$. Similarly, different choices of the initial data $(E_{g_S}, T_{g_R})$ for $\f_g$ in \eqref{if3} again merely give rise to different correspondences in \eqref{mod}. 
\medskip

\subsection{Further remarks}
The results above on the structure of solutions $(\cT, g)$ to the IBVP formally resemble the well-known results on the structure of solutions to the Cauchy problem. There are two significant differences here however.

First, in the parametrization \eqref{mod}, we need to fix a choice of $\Lambda$ which essentially determines the gauge source field $\Theta_\cC$, while this is not needed in \eqref{mod1}. Another aspect of this issue is following. 
The initial data $(\g,\k)$ is geometric in that it does not depend on any further data or information on the solution $(V, g)$. In other words, to check whether two Cauchy developments $(V_1,g_1)$, $(V_2,g_2)$ are equivalent, one only needs to read off the intrinsic and extrinsic geometry of the initial surface $S$. However, to compare two ST-spacetimes $(\cT_1,g_1)$ and $(\cT_2,g_2)$, we need to read off the boundary data \eqref{bg3} of each spacetime, which is not purely determined by the geometry (curvatures) of $(\cT_i,g_i)$ at the boundary; one actually needs the full information of $g_i$ to solve for the associated gauge $\phi_{g_i}$, and only after that will the boundary data be available for comparison. This issue is mentioned in \cite{Friedrich:2009} -- to establish an ideal geometric uniqueness result as in the Cauchy problem, we need to include the gauge source field $\Theta_\cC$ in the boundary data, and try to construct equivalence relations for the triple $(\Theta_\cC,[\sigma], H)$. In the Friedrich-Nagy work, the issue is to construct an equivalence relation for $(F^A,\cB(g))$. Both issues remain widely open. (Note however that it is unknown whether a parametrization as in \eqref{mod} can be constructed based on the well-posed IBVP in \cite{Friedrich-Nagy:1999}; one main reason is that the mean curvature of the timelike boundary serves both as a geometric data and as a gauge source function).

Secondly, the initial data $E_{g_S}$ (or the time derivative initial data) for the associated gauge $\f_g$ does not propagate forward in time. Moreover, while the boundary data $([\s], H)$ of $\bB$ are freely specifiable, the boundary {\it conditions} \eqref{bg3} are expressed in terms of the evolution of $\f_g$ along the boundary $\cC$. This evolution is globally dependent on the solution $g$. 
In more detail, let $(\cT, g)$ be a vacuum solution with associated gauge $\f_g$ and initial-boundary data $(\bI, \bB)$ as in \eqref{maing3}-\eqref{bg3}. Let $\w S \subset \cT$ be a Cauchy surface in $(\cT, g)$ to the future of $S$ with induced metric $g_{\w S}$ and second fundamental form $K_{g|\w S}$. The associated gauge $\f_g$ no longer satisfies the conditions \eqref{if3} on $(\w S,g_{\w S})$. Solving the system 
\eqref{maing3}-\eqref{bg3} with the new initial data set $(\w S, g_{\w S},K_{g|\w S})$ and the same boundary data $\bB$ gives a new solution $\w g$ with associated gauge $\f_{\w g}$ starting at $\w S$. Since the maps $\f_g$ and $\f_{\w g}$ are distinct, the solutions $g$ and $\w g$ can not be expected to be isometric near $\w S \cap \cC$. This behavior is different from that of solutions to the Cauchy problem.\footnote{We are grateful to Jacques Smulevici for pointing out this fact.} It would be interesting to understand if there are other methods of determining a preferred gauge $\f_g$ which are more local and in particular independent of the initial data. 

On the other hand, the use of associated gauge and the large symmetry group $\mathrm{Diff}(\cC)$ of the boundary data space $\mathbb{B}$ 
is very useful in developing a quasi-local Hamiltonian for the IBVP, which has not been accomplished by other means. This is  discussed elsewhere, cf.\cite{An-Anderson:2021}. 

\medskip

{
Finally we mention the recent work of Fournadavlos-Smulevici which proves existence and geometric uniqueness of vacuum solutions with totally geodesic boundary condition \cite{Fournodavlos-Smulevici:2021}  or totally umbilic boundary condtion \cite{Fournodavlos-Smulevici:2023}. In these works, the boundary conditions are fixed to be of a very special type, in contrast to general boundary data considered here. \footnote{As pointed out in \cite{An-Anderson:2021}, boundary data consisting of the second fundamental form $A$ of the boundary $\cC$ does not lead to a well-posed IBVP.}
A comprehensive survey and numerous further references regarding the IBVP for the Einstein equations are given in \cite{Sarbach-Tiglio:2012}. 

  We also note that the results above hold, with minor changes in the proofs, to vacuum spacetimes with a non-zero cosmological constant, where \eqref{vac} is replaced by the equation $\Ric_g = \Lambda g$, $\Lambda \in \bR$.  We also expect the results to hold for the Einstein equations coupled to matter fields which admit a well-posed IBVP with respect to a fixed background metric. However, the actual verification of this is left for future work. Similarly, it would be interesting to identify exact gravitational boundary conditions for which the results above hold in higher dimensions. 

}
\medskip

The contents of the paper are briefly as follows.  In \S 2, we construct expanded IBVP's of vacuum Einstein equations for the metric $g$ and wave equations for the gauge field $F$ with two different sets of initial-boundary data $({\bf I},{\bf B})$ in \eqref{I1}-\eqref{B1} and $({\bf I},{\bf B_{\cC}})$ in \eqref{I2}-\eqref{B2}. We also discuss in detail the gauge reduced systems and derive the corresponding frozen coefficient linear systems for both of them. 
In \S 3, we derive the requisite energy estimates for these linear systems, based on Sommerfeld and Dirichlet energy estimates. These are then used in \S 4 to prove local well-posedness of the gauge reduced systems of the expanded IBVP's. We also prove local versions of geometric uniqueness results for the expanded IBVP's. Building on these prior results, the main section of the paper, \S 5, then discusses the gluing of local solutions to obtain the global solutions of the IBVP's,  both for the expanded systems of $(g, F)$ as well as in the context of vacuum solutions $g$ (Theorems \ref{exist} and \ref{geom_unique}). {In addition, we also prove Theorem \ref{phase_space} and in particular discuss the existence and uniqueness of a maximal solution to the IBVP \eqref{maing3}-\eqref{bg3} analogous to  the corresponding result for the Cauchy problem.} Finally in the Appendix, \S 6, we collect and derive a number of results used in the main text. 
\medskip

This work benefited greatly from participation at the BIRS-CMO conference on Timelike Boundaries in General Relativistic Evolution Problems held at Oaxaca, Mexico in July 2019. We would like to thank P. Bizon, H. Friedrich, O. Reula and O. Sarbach for organizing such a fine meeting and thank in particular Helmut Friedrich and Jacques Smulevici for very useful discussions and comments on an earlier draft of this work.   


\section{The expanded system of spacetime metrics and {wave maps}}

It is well-known that the vacuum Einstein equations \eqref{vac} form a degenerate hyperbolic system on the spacetime metric.
In the following, we will consider formulations of the IBVP where the vacuum Einstein equations are reduced to be strictly hyperbolic using a harmonic gauge. Notice that, in an ST-spacetime $M=I\times S$, the choice of a harmonic gauge is not unique; it depends on a suitable choice of initial and boundary conditions. More precisely, working locally in $M$ for the moment, local harmonic or wave coordinates are functions $x^{\a}$, $\a = 0,1,2,3$ with $\Box_g x^{\a} = 0$. Such coordinates are uniquely determined by their initial data on $S$ and boundary data on $\cC$. Since the initial data of the spacetime metric $g$ consists of the Cauchy data $(g_S,K_{g|S})$ on $S$, and since the Cauchy data transforms naturally between different gauges, the initial data of the gauge field ($x^{\a}$) will not impact the existence and geometric uniqueness of the solution to the Cauchy problem for the vacuum Einstein equations. 

On the boundary $\cC$, one may impose (for instance) Dirichlet or Sommerfeld-type boundary conditions for $x^{\a}$ on $\cC$. The boundary $\cC$ may be defined locally as the locus $\{x^1 = 0\}$, so that $x^1$ is a local defining function for $\cC$; this gives a fixed Dirichlet boundary value to $x^1$. There remain 3 degrees of freedom in the choice of boundary data for $x^\a$, $\a = 0, 2,3$ on $\cC$. This freedom formally corresponds to the freedom in the choice of timelike vector field $T$ as in the work \cite{KRSW:2009}, which is involved in boundary conditions of the metric $g$ in the IBVP of vacuum Einstein equations, and hence impacts the geometric uniqueness of the solutions in \cite{KRSW:2009}.
There appears to be no general method to remove this freedom by some more canonical choice, (although see the remarks in \cite{Friedrich:2009}). Thus, we first take the approach to expand the system \eqref{vac} to a system of equations with unknowns consisting of both a metric $g$ and a gauge field{, i.e. a wave map} $F$, and establish the well-posedness of the IBVP for the expanded system.

\subsection{The expanded system for $(g,F)$}
As discussed in \S 1, let $(M_0,g_R)$ denote the target space, equipped with the time function $t_0$ and distance function $r_0$. { Simultaneously with} solving for a vacuum spacetime metric $g$ in $M$, consider wave maps 
\be \label{wave1}
{F:(M, g) \to (M_0, g_{R}) }
\ee
coupled to $g$, i.e.~critical points of the Dirichlet energy $\int_{M}|DF|^2 dV_{g}$, cf.~\cite{Geba-Grillakis:2017} for instance. Such maps satisfy the wave map equation 
\be \label{wave2}
\Box_g F + \Gamma_{g_R}(F)g(\nabla F, \nabla F)= 0.
\ee
In a local chart $\{y^\rho\}_{\rho=0,1,2,3}$ of $M_0$, the equation above is equivalent to 
\bes
\Box_g F^\rho+(\Gamma_{\rho_1\rho_2}^\rho\circ F)g(\nabla F^{\rho_1},\nabla F^{\rho_2})=0
\ees
where $F^\rho=y^\rho\circ F$, $\Gamma_{\rho_1\rho_2}^\rho$ are the Christoffel symbols of $g_R$ in the chart $(y^\rho)$, and $\nabla F^{\rho_1}$ is the gradient of $F^{\rho_1}$ with respect to the metric $g$. In addition, we recall that in a local chart $\{x^\a\}_{\a=0,1,2,3}$ of $M$, the wave operator $\Box_g F^\rho=\tfrac{1}{\sqrt{-{\rm det} g}}\p_{\a_1}[\sqrt{-{\rm det} g}\cdot g^{\a_1\a_2}\p_{\a_2}F^\rho]$.

We will impose Dirichlet boundary conditions for $F$ on the boundary $\cC$. Moreover, initial-boundary data for $F$ will be prescribed so that $F$ is a diffeomorphism onto its image in a neighborhood of $\cC$ for at least a short time. Thus, locally and near the boundary, $F$ gives a gauge choice of generalized harmonic (or wave) coordinate system depending on the choice of its Dirichlet boundary values, as described above with the local chart $\{x^{\a}\}$. 

Now we establish an expanded system for the pair $(g,F)$. In the bulk $M$, consider evolution equations 
\be \label{main}
\begin{cases}
\Ric_{g} = 0,\\
\Box_g F + \Gamma_{g_R}(F)g(\nabla F, \nabla F)= 0
\end{cases}
\mbox{ in }M.
\end{equation}
Note that while $F$ is coupled to $g$, $g$ is not coupled to $F$, i.e.~we are not considering the coupled system of Einstein-wave map equations.\footnote{We expect that the existence and geometric uniqueness results below, Theorems \ref{exist1}-\ref{exist2},\ref{unique1}-\ref{unique2}, also hold for the fully coupled Einstein-wave map system, where the vacuum equation $\Ric_g = 0$ is replaced by $\Ric_g = T_F$, where $T_F$ is the stress-energy tensor of the wave map $F$.} 
We impose the initial conditions 
\be \label{i1}
\begin{split}
\begin{cases}
g_{S}=\g,~ K_{g|S}=\k \\
F=E_0,~F_*(\cN_g) = E_1
\end{cases}
\quad\text{on }S,
\end{split}
\ee
and the boundary conditions 
\be \label{b1}
\begin{split}
\begin{cases}
F=G\\
[g_F^\intercal]=[\s]\\
F_*(T_g+\nu_g)=\Theta
\end{cases}
\quad\text{on }\mathcal C.
\end{split}
\ee
In the initial conditions \eqref{i1}, the pair $(\g,\k)$ is an initial data set satisfying the vacuum constraint equations \eqref{constraint}.  The pair $(E_0,E_1)$, assigning initial conditions for $F$, consists of a map $E_0: S \to S_0$ which induces a diffeomorphism $E_0|_{\Si}: \Si \to \Si_0$, and a vector field $E_1: S \to (TM_0)|_{S_0}$ transverse to $S_0$. Thus the initial conditions in \eqref{i1} are that $g$ induces Riemannian metric $g_{S}=\g$ and second fundamental form $K_{g|S}=\k$ on the initial surface $S$; moreover, $F$ induces the map $E_0$ on $S$ and the push-forward vector field $F_*(\cN_g)$ equals to the prescribed vector field $E_1$. Here $\cN_g$ denotes {some future pointing timelike vector transverse to the initial surface $S\subset (M,g)$}. In the following we call the free data in \eqref{i1} an initial data set on $S$ and denote it as 
\begin{equation}\label{I1}
\begin{split}
{\bf I}=(\g,\k,E_0, E_1).
\end{split}
\end{equation}

\begin{remark}\label{N_g_con}
For later purposes (cf.~\eqref{i3}), {we choose $\cN_g$ in the initial condition $F_*(\cN_g)=E_1$ in \eqref{i1} (and \eqref{i2} below) to be $N_g$ in a neighborhood of the corner $\Si\subset S$.} Here $N_g$, at the corner, equals the future pointing timelike unit normal to the hypersurface $\Si$ in the ambient manifold $(\cC,g_\cC)$ and is defined over a collar neighborhood of $\Si$ in $S$ by parallel extension along the flow of $n_{g_S}$. Recall that $n_{g_S}$ denotes the field of inward unit normal vectors to the equidistant foliations from $\Si$ in $(S,g_S)$. 
\end{remark}

In the boundary conditions \eqref{b1}, we first prescribe the Dirichlet boundary value of $F$ to be $G:\cC\to\cC_0$, where $G$  is a diffeomorphism mapping the corner $\Si$ to $\Si_0$ and $G$ is the restriction of a diffeomorphism in a thickening of $\cC$. It will always be assumed that $G$ is both an orientation and time-orientation preserving diffeomorphism. It is easy to observe that $r_0\circ F = 0$ on $\cC$ due to this boundary condition. 

The other two equations of \eqref{b1} can be understood as prescribing the boundary geometry of $g$ in terms of the wave map $F$. Recall that $S_\tau$ and $\Si_\tau$ denote the {$\tau$-level sets of $t_0$ in the target space $M_0$ and $\cC_0$.} The pull-backs $F^{-1}(S_\tau)$ and $F^{-1}(\Si_\tau)$ define foliations of $M$ and $\cC$ in a neighborhood of $\Si$ and the geometry of $g$ is examined in these foliations. This is equivalent to decomposing the pull-back metric $g_F=(F^{-1})^*g$ on $M_0$ with respect to the foliations $\Si_\tau$ and $S_{\tau}$ near $\Si_0$. 
Let $\s=\s(\tau)$ be a 1-parameter family of Riemannian metrics, {with $\s(\tau)$ defining a metric on the level set $\Si_{\tau}$ for each $\tau$. Let $g_F^\intercal$ denote the field on $\cC_0$ of 2-dim metrics on the level sets $\Si_\tau$ induced by the pull-back metric $g_F$.} Then the second equation in \eqref{b1} states that $g_F^\intercal$ is pointwise conformal to the given metric $\s(\tau)$ on each level set $\Si_\tau$ of $\cC_0$.  
The last boundary equation means the push forward of the vector field $T_g+\nu_g$ by $F$ equals the prescribed vector field $\Theta$, where  $\Theta$ is a nowhere vanishing vector field on $M_0$ restricted to $\cC_0$. Here and in the following, $T_g$ denotes the field of timelike unit normal vectors to the level sets $F^{-1}(S_\tau)$ in $M$ with respect to the metric $g$ and $\nu_g$ denotes the outward spacelike unit normal to $\cC\subset (M,g)$. 
Observe the last equation in \eqref{b1} can be equivalently written as 
$$T_{g_F}+\nu_{g_F}=\Theta \ \ {\rm on}\ \ \cC_0,$$
where $T_{g_F}$, $\nu_{g_F}$ denote the timelike unit normal to $S_{\tau}$ and spacelike unit normal to $\cC_0$ in $(M_0,g_F)$ respectively. 
The free data in \eqref{b1} is considered as a collection of boundary data on $\cC_0$ and denoted by 
\be\label{B1}
\begin{split}
{\bf B}=(G,[\s],\Theta).
\end{split}
\ee

As noted following \eqref{wave2}, locally $F$ may be viewed as a (generalized) harmonic or wave coordinate chart near $\Si$. Thus the IBVP \eqref{main}-\eqref{b1} is to construct vacuum Einstein metrics $g$ satisfying boundary conditions which locally are expressed in the chart $F$. Note however that the vacuum equations for $g$ are not solved in this $F$-chart -- $g$ and $F$ are solved simultaneously in the appropriate local chart, 
cf. the gauged system \eqref{main3} below. 

Observe that the 6 boundary data $([\s], \Theta)$ prescribe certain Dirichlet boundary conditions on the metric $g$ on $\cC$, given the local chart $F$. We note here that there are a number of possible modifications to the boundary conditions \eqref{b1} which can be well-posed locally.
As a trivial example, one may change the last boundary condition in \eqref{b1} locally to 
\be \label{altb}
g_{0\a} + g_{1\a} = \t_{\tilde\a},
\ee 
where $\theta_{\tilde \a}$ are the components of $\Theta$ expressed in some chart $\chi_0$ of $M_0$ and $g_{\a\b}$ are the components of $g$ expressed in the chart $\chi_0\circ F$. However, for many or most of these possible modifications, it may not be possible to extend the local existence to existence of solutions in a full domain $M$ containing $S$ by patching together local solutions; it is the invariance of the geometric quantity $\big([g_F^\intercal], F_*(T_g+\nu_g)\big)$ of a pair $(g,F)$ in \eqref{b1} which makes this possible (cf. also Remark~\ref{alt_theta}). 

\begin{remark}
The equations in \eqref{main} comprise 14 coupled equations for the 14 unknowns $(g, F)$, ($(g_{\a\b}, F^{\a})$ in components). There are 
only 10 boundary conditions in \eqref{b1}; 4 Dirichlet conditions on $F$ and $6$ on $g$ coupled to $F$; these latter will primarily be viewed as 
conditions for $g$ (given $F$). This discrepancy corresponds to the fact that the equation $Ric_{g} = 0$ is degenerate hyperbolic. As is common 
and carried out in \S 2.2 below, one adds a gauge term $\d^{*}V_g$ to make the equations \eqref{main} hyperbolic, giving the gauge reduced Einstein equations \eqref{main3} below. This requires adding the 4 extra boundary conditions $V_g= 0$ at $\cC$ in \eqref{b3} below to ensure that solutions of the gauge reduced Einstein equations are actually solutions of the vacuum Einstein equations. Similarly, there are only 20 initial conditions in \eqref{i1} for the 14 unknowns $(g, F)$. For the gauge reduced Einstein equations, the 8 extra components $g_{0\a}$ and $\p_{t}g_{0\a}$ are added to the initial data, subject to the constraint $V_g= 0$ on $S$ which consists of 4 equations; the action of the diffeomorphism group $\mathrm{Diff}_0(M)$ then accounts for the remaining 4 degrees of freedom; this is described in detail in \S 2.2. 
\end{remark}

It is of basic interest to understand if the 4 degrees of freedom in the choice of $\Theta$ can be reduced to 3 (or less). Using harmonic gauges, it appears to be unlikely that they can be made ``fully geometric" (in that the boundary data is expressed completely in terms of the induced metric and second fundamental form of the boundary $\cC$), but we present below a class of boundary conditions based on the mean curvature of various slices at $\cC$. 

Prescribe then a collection of boundary data on $\cC_0$
\begin{equation}\label{B2}
\begin{split}
{\bf B_{\cC}}=(G,[\s],H,\Theta_{\cC}),
\end{split}
\end{equation}
where $G$, $[\s]$ have the same meaning as in the ${\bf B}$ boundary data, $H$ is a scalar field on $\cC_0$ and $\Theta_{\cC}$ 
is a vector field tangent to $\cC_0$. 
Pairing with ${\bf B_{\cC}}$, we define a collection of initial data on $S$
\begin{equation}\label{I2}
\begin{split}
{\bf I}=(\g,\k,E_0, E_1),
\end{split}
\end{equation}
of exactly the same type as in \eqref{I1}. Then we consider the second system of IBVP for $(g,F)$: 
\be \label{main2}
\begin{split}
\begin{cases}
\Ric_g=0\\
\Box_g F + \Gamma_{g_R}(F)g(\nabla F, \nabla F)= 0
\end{cases}
\quad\text{in }M
\end{split}
\ee 
\be \label{i2}
\begin{split}
\begin{cases}
g_S=\g,~ K_{g|S}=\k\\
F=E_0,~F_*(\cN_g) = E_1
\end{cases}
\quad\text{on }S
\end{split}
\ee
and
\be \label{b2}
\begin{split}
\begin{cases}
F=G\\
[g_F^\intercal]=[\s]\\
H_{g_F}=H\\
[F_*(T_g^c+\nu_g)]^T=\Theta_{\cC}
\end{cases}
\quad\text{on }\cC.
\end{split}
\ee
The notation above is the same as in \eqref{main}-\eqref{b1} while $H_{g_F}$ is a linear combination of different types of mean curvature measured on the boundary as in \eqref{Hintro} (cf. also ~\eqref{h2n} and Proposition \ref{energy2}). In the last boundary equation, $T_g^c$ denotes the field of future pointing timelike unit normal vectors to the hypersurfaces $F^{-1}(\Si_\tau)$ in the ambient manifold $(\cC,g_\cC)$, and the superscript $[\cdot]^T$ denotes the projection of a vector at $\cC_0$ to $T \cC_0$ with respect to $g_R$, i.e. 
\be\label{prj}
[F_*(T^c_g+\nu_g)]^T=F_*(T^c_g+\nu_g)-g_R(F_*(T^c_g+\nu_g),\nu_{g_R})\cdot\nu_{g_R}
\ee
where $\nu_{g_R}$ is the outward unit normal to $\cC_0\subset(M_0,g_R)$. The vector field $\Theta_{\cC}$ is intrinsic to the boundary $\cC_0$, in contrast to the boundary data $\Theta$ in \eqref{B1}. 

The existence of solutions to the IBVP's \eqref{main}-\eqref{b1} and \eqref{main2}-\eqref{b2} in sufficiently small neighborhoods of a corner point $p \in \Si$ relies on the existence of strong or boundary stable energy estimates for their localized or frozen coefficient systems.  
In the following we will reduce these systems using a local harmonic gauge (independent of $F$) and calculate the linearized systems in a local corner neighborhood. Energy estimates for these systems are then derived in \S 3. These together with basically standard methods from the theory of quasi-linear hyperbolic systems of IBVP's are used to establish local existence of solutions (cf.~Theorem \ref{exist1}, \ref{exist2}). 

\subsection{The expanded systems in the harmonic gauge}

In the following, we consider the IBVP's \eqref{main}-\eqref{b1} and \eqref{main2}-\eqref{b2} in a neighborhood of a corner point of $M$, and use the harmonic gauge to reduce the vacuum Einstein equations to be strictly hyperbolic. Then following the standard localization or frozen coefficient method of proving well-posedness of IBVP, we set up the linearizations of the problems at background flat solutions. The system \eqref{main}-\eqref{b1} is discussed first, and then followed with a similar analysis for the system \eqref{main2}-\eqref{b2}. The following conventional index notation will be used throughout: Greek letters $\a \in \{0,1,2,3\}$, lower case Roman indices $i \in \{1,2,3\}$ while upper case Roman $A \in \{2, 3\}$. Similarly, the Einstein summation convention that repeated indices are summed will always be used. 
 
Choose standard Cartesian coordinates $\{x^{\a}\} _{\a=0,1,2,3}= \{x^0=t, x^i\}_{i=1,2,3}$ on $\bR^4$. The standard corner domain is given by ${\bf R} = \{t \geq 0, x^1 \leq 0\}$, so that $\p_t$ is future pointing, $\p_{x^1}$ is outward pointing on the boundary hypersurface $\{x^1=0\}$, and $\p_{x^A}~(A=2,3)$ is tangent to the corner $\{t=x^1=0\}$. We work locally and in a neighborhood $U \subset M$ of an arbitrary corner point $p \in \Si = S\cap \cC$. Assume that $U$ is in the domain of a chart $\chi: U \to {\bf R}$ such that $\chi(p) = 0$. In addition, $\chi$ carries the boundary $\cC\cap U$ to the locus $\{x^1 = 0\}$ and carries the initial surface $S\cap U$ to the locus $\{t = 0\}$. The corner $\Si\cap U$ is thus mapped to a flat domain in $\bR^2$ with coordinates $\{x^A\}_{A=2,3}$. In the following we call a local chart $\chi$ which satisfies the conditions above as a \textit{standard corner chart} at $p$; and call the domain $U$ for this chart as a \textit{local corner neighborhood}. The same process can be carried on the target manifold $(M_0,g_R,t_0)$. So for any corner point $q\in\Si_0$ there is a local corner neighborhood $U_0$ admitting a standard corner chart $\chi_0:U_0\to {\bf R}_0$ which defines coordinates $\{x_0^\a\}_{\a=0,1,2,3}$ on $U_0$, (here ${\bf R}_0={\bf R}$ and we use ${\bf R}_0$ to emphasize $\chi_0$ is a chart on the target manifold $M_0$). In addition, it will always be assumed that the time function $x_0^0$ in a standard corner chart $\chi_0$ of $M_0$ equals to the fixed time function $t_0$ on $M_0$. 

To solve for a local solution of the system \eqref{main}-\eqref{b1} on $U$, we will expand the coordinate-free system to a system of (nonlinear) hyperbolic equations with complete initial and boundary conditions in the chart $\chi$. To begin, according to common practice we introduce a local gauge condition. Given a chart $\chi$ with coordinate functions $x^\a~(\a=0,1,2,3)$, for any metric $g$, let $V_g=V_g(\chi)$ be the vector field on $U$ given by 
\be
\label{gauge}V_g= (\Box_{g}x^{\a})\cdot\p_{x^\a}.
\ee
(The field $V_g$ may be viewed as the tension field of the identity map ${\rm Id}: (U, g) \to (\bR^4, g_R)$, cf.~\cite{Geba-Grillakis:2017} for example). When $V_g = 0$, the coordinates $x^{\a}$ of $\chi$ are harmonic (wave) coordinates with respect to $g$. {In the following, we will use $\d^*_g$ to denote the formal adjoint of the divergence operator $\d_g=-{\rm div}_g$, i.e. for a vector field $V$, $\d^*_gV=\tfrac{1}{2}\cL_V g$.}
For the system \eqref{main}-\eqref{b1}, we then consider the following system of reduced Einstein equations coupled to the wave map $F$ in a local corner neighborhood $U$ : 
\be \label{main3}
\begin{split}
\begin{cases}
\Ric_g+\delta_g^*V_g=0\\
\Box_g F + \Gamma_{g_R}(F)g(\nabla F, \nabla F) + F_*(V_g) = 0
\end{cases}
\quad\text{in }U
\end{split}
\ee
with initial conditions: 
\be \label{i3}
\begin{split}
\begin{cases}
g = q, \ \ {\tfrac{1}{2}}\cL_{T^0_g}g = k, \ \ V_g = 0 \\
F=E_0,~F_*(N_g) = E_1\\
\end{cases}
\quad\text{on }U\cap S
\end{split}
\ee
and boundary conditions: 
\be \label{b3}
\begin{split}
\begin{cases}
V_g=0\\
F=G\\
[g_F^{\intercal}]=[\s]\\
F_*(T_g+\nu_g) = \Theta \\
\end{cases}
\quad\text{on }U\cap\mathcal C.
\end{split}
\ee
It is well-known that the equations in \eqref{main3} are of the form 
\be \label{harmE}
g^{\g\t}\p_\g\p_\t g_{\a\b} + Q_{\a\b}(g, \p g) = 0,\ \ { g^{\g\t}\p_\g\p_\t F^\a + P^\a(g, \p g, F, \p F) = 0}
\ee
where $Q$ is quadratic in $g$ and $\p g$; and $P$ is a polynomial in $g$, $\p g$, $F$ and $\p F$. 
Here we regard $F$ as a map from $U$ to an open set $U_0\subset M_0$. 
When the initial and boundary Dirichlet data $E_0, G$ are given, $U_0$ is understood as an open neighborhood of the target corner point $p_0=E_0(p)=G(p)\in \Sigma_0$. Local representations $F^\a$ of $F$ are given by $F^\a = \chi_0\circ F \circ \chi^{-1}: \chi(U) \subset {\bf R} \to {\bf R}_0$, where $\chi_0$ is a standard corner chart at the image $p_0\in\Si_0$. 
The term $F_*(V_g)$ in the second equation in \eqref{main3} is introduced to simplify the form of the linearization, cf.~\eqref{Flin} below.

The initial and boundary conditions \eqref{i3}-\eqref{b3} are understood to be the restriction of equations \eqref{i1}-\eqref{b1} to $U$ plus choices of gauge source functions on the (local) initial surface. Different from the initial conditions in \eqref{i1} where we only prescribe the 3-dim Riemannian metric $g_S$ and 3-dim symmetric 2-tensor $K_{g|S}$ on $S$, in \eqref{i3} $q$ is a 4-dim Lorentz metric on $M$ restricted to $S$ while $k$ is a 4-dim symmetric bilinear form on $M$ restricted to $S$. The pair $(q,k)$ is understood as an extension of the 3-dim geometric initial data $(\g,\k)$ via a certain choice of gauge source functions (lapse function and shift vector of the spacetime), and it prescribes the full spacetime metric $g$ and the full Lie-derivative $\tfrac{1}{2}\cL_{T_g^0}g$ at $S$. Here $T_g^0$ denotes the future pointing timelike unit normal vector to the initial surface $S\cap U$ in the spacetime $(U,g)$. (By definition of the second fundamental form, $K_{g|S}$ is obtained by restricting $\tfrac{1}{2}\cL_{T_g^0}g$ to the tangent space of $S$.) The initial condition $V_g=0$ in \eqref{i3} is an implicit restriction on the choice of the initial data $(q,k)$. We write it explicitly to emphasize this gauge condition. 

Note here (as is also mentioned in Remark~\ref{N_g_con}) we prescribe the initial value for $F_*(N_g)$ in \eqref{i3} {instead of $F_*(\cN_g)$ with a general choice of $\cN_g$.} We choose $F_*(N_g)$ here and in the following analysis because the corner compatibility conditions in this case have simpler expressions. One can also analyze the case with more general $F_*(\cN_g)$ by following the same lines below. 

Recall that $({\bf I}, {\bf B})$ denotes the geometric initial-boundary data as in \eqref{i1}-\eqref{b1}. We will make the following assumptions on the global data $({\bf I},{\bf B})$ in \eqref{i1}-\eqref{b1} throughout the paper:
\begin{enumerate}
\item The Riemanian metric $\g$ and symmetric bilinear form $\k$ in \eqref{i1} satisfy the constraint equations \eqref{constraint}.
\item The global initial data $E_0:S\to S_0$ is diffeomorphism in a collar neighborhood of $\Si$ in $S$, and for an ST-corner neighborhood $\cU$ the restricted map $E_0|_\cU: S \cap \cU \to S_0\cap \cU_0$ in \eqref{i3} is an orientation preserving diffeomorphism onto its image in $S_0$. 
\item The vector field {$E_1|_\cU:S\cap \cU\to E_0^*(TM)|_{S_0\cap \cU_0}$} is transverse to $S_0$, and at the corner $E_1|_\Si$ is tangent to $\cC_0$. 
\item The map $G|_\cU: \cC\cap \cU \to  \cC_0\cap \cU_0$ in \eqref{b3} is the restriction of a diffeomorphism $\cC \to \cC_0$ which is both 
orientation and time-orientation preserving. 
\end{enumerate}

We will use $(I, B)$ to denote the local extended (or gauged) initial-boundary data raised from $({\bf I},{\bf B})$, i.e.~$I=(q,k,E_0,E_1),~B=(G,[\s],\Theta)$ on the initial and boundary surface of $U$ as in \eqref{i3}-\eqref{b3}. One can always choose the appropriate gauge source functions (locally) so that the extended initial-boundary data $(I,B)$ satisfy the following conditions:
\begin{enumerate}
\item The Lorentz metric $q$ is chosen so that the induced metric on $S\cap U$ by $q$ is equal to the restriction of $\g$ from \eqref{i1} on $S\cap U$, i.e. in a standard corner chart $q_{ij}=\g_{ij}$ for $i,j=1,2,3$ on $\{t_0=0\}$, and we consider $q_{0\a}~(\a=0,1,2,3)$ as free gauge source functions.
\item The symmetric bilinear form $k$ is chosen so that the induced symmetric 2-tensor on $S\cap U$ by $k$ is equal to the restriction of $\k$ from \eqref{i1} on $S\cap U$, i.e.~in a standard corner chart $k_{ij}=\k_{ij}$ for $i,j=1,2,3$ on $\{t_0=0\}$, and we consider $k_{0\a}~(\a=0,1,2,3)$ as free gauge source functions.
\item The gauge source functions $(q_{0\a}, k_{0\a})$ are chosen so that $\Box_{g}x_0^{\a}=0$ on $S$, which is exactly the gauge constraint $V_g = 0$ listed explicitly in \eqref{i3}. 
\end{enumerate}

In addition to the assumptions above, the initial-boundary data $({\bf I},{\bf B})$ and the extended local initial-boundary data $(I,B)$ must also satisfy $C^k$ compatibility assumptions at the corner $\Si$ for suitable $k \geq 1$. 
The $C^k$ compatibility conditions are the relations induced between the initial data $I$ and boundary data $B$ at the corner $\Si$ by a solution $(g, F)$ of the system \eqref{main3}-\eqref{b3} which is $C^k\times C^{k+1}$ up to the boundary $\p M = S \cup \cC$, (i.e.~the data $(g, F)$ extend as $C^k \times C^{k+1}$ data to an open neighborhood of the closed domain $M$). 
{ We list the conditions for $(g,F)$ to be at least $C^0\times C^1$ up to the boundary below.} Higher order compatibility requires using the bulk equations \eqref{main3} and \eqref{harmE} to replace $\p_{t}^2$ and higher order $t$-derivatives by $x^{\a}$ derivatives of lower order in $t$. Since it will not be necessary, we do not explicitly express the (complicated) higher order compatibility relations. 

Notice that the Riemannian metric on the corner $\Si_0$ of $M_0$ induced by the pull-back metric $g_F=(F^{-1})^*g$ is given by $(F^{-1})^*g|_{\Si_0}=(E_0^{-1})^*\g|_{\Si_0}$ according to the initial conditions in \eqref{i1}. On the other hand, the (conformal) metric on $\Si_0$ is prescribed by the boundary data $[\s]$ restricted on $\Si_0$. So $(\bf I,\bf B)$ must satisfy 
\be\label{cib1}[(E_0^{-1})^*\g|_{\Si_0}]=[\s]\ \ \mbox{on }\Si_0.\ee
Similarly, since the restriction of the map $F$ on the corner $\Si$ is prescribed by both the restriction of the initial data $E_0|_\Si$ and the restriction of the boundary data $G|_\Si$, $C^0$ compatibility also requires
\be\label{cib2}E_0|_{\Si}=G|_{\Si}.\ee
Moreover, observe that at the corner $T_g+\nu_g$ is in the vector space spanned by $N_g$ and $n_{g_S}$ and the same is true for the push-forward vectors via $F$. Thus for $F$ to be $C^1$ up to the boundary, $(\bf I,\bf B)$ must also satisfy 
\be\label{cib3}\Theta=\l_1 E_1+\l_2 (E_0)_*(-n_\g)\ \ \mbox{on }\Si_0\ee
for some functions $\l_1,\l_2$ on $\Si_0$. By further inspection, we note $\l_1,\l_2$ must be chosen such that $\l_1^2+\l_2^2=2,~\l_1+\l_2>0$ ({cf. \S 6.3} for the detailed calculation).

{ All the compatibility conditions above on $({\bf I},{\bf B})$ naturally induces corner conditions on the local data $(I,B)$. In addition, for the solution $(g,F)$ of the system \eqref{main3}-\eqref{b3} to be $C^0\times C^1$ up to the boundary, the data $(I,B)$ should further satisfy} 
\be\label{cib4}N_g(q)=G^{-1}_*(E_1),~q(N_g,n_\g)=\tfrac{\l_1-\l_2}{\l_1+\l_2}\ \ \mbox{on }\Si\cap U.\ee
The first equation above is due to the fact that the normal vector $N_g=N_g(q)$ is determined by the full Lorentz metric $q$ at the corner, while based on the boundary condition \eqref{b1} $N_g$ is also given by $N_g=F^{-1}_*(E_1)=G^{-1}_*(E_1)$ at the corner. The second equation above means the functions $\l_1,\l_2$ in the choice of $\Theta$ in \eqref{cib3} determines the corner angle (angle between the hypersurfaces $S$ and $\cC_0$), i.e. $q(N_g,n_\g)$ in the choice of the gauge source functions $q_{0\a}$ (cf. \S 6.3 for the detailed calculation).

\begin{remark} \label{loc_diff}
{\rm {\bf (i).} 
As stated in the Introduction, $\nu_g$ is the spacelike outward unit normal to $\cC \cap U \subset (U, g)$. In particular $\nu_g$ is never a null-vector and so never tangent to $\cC$. It follows then from the boundary conditions in \eqref{b3}, (as well as the boundary condtions in \eqref{b1} or \eqref{b2}), that $(\cC\cap U, g)$ is Lorentzian (i.e.~$\cC \cap U$ is timelike with respect to $g$) on its full domain.

 {\bf (ii).} Note also the assumptions above that $E_0: S \to S$ and $G: \cC \to \cC_0$ are diffeomorphisms near $\Si$ and that the vector field $E_1$ is transverse to $S$ imply that for any solution $(g, F)$, there are ST-corner neighborhoods $\cU$ of $\Si$ in $M$ and $\cU_0$ of $\Si_0$ in $M_0$ such that $F$ induces a diffeomorphism $F|_{\cU}: \cU \to \cU_0 = F(\cU) \subset M_0$. 
}
\end{remark}

It is well-known that solutions to the reduced vacuum Einstein equations give rise to vacuum metrics in the harmonic gauge. For simplicity of notations, throughout the following, we will use $\Upsilon$ to denote the boundary of the ST-spacetime, i.e. $\Upsilon=S\cap\cC$.
\begin{lemma}\label{Hvac1} Suppose $U$ is a local corner neighborhood around $p\in\Si$ with a standard corner chart $\chi=\{x^\a\}_{\a=0,1,2,3}~(x^0=t)$, and $g$ is a spacetime metric on $U$ for which $S\cap U$ is spacelike and $\cC\cap U$ is timelike. If $g$ solves the gauged Einstein equations
\be\label{vac1}Ric_g+\d^*_g V_g=0\ \ \mbox{on }U\ee
where $V_g=V_g(\chi)$ is defined as in \eqref{gauge} and $V_g=0$ on $S\cap U$ and $\cC\cap U$, then 
\be \label{V0}
V_g = 0  \ \ {\rm on} \ \ \cD^+(\Upsilon\cap U,g).
\ee
Here $\cD^+(\Upsilon\cap U,g)$ is the future domain of dependence of the hypersurface $\Upsilon\cap U=(S\cup\cC)\cap U$ in $(U,g)$.
Thus $g$ is a Ricci flat metric with harmonic (or wave) coordinate chart $\chi$ in $\cD^+(\Upsilon\cap U,g)$. Moreover, if an infinitesimal deformation $h$ is a solution of the linearization of the equation \eqref{vac1} at a solution $g$, with $V'_h = 0$ on $S\cap U$ and $\cC\cap U$, then 
$$V'_h = 0 \ \ {\rm on} \ \ \cD^+(\Upsilon\cup\cC,g).$$ 
\end{lemma}

\begin{proof}  {Let $\b_g = -{\rm div}_g + \frac{1}{2}\tr_g$ be the Bianchi operator with respect to $g$ on symmetric bilinear forms.} 
The Bianchi identity applied to \eqref{vac1} gives 
$$\b_g \d^{*}_gV_g = 0 \ \ {\rm on} \ \ U.$$
By a standard Weitzenbock formula, $2\b_g \d^{*}_g V_g = -\Box_g V_g -  Ric_g(V_g)$. Thus it follows that $\Box_g V_g +  Ric_g(V_g)=0$. This is a linear system of wave equations on $V_g$. As is well-known, cf.~\cite{Hawking-Ellis:1973}, given $V_g = 0$ on $S\cap U$, the constraint equations \eqref{constraint} imply that $\p_{t} V_g = 0$ on $S\cap U$, so the initial data for $V_g$ vanish. Since $V_g = 0$ also on the timelike boundary $\cC\cap U$, standard results on uniqueness of solutions of such linear wave systems imply \eqref{V0}, cf.~\cite{Benzoni-Serre:2007} for instance. The same argument applies to the linearized problem. 
\end{proof}

Conversely, a general spacetime metric on $U$ can be brought into the harmonic gauge by suitable diffeomorphisms in $\mathrm{Diff}_1(U)$ -- the space consisting of maps $\psi:U\to M$ with $\psi$ being a diffeomorphism from $U$ onto its image $\psi(U)$ and $\psi$ is equal to the identity to the first order on $S\cap U$ and equal to the identity to the zero order on $\cC\cap U$. 

\begin{lemma}\label{Hvac2}
Suppose $U$ is a local corner neighborhood around $p\in\Si$ with a standard corner chart $\chi=\{x^\a\}_{\a=0,1,2,3}~(x^0=t)$, and $g$ is a spacetime metric on $U$ for which $S\cap U$ is spacelike and $\cC\cap U$ is timelike. Then there is an open subset $U'\subset U$ covering $S\cap U$ and a diffeomorphism $\psi \in \mathrm{Diff}_1(U')$,  such that 
\be \label{wV}
V_{\psi^*g}(\chi) =(\Box_{\psi^*g}x^{\a}) \p_{x^\a}= 0 \ \ {\rm on} \ \ U'.
\ee
\end{lemma}

\begin{proof} 
Given the background coordinates $x^{\a}$ on $U$, define new coordinates $\w x^{\a}$ by solving the wave equations 
$$\Box_g \w x^{\a} = 0,$$
with the same initial and boundary conditions as that formed by $x^{\a}$, i.e. $\w x^{\a}= x^{\a},\partial_{t}\w x^{\a}=\partial_{t}x^{\a}$ on $U\cap S$ and $\w x^{\a}=x^{\a}$ on $\cC\cap U$, for $\a=0,1,2,3$. This a well-posed IBVP for the simple linear wave equation, and so there is a unique solution in $\cD^+(\Upsilon\cap U,g)$. Based on the initial and boundary conditions of $\w x^{\a}$, there is an open subdomain $U'\subset U$ with $S\cap U'=S\cap U$ such that the map $\psi:U'\to U$ given by
$$\tilde x^{\a}\circ{\psi(q)} = x^{\a}(q)\ \ \forall q\in U,~\a=0,1,2,3,$$
is well-defined and $\psi$ is a diffeomorphism from $U'$ onto its image $\psi(U')$. Clearly $\psi \in \mathrm{Diff}_1(U')$, i.e. $\psi|_{\Upsilon\cap U'}={\rm Id}_{\Upsilon\cap U'}$ and on $S\cap U'$, $D\psi={\rm Id}$. Then the pull-back $\w g = \psi^*g$ satisfies \eqref{wV}. Moreover, it is proved in \cite{Planchon-Rodnianski}, see also \cite{Parlongue:2011}, that for $g \in H^s(S_t)$, the new background coordinates $\w x^{\a} \in H^{s+1}(S_t)$. 
\end{proof}

\subsection{Localization and linearization of the expanded systems} 

The system \eqref{main3}-\eqref{b3} is a quasi-linear hyperbolic system with mixed-type boundary conditions. The well-posedness of such an IBVP rests upon analyzing the behavior in small regions, linearized around a point $p \in \Si$, (the frozen coefficient method). We next discuss in detail how this localization is done in the current setting. 

For convenience, consider first a special case. Let ${\bf R}_0$ be a copy of the standard corner domain ${\bf R}$ with coordinates $\{x_0^\a\}_{\a=0,1,2,3}~(x_0^0=t_0)$, i.e. ${\bf R}_0=\{t_0\geq 0,~x_0^1\leq 0\}$; and let $\bar g_{0R} = dt_0^2 + \sum_{i=1}^3 (dx_0^i)^2$ be the standard Euclidean metric on $\bR^4\supset{\bf R}_0$. The pair $(\bar g_0, \bar F_0)$ consisting of the Minkowski metric $\bar g_0=-dt^2+\sum_{i=1}^3 (dx^i)^2$ on ${\bf R}$ and the identity map $\bar F_0={\rm Id}:({\bf R},\bar g_0)\to({\bf R}_0,\bar g_{0R})$, given by $x_0^\a\circ \bar F_0=x^\a$, is then the unique solution to \eqref{main3}-\eqref{b3} on ${\bf R}$ with Cartesian initial-boundary data $(\bar I_0,\bar B_0)$ given by
\begin{equation*}
\begin{split}
&\bar I_0=\big( q_{\alpha\beta}=\eta_{\alpha\beta},~k_{\alpha\beta}=0,~E_0={\rm Id}_{\{t_0=0\}},~E_1=\p_{t_0}\big),\\
&\bar B_0=\big(G={\rm Id}_{\{x_0^1=0\}},~[\s]=[\delta_{AB}],~\Theta = \partial_{t_0}+\partial_{x_0^1}\big).
\end{split}
\end{equation*}
Here $\eta_{\alpha\beta} =\text{diag}(-1,1,1,1)$ and $\delta_{AB}=\text{diag}(1,1)$.  

More generally, let $g_{0R} = (g_{0R})_{\a\b}dx_0^{\a}\cdot dx_0^{\b}$ be a complete flat Riemannian metric on ${\bf R}_0$, with $(g_{0R})_{\a\b}$ being constant functions. Then the pair $(g_0, F_0)$ consisting of a flat Lorentz (Minkowski-type) metric $g_0= (g_0)_{\a\b}dx^{\a}dx^{\b}$ and a linear map $F_0 = L: ({\bf R},g_0) \to ({\bf R}_0,g_{0R})$ given by $x_0^\a\circ L=L_{\a\b}\cdot x^\b$ with constant coefficients $ (g_0)_{\a\b},L_{\a\b}$, is the unique solution to \eqref{main3}-\eqref{b3} on ${\bf R}$ with flat initial-boundary data $(I_0,B_0)$ given by
\be \label{flat}
\begin{split}
&I_0=\big(q_{\a\b}=(g_0)_{\a\b},~k_{\a\b}=0,~E_0=L_0,~E_1=L_1\big)\\
&B_0=\big(G=G_0,~[\s]=[\s_0],~\Theta = \Theta_0\big).
\end{split}
\ee 
In the above, $L_0$ is the restriction $L|_{\{t=0\}}$ of the linear map $L$, and $L_1$ is the constant vector field determined by $L$ and $g_0$, i.e. $L_1=L_*(N_{g_0})$. Because the wave map $F$ sends boundary hypersurfaces to themselves, here we require $L: \{t= 0\} \to \{t_0= 0\}$, $L: \{x^1 = 0\} \to \{x_0^1 = 0\}$. Then it follows that $L$ maps the level sets $\{t=\text{constant}\}$ in ${\bf R}$ to the level sets $\{t_0=\text{constant}\}$ in ${\bf R}_0$, since $L$ is a constant linear map. Moreover, $G_0$ is the restriction $L|_{\{x^1=0\}}$ of the linear map $L$, $\s_0$ is pointwisely conformal to the flat metric $((L^{-1})^*g_0)^\intercal$ on $\{t_0=\text{constant}, x_0^1=0\}$, and $\Theta_0$ is the constant vector field given by $L_0=L_*(T_{g_0}+\nu_{g_0})$.

We now show that for a general choice of initial-boundary data $(I,B)$ in \eqref{i3}-\eqref{b3}, the IBVP \eqref{main3}-\eqref{b3} can be reduced to a problem with initial-boundary data sufficiently close to the flat data \eqref{flat} above.

Suppose $I=(q,k,E_0,E_1),~B=(G, [\g],\Theta)$ are arbitrary data given in a neighborhood of a corner point $p \in \Si$ which satisfies the compatibility conditions. The map $E_0$ maps $p$ to $p_0=E_0(p) \in \Si_0$. Choose a pair $\chi$, $\chi_0$ of standard corner charts at $p,p_0$, so $\chi(p) = 0$ and $\chi_0(p_0) = 0$. Let $x^{\a}$ denote the coordinates in $\chi$ and $x_0^{\a}$ the coordinates in $\chi_0$. Now choose the new chart $\w\chi$ with coordinates $\tilde x^{\a} = \l^{-1}x^{\a}$ near $p$ and correspondingly $\w\chi_0$ with $\tilde x^{\alpha}_0=\lambda^{-1}x^{\alpha}_0$ near $p_0$, with $\lambda$ being a positive real number. Define new initial data $\w I=(\w q,\w k,\w E_0,\w E_1)$ in the following way
\begin{equation*}
\begin{split}
\w q_{\alpha\beta}(\tilde x)=q_{\alpha\beta}(\lambda\tilde x),~
\w k_{\alpha \beta}(\tilde x)=\lambda k_{\alpha\beta}(\lambda\tilde x),~
\w E_0^{ \alpha}(\tilde x)=\lambda^{-1}E_0^{\alpha}(\lambda\tilde x),~
\w E_1^{ \alpha}(\tilde x)=E_1^{\alpha}(\lambda\tilde x).
\end{split}
\end{equation*}
Here $\w q_{\alpha\beta}(\tilde x)$ denotes the component of $\w q$ at the point $\tilde x$ in the new chart $\w\chi$, i.e. $\w q_{\alpha\beta}(\tilde x)=\w q(\p_{\tilde x^\a},\p_{\tilde x^\b})|_{\tilde x}$. It equals to the corresponding component $q_{\alpha\beta}(x)= q(\p_{x^\a},\p_{x^\b})|_{x}$ of $q$ at the point $x=\lambda\tilde x$ in the original chart $\chi$. The same meaning applies to the defining equations of $\w k$. Furthermore, $\w E_0^{\alpha}(\tilde x)$ denotes the $\tilde x_0^{\a}$ component of the image $\w E_0(\tilde x)$ in the new chart $\w\chi_0$, i.e. $\w E_0^{ \alpha}(\tilde x)=\tilde x_0^\a\circ \w E_0 (\tilde x)$. It equals to the rescaled ($\lambda^{-1}$) component $E^{\a}_0(\lambda \tilde x)=x_0^\a\circ E_0(\lambda \tilde x)$ expressed in the chart $\chi_0$. The equation for $\w E_1$ means that $\w E_1$ assigns the point $\tilde x$ with the vector $\w E^{\a}_1(\tilde x)\p_{\tilde x_0^\a}$ at the image point $\w E_0(\tilde x)$.  The coefficient function $\w E^{\a}_1(\tilde x)$ equals to $E^{\a}_1(\l\tilde x)$, which is the coefficient of $E_1$ in the chart $\chi_0$, i.e. $E_1(\l\tilde x)=E^{\a}_1(\l\tilde x)\p_{x_0^\a}$. 

Notice that when $\lambda$ is very small, $\w q$ is very close to the Minkowski-type metric $g_0= q_{\a\b}(0)d\tilde x^{\a} d\tilde x^{\b}$ where the components of the metric in $\w\chi$ are constants $(g_0)_{\a \b}=q_{\a\b}(0)$; similarly $\w k$ is very close to zero. Moreover, as $\l \to 0$, the map $\w E_0$ approaches to the constant linear map $L_0:\{\tilde t=0\}\to\{\tilde t_0=0\}$ given by $\tilde x_0^i\circ L_0=(L_0)_{ij} \cdot \tilde x^j~(i,j=1,2,3)$ where $(L_0)_{ij}$ equals to the constant coefficient of the linearization of $E_0:\{t=0\}\to\{t_0=0\}$ at the origin, i.e. ~$(L_0)_{ij}=\p_{x^j}(x_0^i\circ E_0)|_{x=0}$. Since $E_0$ maps the corner to the corner, the map $L_0$ must also map the corner $\{\tilde t=\tilde x^1=0\}$ to the corner $\{\tilde t_0=\tilde x_0^1=0\}$. At the same time $\tilde E_1$ becomes close to the constant vector field $L_1=E_1|_{x=0}$.

Similarly define new boundary data $\w B=(\w G, [\w \s], \w\Theta )$ as 
\begin{equation*}
\begin{split}
\w G^{\alpha}(\tilde x)=\lambda^{-1}G^{\alpha}(\lambda\tilde x),~
\w \s_{ A B}(\tilde x_0)=\s_{AB}(\lambda\tilde x_0),~
\w \Theta^{ \alpha}(\tilde x_0)=\Theta^{\alpha}(\lambda\tilde x_0).
\end{split}
\end{equation*}
It is easy to check that when $\lambda\to 0$, one has $\w G\to G_0$ where $G_0:\{\tilde x^1=0\}\to\{\tilde x_0^1=0\}$ is given by $\tilde x_0^\a\circ G_0=\big(\p_{x^\b}(x_0^\a\circ G)|_{x=0}\big)\cdot \tilde x^\b~(\a,\b=0,2,3)$ in analogy to the limiting approach of $\w E_0$ above. In addition, since $G$ maps the corner to the corner, the map $G_0$ also maps the corner $\{\tilde t=\tilde x^1=0\}$ to the corner $\{\tilde t_0=\tilde x_0^1=0\}$. 
At the same time, as $\l\to 0$, $\w\s\to \s_0= \s|_{x_0=0}$ and $\w\Theta\to\Theta_0=\Theta|_{x_0=0}$. 

Moreover, based on the compatibility conditions \eqref{cib2}, $L_0=G_0$ on $\{\tilde t=\tilde x^1=0\}$. Thus there is a unique constant linear transformation $L:{\bf R}\to{\bf R}_0$ such that $ L|_{\{\tilde t=0\}}=L_0$ and $ L|_{\{\tilde x^1=0\}}=G_0$. 
Compatibility equation \eqref{cib4} further gives $G_*(N_g)=E_1$ at $p$, so the limit $L_1$ of $\w E_1$ (as $\lambda\to0$) satisfies $L_1= L_*(N_{g_0})$. 
For the same reason, the limit of $\w \Theta$ satisfies $\Theta_0 = L_*(T_{\tilde g_0} + \nu_{\tilde g_0})$ and the limit of $\w \s$ satisfies $[\s_0]=[\big((L^{-1})^*g_0\big)^\intercal]$. 

The analysis above shows that the rescaled initial-boundary data $(\w I,\w B)$ limits to a set of flat initial-boundary data which is exactly of the form \eqref{flat}. Meanwhile, the same rescaling process is also applied to the given background Riemannian metric $g_R$ on $U \subset M$, i.e. set $(\w g_R)_{\a \b}(\tilde x_0)$ in the chart $\w\chi_0$ equal to $(g_R)_{\a\b}(\lambda \tilde x_0)$ in the chart $\chi_0$. It then limits to a flat metric $g_{0R}= (g_R)_{\alpha\beta}(0)d\tilde x_0^{\a} d\tilde x_0^{\b}$ as $\l \to 0$. 

One may now set up a hyperbolic system in the same way as \eqref{main3}-\eqref{b3} in the new charts $\w\chi,\w\chi_0$ with initial-boundary data $(\w I,\w B)$ constructed above in a fixed size neighborhood $\w U$ of $p$ and with the rescaled background metric $\w g_R$. Choose $\l$ small enough so that the data $(\w I,\w B)$ is sufficiently close to flat-type data as described above, and the terms in the wave equation of $F$ with coefficients contributed by $\Gamma_{\w g_R}$ are close to zero. If $(\w g, \w F)$ is a solution of this rescaled IBVP in a (possibly smaller) domain $\w U$, then the pair $g(x)=\l^2\w g(\l^{-1}x)~F(x)=\w F(\lambda^{-1}x)$ solves the system \eqref{main3}-\eqref{b3} with the original initial-boundary data $(I,B)$ and background metric $g_R$. Thus there is a one-to-one correspondence between such local solutions in sufficiently small neighborhoods of corner points $p \in \Si$. 

For the rest of this section, we assume that the initial-boundary data $(I, B)$ has been localized as above, so they are close to flat data. To analyse the solvability of the system \eqref{main3}-\eqref{b3} near flat data $(g_0,F_0)$ on $\bf R$ (with the flat Riemannian metric $g_{0R}$ on ${\bf R}_0$), we first consider its linearization at $(g_0,F_0)$. Without loss of generality (by linear transformation of the chart), we can assume $g_0=-(dt)^2+\sum_{i=1}^3(dx^i)^2$ in the standard corner chart $\chi$ at $p$. For convenience, here we make the assumption that $\cC\cap U$ is orthogonal to $S\cap U$ with respect to the flat metric $g_0$. In particular, one has then $T_{g_0}=\partial_t$ and $\nu_{g_0}=\partial_{x^1}$. This assumption, although not necessary, simplifies some of the computations to follow in \S 3. As noted in \cite{Kreiss-Winicour:2014}, it can always be realized by choosing a new initial slice $S' \subset \cV \subset M$, with $\p S' = \p S = \Si$, where $\cV$ is the (maximal) Cauchy development of the initial data set $(S, I)$; { and well-posedness of the IBVP in the orthogonal case implies well-posedness in general}.  

In the following $(h,f)$ denotes an infinitesimal deformation at $(g_0,F_0)$, i.e. $h=\tfrac{d}{ds}|_{s=0}g(s), f=\tfrac{d}{ds}|_{s=0}F(s)$ for a smooth family of solutions $(g(s),F(s))$ starting at $g(0)=g_0, F(0)=F_0$. We will use $'$ to denote the linearization of geometric quantities at $(g_0,F_0)$. For example, $\Box'_h=\tfrac{d}{ds}|_{s=0} \Box_{g(s)}$ is the linearization of the wave operator. 
It is well-known that the linearization of the first equation in \eqref{main3} at $g_0$ is given by the standard system $\Box_{g_{0}}h_{\a\b} = 0$ of wave equations. Here and in the following, we work in ${\bf R}$ with harmonic coordinates $x^{\a}$ and use the abbreviation $\p_{\a}=\p_{x^\a}$. For the second equation in \eqref{main3}, observe that the linearization of the term involving $\Gamma_{g_R}(F)$ with $g_R=g_{0R}$ vanishes and 
\be \label{Flin}
\begin{split}
&(\Box_{g}F+F_*(V_g))'_{(h,f)}=\Box_{g_0}f+\Box'_hF_0+(F_0)_*(V'_h)\\
&=\Box_{g_0}f+g_0(D_{g_0}^2F_0,h)-g_0(dF_0,\beta_{g_0}h)+(F_0)_*(V'_h) =\Box_{g_0}f,
\end{split}
\ee
where we use the fact that $V_{g_0}= 0$, $V'_h=\beta_{g_0}h$ and $D^2_{g_0}F_0=0$. Thus, writing $f^\a=x_0^\a\circ f$, the linearization of the system \eqref{main3}-\eqref{b3} at $(g_0,F_0)$ is given by:
$$
\Box_{g_0} h_{\a\b}=0,\ \ \Box_{g_0} f^\a=0\ \ (\a,\b=0,1,2,3)
\quad\text{in }\mathbf R
$$
with initial and boundary conditions 
\be \label{lini1}
\begin{split}
\begin{cases}
h_{\a\b}=q'_{\a\b},~(\tfrac{1}{2}\cL_{T^0_g}g)'_h=k',~ V'_h = 0\\
f=E'_0, \ \ \big(F_*(N_g)\big)'_{(h,f)}=E'_1
\end{cases}
\quad\text{on }\{t=0\},
\end{split}
\ee
\be \label{linb1}
\begin{split}
\begin{cases}
V'_h=0\\
f=G' \\
([g_F^\intercal])'_{(h,f)}=[\s']\\
\big(F_*(T_g+\nu_g)\big)'_{(h,f)}=\Theta'
\end{cases}
\quad\text{on }\{x^1=0\}.
\end{split}
\ee

We discuss in more detail the linearization of the initial and boundary conditions. 
{ In the above, we use $\{(q',k',E_0',E_1'),(G',\s',\Theta')\}$ to denote a deformation of the local initial-boundary data $(I,B)$ and we assume that these quantities are extended to be defined on the entire boundary hypersurfaces $\{t=0\}$ or $\{x^1=0\}$ of $\mathbf R$ by composing with a compactly supported bump function on the boundary hypersurface $\Upsilon\cap U$ which equals to 1 near the fixed corner point $p\in\Si$.}

To begin, linearization of the unit normal vector $T_g^0$ on the initial surface is given by $(T_g^0)'_h=\tfrac{1}{2}h_{00}\p_0-h_{0i}\p_i=\tfrac{1}{2}q'_{00}\p_0-q'_{0i}\p_i$, which is uniquely determined by the data $q'$. Furthermore, the equation $(\tfrac{1}{2}\cL_{T^0_g}g)'_h=k'$ means $\tfrac{1}{2}\cL_{\p_0}h=-\tfrac{1}{2}\cL_{(T_g^0)'_h}g_0+k'$. It follows that the first two equations in \eqref{lini1} are equivalent to prescribing the quantities $h_{\a\b}$ and $\p_t h_{\a\b}$ on $\{t=0\}$. Similarly, since $\big(F_*(N_g)\big)'_{(h,f)}=\p_t f+F_*\big((N_g)'_h\big)$ and the linearization of $N_g$ on the initial surfaces depends only on the information $h_{\a\b}=q_{\a\b}$ on $\{t=0\}$, the last two equations in \eqref{lini1} are equivalent to prescribing the quantities $f^\a,~\p_tf^\a$ on $\{t=0\}$.

As for the boundary equations, first notice that since $G'$ maps the boundary $\{x^1=0\}$ to $\{x_0^1=0\}$, we must have $f^1=0$ on $\{x^1=0\}$. Since the conformal class of a 2-dim metric $g_F^\intercal$ uniquely determines $(\det g_F^\intercal)^{-1/2}g_F^\intercal$, the third equation in \eqref{linb1} is equivalent to
$$(g_F^\intercal)'-\tfrac{1}{2}(\tr_{g_{0F}^\intercal}(g_F^\intercal)')g_{0F}=\sqrt{\det g_{0F}^\intercal}\big((\det \s)^{-1/2}\s\big)'.$$
Here $g_{0F}^\intercal$ denotes the 2-dim Riemannian metric induced by the pull back metric $(F_0^{-1})^*g_0$ on $\{t_0={\rm constant}\}\cap \{x_0^1=0\}$. The linearization of $g_F^\intercal$ is given by $(g_F^\intercal)'=\big((F_0^{-1})^*h\big)^\intercal+(\cL_{f}g_{0F})^\intercal$. Here $\cL_{f}$ denotes the Lie derivative with respect to the deformation $f$. Since $F_0$ is a constant linear map sending $\{t=\text{constant}\}$ to $\{t_0=\text{constant}\}$, we have $((F_0^{-1})^*h)^\intercal=(F_0^{-1})^*(h^\intercal)$. Thus the linearized equation above is equivalent to prescribing 
\be\label{linconformal}
h^\intercal-(\tr_{g_0}h^\intercal)g_0^\intercal=F_0^*\{\sqrt{\det g_{0F}^\intercal}\big((\det \s)^{-1/2}\s\big)'-(\cL_fg_{0F})+({\rm div}_{g_{0F}^\intercal}f)g_{0F}\} \quad \mbox{ on } \{x^1=0\}.
\ee 
Recall that the superscript $^\intercal$ denotes the reduction of a tensor field on $M_0$ (or $\cC_0$) to the tangent space of the 2-dim time level sets  $\Si_\tau$ of $\cC_0$.

The last equation in \eqref{linb1} can be expanded as
$$(F_0)_*(T'_h+T'_{f}+\nu'_h)+\cL_{f}(T_0+\nu_0)=\Theta'.$$ 
This equation is further equivalent to prescribing 
\be\label{lintn}
T'_h+\nu'_h=(F_0^{-1})_*[-(F_0)_*(T'_f)-\cL_f(T_0+\nu_0)+\Theta'] \ \ \mbox{ on } \{x^1=0\}.
\ee 
Linearization of the normal vector $T_g$ is given by $T'_{(h,f)}=T'_h+T'_f$, where $T'_h=\frac{1}{2}h_{00}\partial_0-h_{0i}\partial_i$ (cf.~\S 6.2) and $T'_f$ denotes the variation of $T$ caused by the variation of the time level sets $F^{-1}(S_{\tau})$.  In addition, since $\nu_g$ does not depend on the foliation, linearization of $\nu_g$ only depends on $h$, i.e. $\nu'_{(h,f)}=\nu'_h=h_{10}\partial_0-\frac{1}{2}h_{11}\partial_1-h_{1A}\partial_A$.  All the subscript indices here are with respect to the standard coordinates $x^{\a}$. 

The analysis above leads to the following result. 

\begin{proposition}
Let $g_0 = \eta$ be the standard Minkowski metric on ${\bf R}$ and $F_0$ be a constant linear map $F_0:({\bf R},g_0)\to ({\bf R}_0,g_{0R})$. The  frozen coefficient or blow-up linearization of the equations \eqref{main3}-\eqref{b3} at a flat solution $(g_0, F_0)$ { with respect to the deformation $I'=(q',k',E_0',E_1')~B'=(G',\s',\Theta')$ of local initial-boundary data}, near a corner point $p \in \Si$, may be written in the form 
\be \label{linmain11}
\begin{split}
\begin{cases}
\Box_{g_0} h_{\a\b}=0\\
\Box_{g_0} f^\a=0\\
\end{cases}
\quad\text{in }\mathbf R
\end{split}
\ee
with initial and boundary conditions 
\be \label{lini11}
\begin{split}
\begin{cases}
h_{\a\b}=q'_{\a\b},~\partial_{t} h_{\a\b}=u_{\a\b},~V'_h = 0\\
f^\a=(E'_0)^\a, \ \ \p_{t}f^\a = e^\a\\
\end{cases}
\quad\text{on }\{t=0\}
\end{split}
\ee
and
\be \label{linb11}
\begin{split}
\begin{cases}
V'_h=0\\
f^1= 0\\
f^\rho=(G')^\rho~(\rho=0,2,3)\\
h_{22} - h_{33} = c_2\\
h_{23} = c_3\\
\tfrac{1}{2}h_{00} + h_{10} = b_{0}\\
h_{01} + \tfrac{1}{2}h_{11} = b_{1}\\
h_{0A} + h_{1A} = b_{A}\ \ A=2,3\\
\end{cases}
\quad\text{on }\{x^1=0\}.
\end{split}
\ee
In the system above, the terms $u_{\a\b}, e^a$ in \eqref{lini11} are uniquely determined by the deformation $(q',k',E'_0,E'_1)$; the terms $c_A~(A=2,3)$ are computed based on the right-side term of \eqref{linconformal}, so they are determined by the deformation $\s'$ and $f$; and $b_\a~(\a=0,1,2,3)$ are determined by the right-side term of \eqref{lintn}, which depends on $\Theta'$ and $f$.
\end{proposition}

The compatibility or corner conditions for $(h, f)$ at $\Si$ are somewhat simpler to express in the system \eqref{linmain11}-\eqref{linb11}. 
Thus, the $C^0$ and $C^1$ compatibility conditions for $f$ are
\be \label{comp}
(G')^A= (E_0')^A~(A=2,3), \ \ \p_{t}(G')^\rho = e^\rho~(\rho=0,2,3)\ \ {\rm on} \ \ \{t=x^1=0\}.
\ee
At $2^{\rm nd}$ order and $3^{\rm rd}$ order, $\p_{t}^2 G' = \D E_0'$ and $\p_{t}^3G' = \D e$ on $\{t=x^1=0\}$ respectively and thus similarly at higher order. The $C^0$ compatibility conditions for $h$ are $b_0=\tfrac{1}{2}q'_{00}+q'_{10}$,$b_1=q'_{01}+\tfrac{1}{2}q'_{11}$ and $b_{A} = q'_{0a}+q'_{1A}~(A=2,3)$. The terms $c_A$ are determined by the trace-free part of $q'_{AB}$. The $t$-derivatives of $b_{\a}$ and $c_A$ are similarly determined by $u_{\a\b}$. One may compute the higher order compatibility conditions in a similar way.  

\begin{remark}\label{choice_T}
{\rm A detailed analysis of the linearized IBVP \eqref{linmain11}-\eqref{linb11} is given in the next section.
Here we note that in the last boundary condition of \eqref{b1} one cannot replace the unit normal $T_g$ of the time level sets in $M$ by the unit timelike normal $T_g^c$ of the time level sets in $\cC$. The linearization of $T_g^c$ at $(g_0,F_0)$ only involves $\frac{1}{2}h_{00}\p_0-h_{0A}\p_A$, i.e.~the term $h_{01}$ does not appear in the linearization; and this will be problematic for the energy estimates of the linearized system.}
\end{remark}
 
\medskip   

Next we provide a similar discussion for the system \eqref{main2}-\eqref{b2}. In analogy to \eqref{main3}-\eqref{b3}, we first modify the system with harmonic gauge in a local corner neighborhood $U$ as
\be \label{main4}
\begin{split}
\begin{cases}
\Ric_g+\delta_g^*V_g=0\\
\Box_g F + \Gamma_{g_R}(F)g(\nabla F, \nabla F) + F_*(V) = 0
\end{cases}
\quad\text{in }U
\end{split}
\ee
with initial conditions:
\be \label{i4}
\begin{split}
\begin{cases}
g = q, \ \ {\tfrac{1}{2}}\cL_{T^0_g}g = k, \ \ V_g = 0 \\
F=E_0,~F_*(N_g) = E_1\\
\end{cases}
\quad\text{on }S\cap U
\end{split}
\ee
and boundary conditions:
\be \label{b4}
\begin{split}
\begin{cases}
V_g=0\\
F = G\\
[g_F^\intercal]=[\s]\\
H_{g_F}=H\\
[F_*(T_g^c+\nu_g)]^T=\Theta_{\cC}
\end{cases}
\quad\text{on }\mathcal C\cap U.
\end{split}
\ee
Recall here from \eqref{prj} that $[F_*(T_g^c+\nu_g)]^T=F_*(T_g^c+\nu_g)-g_R(F_*(T_g^c+\nu_g),\nu_{g_R})\nu_{g_R}$, where $\nu_{g_R}$ denotes the outward unit normal to $\cC_0\subset(M_0,g_R)$. { Here the initial-boundary data $({\bf I, B_{\cC}})$ as defined in \eqref{I2}-\eqref{B2} and the local extended initial-boundary data $I=(q,k,E_0,E_1)$, $B_\cC=(G,[\s],H,\Theta_\cC)$ must satisfy the same conditions for $(\bf I,\bf B)$ and $(I,B)$ as discussed below \eqref{harmE}. For the compatibility conditions, $({\bf I, B_{\cC}})$ should satisfy the same equations \eqref{cib1}-\eqref{cib2} as for $(\bf I,\bf B)$. In addition, we choose $g_R$ on $M_0$ so that $S_0$ is orthogonal to $\cC_0$ along the corner $\Si_0$ and impose the extra condition}
\be \label{ccorner}
(E_0)_*(n_\g) \in \text{span}\{\nu_{g_R}\} \ \ \mbox{ on }\Si_0
\ee
together with the compatibility condition that
\be\label{cibc3}
\Theta_\cC=\ell E_1  \ \ \mbox{ on }\Si_0,
\ee
where $\ell$ is a function on $\Si_0$ such that $0<\ell<\tfrac{1}{\sqrt{2}}$.
In analogy to \eqref{cib4}, the local extended initial-boundary data $(I, B_{\cC})$ should satisfy 
\be\label{cibc4}
N_g(q)=G^{-1}_*(E_1),~q(N_g,n_\g)=\tfrac{\ell-1}{\sqrt{2\ell-\ell^2}}\ \ \mbox{on }\Si\cap U.
\ee
{ We refer to \S 6.3 for the detailed calculation of the compatibility conditions \eqref{cibc3}-\eqref{cibc4}. }

As before in \eqref{flat}, the flat pair $(g_{0}, F_0)$ in the standard corner neighborhood ${\bf R}$ (with flat Riemannian metric $g_{0R}$ on ${\bf R}_0$) solves the IBVP above with flat initial-boundary data: 
\begin{equation*}
\begin{split}
&I_0=\big(q_{\a\b}=(g_0)_{\a\b},~k_{\a\b}=0,~E_0=L|_{\{t=0\}},~E_1=L_*(N_{g_0})|_{\{t=0\}} \big),\\
&B_{0\cC}=\big(G=L|_{\{x^1=0\}},~[\s]=[(L_0^{-1})^*g_0^\intercal],~H= 0,~\Theta_{\cC} = L_*(T^c_{g_0}+\nu_{g_0})^T\big).
\end{split}
\end{equation*}
Since $(I,B)$ and $(I,B_{\cC})$ differ only by the last two terms in the boundary data, the same localization discussion as following \eqref{flat} holds here, except that in the rescaling process one defines 
\begin{equation*}
\tilde H(\tilde x_0)=\lambda H(\lambda \tilde x_0),\ \ \tilde\Theta_{\cC}^{\a}(\tilde x_0)=\Theta_{\cC}^{\a}(\lambda\tilde x_0)~\a=0,2,3.
\end{equation*}
Making $\l$ sufficiently small, $\tilde H$ is close to zero and $\tilde\Theta_{\cC}$ is close to $ L_*(T_{g_0}^c+\nu_{g_0})^T$, where $g_0$, $L$ are the same (limits) as in the rescaling discussion for the data $(I,B)$, and the projection operator $^T$ is with respect to the flat Riemannian metric $ g_{0R}$. Given these minor modifications, the proof of the validity of the blow-up or rescaling process for the system \eqref{main4}-\eqref{b4} is exactly the same as that for \eqref{main3}-\eqref{b3}.
 
Passing then to the linearization at the flat data $(g_0, F_0)$ as before, the linearization of the last equation in \eqref{b4} at $(g_0,F_0)$ is given by
\begin{equation*}
\begin{split}
F_{0_*}\big((T_g^c)'_h+\nu'_h\big)-g_{0R}\big(F_{0_*}((T_g^c)'_h+\nu'_h),\nu_{g_{0R}}\big)\nu_{g_{0R}}=b,
\end{split}
\end{equation*}
where $b$, similar as in \eqref{linb11}, only involves $\Theta_\cC',f$ in addition to $g_0,F_0,g_{0R}$.
{As before, the conditions on $E_0$ and $G$ imply that the push forward $F_{0_*}:\text{span}\{\partial_2,\partial_3\}\to\text{span}\{\p_{x_0^2},\p_{x_0^3}\}$, and $F_{0_*}(\p_0)\to\text{span}\{\p_{t_0},\p_{x_0^2},\p_{x_0^3}\}$. }Further, \eqref{ccorner} implies that $F_{0_*}(\p_1)\in\text{span}\{\nu_{g_{0R}}\}$. 
Now recall that $(T_g^c)'_h=\frac{1}{2}h_{00}\partial_0-h_{0A}\partial_A$ and $\nu'_h=h_{10}\partial_0-\frac{1}{2}h_{11}\partial_1-h_{1A}\partial_A$, cf.~also \S 6.2. Therefore, the linearization equation above is equivalent to 
\be\label{linTC}
\begin{split}
({\tfrac{1}{2}}h_{00}+h_{10})\partial_0-(h_{0A}+h_{1A})\partial_A=(F_{0_*})^{-1}(b).
\end{split}
\ee
Note that in contrast with the boundary condition \eqref{b1} with data ${\bf B}$ (cf.~Remark~\ref{choice_T}), in the last boundary condition in \eqref{b2} one may replace $T_g^c$ by the unit timelike normal $T_g$ of the time level sets in $M$. This is because by taking the tangential projection as above, one obtains the same linearized term for these different choices of the timelike normal.

The analysis above leads to the following analog of Proposition 2.5. 
\begin{proposition}
Let $g_0 = \eta$ be the standard Minkowski metric on ${\bf R}$ and $F_0$ be a constant linear map $F_0:({\bf R},g_0)\to ({\bf R}_0,g_{0R})$. The  frozen coefficient or blow-up linearization of the equations \eqref{main4}-\eqref{b4} at a flat solution $(g_0, F_0)$ { with respect to the deformation $I'=(q',k',E_0',E_1')~B_\cC'=(G',\s',H',\Theta_\cC')$ of local initial-boundary data}, near a corner point $p \in \Si$, may be written in the form \be \label{linmain2}
\begin{split}
\begin{cases}
\Box_{g_0} h_{\a\b}=0\\
\Box_{g_0} f^\a=0\\
\end{cases}
\quad\text{in }\mathbf R
\end{split}
\ee
with initial and boundary conditions 
\be \label{lini2}
\begin{split}
\begin{cases}
h_{\a\b}=q'_{\a\b},~\partial_{t} h_{\a\b}=u_{\a\b},~V'_h = 0\\
f^\a=(E_0')^\a, \ \ \p_{t}f^\a = e^\a
\end{cases}
\quad\text{on }\{t=0\}
\end{split}
\ee
and
\be \label{linb2}
\begin{split}
\begin{cases}
V'_h=0\\
f^1=0\\
f^\a = (G')^\a,~\a=0,2,3\\
h_{22} - h_{33} = c_2\\
h_{23} = c_3\\
(H_{g_F})'_{(h,f)} = H'\\
\tfrac{1}{2}h_{00} + h_{10} = b_{0}\\
h_{0A} + h_{1A} = b_{A}\ \ A=2,3\\
\end{cases}
\quad\text{on }\{x^1=0\}.
\end{split}
\ee
In the above, $u_{\a\b}$ and $e^\a$ are determined by $I'$; $c_A~(A=2,3)$ depend on $\s'$ and $f$ as given in \eqref{linconformal}; and $b_\a~(\a=0,2,3)$ depend on $\Theta_\cC'$ and $f$ as given in \eqref{linTC}.
\end{proposition}
The compatibility conditions are essentially the same as those in and following \eqref{comp}.  In the next section we discuss the precise 
definition of $H_{g_F}$ and its linearization.


\section{Analysis of { the Linearized systems}}

In this section, we derive the main $H^1$ and $H^s$ energy estimates for the linearized systems \eqref{linmain11}-\eqref{linb11} and \eqref{linmain2}-\eqref{linb2} at the core of the  well-posedness results to follow, cf.~Theorems \ref{exist1}, \ref{exist2} below. 

We work in the standard corner domain ${\bf R}=\{(t,x^1,x^2,x^3):t\geq 0, x^1\leq 0\}$ and set $\cC=\{x^1=0\}$, $S_\tau= \{t =\tau \mbox{ for a constant }\tau\}$, $\Si_\tau = S_\tau\cap \cC$ and let $\cC_\tau = \cup \{\Si_t: 0 \leq t \leq \tau\}$, $\cT_\tau = \cup \{S_t: 0\leq t \leq \tau\}$. As in \S 2, we assume the initial-boundary data in  \eqref{linmain11}-\eqref{linb11} and \eqref{linmain2}-\eqref{linb2} are of compact support. 

As a simple model for the systems \eqref{linmain11}-\eqref{linb11} and \eqref{linmain2}-\eqref{linb2}, consider a scalar function $u$ on ${\bf R}$ satisfying the inhomogeneous wave equation 
\be \label{wave}
\Box_{g_0}u = \f,
\ee
with initial conditions $u(0, \cdot) = u_0$, $\p_t u(0, \cdot) = u_1$ and boundary condition ${\mathsf B}(u) = b$ on $\cC$, all of compact support. The boundary operator ${\mathsf B}$, specified further below, is assumed to contain derivative operators of order at most $j$, with $j = 0$ or $j=1$. As usual, define the bulk and boundary energies by 
$$\cE_{S_\tau}(u) = \int_{S_\tau}u_t^2 + |du|^2 + u^2 d{\rm vol}_{S_\tau}\ \ {\rm and} \ \ \cE_{\cC_\tau}(u) = \int_{0}^{\tau}\int_{\Si_t}u_t^2 + |du|^2 + u^2 d{\rm vol}_{\Si_t}dt,$$ 
where $du$ is the full collection of spatial derivatives $\p_i u$, $i = 1,2,3$ and the integration is with respect to the volume forms induced on $S_\tau$, $\Si_t$ by the spacetime metric $g_0$. 

For $\O = S_\tau, \Si_\tau$ or $\cC_\tau$, let $H^s(\O)$ denote the Sobolev space of functions with weak derivatives up to order $s$ in $L^2(\O)$, $s \in \bR^+$. For notational convenience, we let $\bar H^s$ denote the analogous norm consisting of all space-time derivatives, (not just those tangent to $\O$). In this notation, 
$$\cE_{S_t}(u) = ||u||_{\bar H^1(S_t)}^2, \ \ {\rm and} \ \ \cE_{\cC_t}(u) = ||u||_{\bar H^1(\cC_t)}^2.$$

As is well-known, the well-posedness of \eqref{wave} and similar more complicated systems of wave equations rests on the existence of the main $H^1$ energy estimate
\be \label{Eu1}
\cE_{S_t}(u) + c\cE_{\cC_t}(u) \leq C[\cE_{S_0}(u) + ||\f||_{L^2(\cT_t)}^2 + ||b||_{H^{1-j}(\cC_t)}^2],
\ee
for constants $c, C > 0$ independent of $u$ and $b$. Similarly, one requires higher order energy estimates of the form 
\be \label{Eus}
||u||_{\bar H^s(S_t)}^2 + c||u||_{\bar H^s(\cC_t)}^2 \leq C[||u||_{\bar H^s(S_0)}^2 + ||\f||_{\bar H^{s-1}(\cT_t)}^2 + 
||b||_{H^{s-j}(\cC_t)}^2].
\ee
These estimates require that $u$ is a smooth, (or at least sufficiently smooth) solution of \eqref{wave}. It will always be assumed $s > \frac{n}{2} + 1 = \frac{5}{2}$, so that by Sobolev embedding $C^{1,\a}\subset H^s$ in dimension 3. It is important to observe that the last term in \eqref{Eu1} or \eqref{Eus} involves only derivatives of $b$ tangent to $\cC$.

For completeness, these energy estimates for solutions $u$ of \eqref{wave} are derived in Appendix \S 6.1 for Sommerfeld (${\mathsf B} (u)= \p_t u+ \p_{x^1} u$) and Dirichlet (${\mathsf B}(u)=u$) boundary conditions, which suffice for our purposes. It is well-known, cf.~\cite{Benzoni-Serre:2007} for example, that the IBVP for \eqref{wave} is well-posed with respect to either of these boundary conditions. 

Returning to the linearized systems \eqref{linmain11}-\eqref{linb11} and \eqref{linmain2}-\eqref{linb2}, throughout the following, we assume the initial data $(q',k')$ for $g'$ are in $H^s(S)\times H^{s-1}(S)$ and the boundary data $(b_{\a}, c_{A})$ are in $H^{s}(\cC)$. Similarly, we assume the initial data $(E_0', E_1')$ for $f$ are in $H^{s+1}(S)\times H^s(S)$ while the boundary data $G'$ for $f$ are in $H^{s+1}(\cC)$. In addition, we assume the $C^{s-1}\times C^s$ compatibility conditions hold for $(I',B')$ at the corner $\Si$. 

\medskip 

We first prove the energy estimates for $f$. 

\begin{proposition}\label{energyF}
Under the assumptions on the initial-boundary data above, one has an $H^{s+1}$ energy estimate for $f$ in \eqref{linmain11}-\eqref{linb11} and \eqref{linmain2}-\eqref{linb2}. Thus 
$$||f||_{\bar H^{s+1}(S_t)}^2 + c||f||_{\bar H^{s+1}(\cC_t)}^2 \leq C[||f||_{\bar H^{s+1}(S_0)}^2 + ||G'||_{H^{s+1}(\cC_t)}^2].$$
\end{proposition}

\begin{proof} As already noted in \S 2, the system for $f$ in \eqref{linmain11}-\eqref{linb11} and \eqref{linmain2}-\eqref{linb2} decouples from the $h$-system. In both systems one has  
\be \label{f}
\begin{split}
&\Box_{g_0}f=0\text{ in }\mathbf R\\
&f=E_0', \ \ \partial_{t}f=e,\text{ on }\{t=0\}\\
&f=G' \text{ on }\{x^1_0=0\}.
\end{split}
\ee
The system \eqref{f} is an uncoupled system of wave equations for $f$ with (inhomogeneous) Dirichlet boundary conditions. It is well-known that such systems admit $H^1$ energy estimates as in \eqref{Eu1} and higher order energy estimates \eqref{Eus} given the $C^s$ compatibility conditions; cf.~again \cite{Benzoni-Serre:2007} or \S 6.1. 

\end{proof}

Next we turn to energy estimates for the blow-up linearization \eqref{linmain11}-\eqref{linb11} of the system \eqref{main3}-\eqref{b3}. { Notice that with the energy of $f$ well controlled, the $H^s$ norms of terms $b_\a$ and $c_A$ in these linear systems are bounded by the data in $B'$ (or $B_\cC'$).}

\begin{proposition}\label{energy1}
For the linear system \eqref{linmain11}-\eqref{linb11} one has an $H^s$ energy estimate
$$||h||_{\bar H^s(S_t)}^2 + c||h||_{\bar H^s(\cC_t)}^2 \leq C[||h||_{\bar H^s(S_0)}^2 + ||\underline{b}||_{H^s(\cC_t)}^2 + ||\underline{c}||_{H^s(\cC_t)}^2],$$
where $\underline{b} = \{b_{\a}\}$, $\a = 0,1,2,3$, $\underline{c} = \{c_A\}$, $A = 2,3$. 

\end{proposition}

\begin{proof}

We begin by analyzing the gauge boundary conditions $V'_h = 0$ on the boundary $\cC = \{x^1 = 0\}$, which have the form: 
\be \label{gaugelin}
\begin{split} 
&(V'_h)_0 = -\partial_0h_{00}+\partial_1h_{01}+\partial_Ah_{0A}-{\tfrac{1}{2}}\partial_0(trh) = 0\\
&(V'_h)_1 = -\partial_0h_{01}+\partial_1h_{11}+\partial_Ah_{1A}-{\tfrac{1}{2}}\partial_1(trh) = 0\\
&(V'_h)_A = -\partial_0h_{0A}+\partial_1h_{1A}+\partial_Bh_{BA}-{\tfrac{1}{2}}\partial_A(trh) = 0.
\end{split}
\ee
We recall that $A,B= 2,3$ and the Einstein summation convention is used. 
Let $\tau_h= \frac{1}{2}(h_{22}+h_{33})$. Since $\tr h = -h_{00}+h_{11}+2\tau_h$, this gives 
\be \label{eqn0}
\begin{split}
&-{\tfrac{1}{2}}\partial_0(h_{00}+h_{11})+\partial_1h_{01}-\partial_0\tau_h+\partial_Ah_{0A} = 0\\
&\ \ \ {\tfrac{1}{2}}\partial_1(h_{00}+h_{11})-\partial_0h_{01}-\partial_1\tau_h+\partial_Ah_{1A} = 0\\
&-\partial_0h_{0A}+\partial_1h_{1A}+{\tfrac{1}{2}}\partial_A(h_{00}-h_{11} - 2\tau_h) + \p_B h_{AB} = 0.
\end{split}
\ee
Simple modification of these equations gives
\be \label{eqn0000}
\begin{split}
&-(\partial_0+\tfrac{1}{2}\partial_1)h_{00}-\partial_0\tau_h-{\tfrac{1}{2}}\partial_0(h_{11}-h_{00})+\partial_1(h_{01}+\tfrac{1}{2}h_{00})+\partial_Ah_{0A} = 0\\
&\ \ \ (\tfrac{1}{2}\partial_0+\partial_1)h_{00}-\partial_1\tau_h+{\tfrac{1}{2}}\partial_1(h_{11}-h_{00})-\partial_0(h_{01}+\tfrac{1}{2}h_{00})+\partial_Ah_{1A} = 0\\
&-\partial_0h_{0A}+\partial_1h_{1A}+{\tfrac{1}{2}}\partial_A(h_{00}-h_{11} - 2\tau_h) + \p_B h_{AB} = 0.
\end{split}
\ee
which further leads easily to the following system: 
\begin{align}
\label{eqn1}
&-(\partial_0+\tfrac{1}{2}\partial_1)h_{00}-\partial_0\tau_h-{\tfrac{1}{2}}\partial_0(h_{11}-h_{00})+\partial_1(h_{01}+\tfrac{1}{2}h_{00})+\partial_Ah_{0A} = 0\\
\label{eqn2}
&\ \ \ (\tfrac{1}{2}\partial_0+\partial_1)h_{00}-\partial_1\tau_h+{\tfrac{1}{2}}\partial_1(h_{11}-h_{00})-\partial_0(h_{01}+\tfrac{1}{2}h_{00})+\partial_Ah_{1A} = 0\\
\label{eqn3}
&-(\partial_0+\partial_1)(h_{0A}-h_{1A})-(\partial_0-\partial_1)(h_{0A}+h_{1A})+\partial_A(h_{00}-h_{11}) -2\p_A \tau_h + 2\p_B h_{AB} = 0.
\end{align}
In light of \eqref{eqn1}-\eqref{eqn3} and the well-known existence of energy estimates for Sommerfeld and Dirichlet boundary conditions discussed in \S 6.1, it is then natural to impose Dirichlet boundary condition on the terms $h_{11}-h_{00}$, $\tfrac{1}{2}h_{00}+h_{01}$ and $h_{0A}+h_{1A}$. Note that these terms are all included in the Dirichlet boundary conditions in \eqref{linb11}. Namely, the full set of boundary conditions in \eqref{linb11} are 
\begin{align}
\label{conf1}
&h_{22}-h_{33}=c_2\\
\label{conf2}
&h_{23}=c_3\\
\label{eqn4}
&{\tfrac{1}{2}}h_{00}+h_{01} = b_0\\
\label{eqn5}
&{\tfrac{1}{2}}h_{11}+h_{01} = b_1\\
\label{eqn6}
&h_{0A}+h_{1A} = b_A.
\end{align}
Observe that \eqref{eqn4}$-$\eqref{eqn5} gives $h_{00}-h_{11} =2( b_0 - b_1)$. All of these linearly combined components of $h$ i.e.~$h_{22}-h_{33}$, $h_{23}$, $h_{00}-h_{11}$, $\tfrac{1}{2}h_{00}+h_{01}$ and $h_{0A}+h_{1A}$, satisfy the wave equation \eqref{wave} with Dirichlet boundary conditions. Hence the $H^s$ energy estimate \eqref{Eus} holds for them.

Next in \eqref{eqn3}, fix the index $A$, say $A=2$. Then the last two terms are $-2\p_A \tau_h + 2\p_B h_{AB}=-2\p_2 \tau_h + 2\p_2 h_{22}+2\p_3 h_{23}=-\p_2 (h_{22}+h_{33}) + 2\p_2 h_{22}+2\p_3 h_{23}=\p_2(h_{22}-h_{33})+2\p_3h_{23}$.  Thus by \eqref{conf1}-\eqref{conf2} this is controlled in $H^{s-1}$. Based on \eqref{eqn4}-\eqref{eqn5}, one also has $H^{s-1}$ control of the term $\partial_A(h_{00}-h_{11})$ in \eqref{eqn3}. {Further, as discussed in the Appendix \S 6.1, control of the Dirichlet boundary value gives control of the Neumann (normal derivative) boundary value; this is the boundedness of the Dirichlet-to-Neumann map. Thus \eqref{eqn6} gives $H^{s-1}$ control of the second term $-(\partial_0-\partial_1)(h_{0A}+h_{1A})$ in \eqref{eqn3}.} It then follows that $(\p_0+\p_1)(h_{02}-h_{12})$ is controlled in $H^{s-1}$ on $\cC$. This is a Sommerfeld boundary operator and since $(h_{02}-h_{12})$ is a solution of the wave equation, $\Box_{g_{0}}(h_{02}-h_{12}) = 0$ (with $\bar H^s$ initial conditions), this gives $H^s$ control on $h_{02} - h_{12}$. In addition we already have $H^s$ energy control of $h_{02}+h_{12}$. Thus we obtain $H^s$ energy estimates for both $h_{02}$ and $h_{12}$. The same argument also applies to the case $A=3$, and we can obtain $H^s$ energy estimates for $h_{03}$ and $h_{13}$.

Now according to \eqref{eqn1}, \eqref{eqn2}, the $H^s$ Dirichlet control on $h_{0A}, h_{1A}$, $\tfrac{1}{2}h_{00}+h_{01}$ and $h_{00}-h_{11}$ at the boundary $\cC$ gives $H^{s-1}$ control of $u=(\partial_0+\tfrac{1}{2}\partial_1)h_{00}+\partial_0\tau_h$ and $v=(\tfrac{1}{2}\partial_0+\partial_1)h_{00}-\partial_1\tau_h$. Consider the combination $u-(2+\sqrt 3)v$, which is then also controlled in $H^{s-1}$. Simple computation gives
$$u-(2+\sqrt 3)v=(\p_0+(2+\sqrt 3)\p_1)(-\tfrac{\sqrt 3}{2}h_{00}+\tau_h),$$
which is a Sommerfeld boundary operator on $w=-\tfrac{\sqrt 3}{2}h_{00}+\tau_h$. Thus the $H^s$ energy of $w$ is controlled since $w$ satisfies the wave equation $\Box_{g_0}w=0$. Furthermore $u$ can be rewritten as
$u=(\partial_0+\tfrac{1}{2}\partial_1)h_{00}+\partial_0(\tfrac{\sqrt 3}{2}h_{00}-\tfrac{\sqrt 3}{2}h_{00}+\tau_h)=[(1+\tfrac{\sqrt 3}{2})\p_0+\tfrac{1}{2}\partial_1]h_{00}+\p_0w$ which thus gives a bound of the Sommerfeld operator $(1+\tfrac{\sqrt 3}{2})\p_0+\tfrac{1}{2}\partial_1$ acting on $h_{00}$. Thus we obtain $H^s$ energy estimates for $h_{00}$, which further yields $H^s$ energy estimates for $\tau_h$ via \eqref{eqn1}. Combined with \eqref{conf1}, \eqref{eqn4}, \eqref{eqn5}, this also gives energy estimates for $h_{22},h_{33}$, $h_{11}$ and $h_{01}$.\footnote{{The proof of Proposition \ref{energy1}, as well as that of Proposition \ref{energy2} below, hold equally well for blow-up linearization of the vacuum Einstein equations with a cosmological constant $\Lambda$, where \eqref{linmain11} is replaced by $\Box_{g_0} h_{\a\b} = c g_0$ for some non-zero constant $c$. }  }

\end{proof}

The method of proof of Proposition \ref{energy1} shows that the harmonic gauge condition $V_g= 0$ on $\cC$ determines a natural choice of Dirichlet-type boundary data \eqref{eqn4}-\eqref{eqn6}, at least given the choice of boundary condition on the conformal class $[g_F^\intercal]$. The method of proof also has an upper-triangular character, similar to the upper-triangular form or bootstrap method introduced and employed in \cite{KRSW:2009}, \cite{KRSW:2007}, cf.~also \cite{Kreiss-Winicour:2014}.

\medskip 

Next we consider boundary conditions more intrinsic to the boundary $\cC$. We will keep the boundary conditions \eqref{eqn4},\eqref{eqn6}, but drop the condition \eqref{eqn5}; instead we seek a replacement for \eqref{eqn5} with a quantity more intrinsic or geometric to the boundary $\cC$. We first present below a general discussion of this situation. The result of this analysis is then summarized in Proposition \ref{energy2} below. 

In the following, we denote by $O_j$ a boundary term which has been controlled in $H^{s-j}(\cC)$ by preceding arguments and let $O = O_0$. Thus from \eqref{conf1} or \eqref{linb2}, we have $h_{22}=h_{33}=\tau_h+O$, while from \eqref{linb2}, $h_{01}=-\frac{1}{2}h_{00}+O$ and $h_{1A}=-h_{0A}+O$. Applying these replacements in \eqref{gaugelin}, we obtain 
\begin{align}
\label{eqn7}
&-{\tfrac{1}{2}}\partial_0h_{11}-{\tfrac{1}{2}}(\partial_0+\partial_1)h_{00}-\partial_0\tau_h+X=O\\
\label{eqn8}
&{\tfrac{1}{2}}\partial_1h_{11}+{\tfrac{1}{2}}(\partial_0+\partial_1)h_{00}-\partial_1\tau_h-X=O\\
\label{eqn9}
&-(\partial_0+\partial_1)h_{0A}+{\tfrac{1}{2}}\partial_A(h_{00}-h_{11})=O,
\end{align}
where $X=\partial_Ah_{0A}$.

We first seek an equation involving only the terms $h_{00}$ and $h_{11}$. To do this, we use Hamiltonian constraint \eqref{Ham1}, \eqref{Ham2} on both the timelike boundary $\cC=\{x^1=0\}$ as well as the spacelike hypersurfaces $S_t$. The linearization of these equations is given \S 6.2. From \eqref{Hamlin1}, on $\cC$ one has 
\be \label{hamlin1}
\begin{split}
(\partial_0\partial_0-\partial_1\partial_1)h_{00}+(\partial_0\partial_0+\partial_1\partial_1)\tau_h-2\partial_0X=O_2,
\end{split}
\ee
while on the hypersurfaces $S_t$, from \eqref{Hamlin2} one has 
\be \label{hamlin2}
\begin{split}
(\partial_0\partial_0-\partial_1\partial_1)h_{11}+(\partial_1\partial_1+\partial_0\partial_0)\tau_h+2\partial_1X=O_2.
\end{split}
\ee
Taking the difference \eqref{hamlin1}$-$\eqref{hamlin2} gives  
\be \label{diff1}
(\partial_0\partial_0-\partial_1\partial_1)(h_{00}-h_{11})-2(\partial_0+\partial_1)X=O_2,
\ee
which is equivalent to 
\be \label{diff2}
(\partial_0+\partial_1)[(\partial_0-\partial_1)(h_{00}-h_{11})-2X]=O_2. 
\ee
This is a Sommerfeld boundary condition on $[(\partial_0-\partial_1)(h_{00}-h_{11})-2X]$ and thus as in the proof of Proposition \ref{energy1}, we can obtain
\begin{equation}
\begin{split}
(\partial_0-\partial_1)(h_{00}-h_{11})-2X=O_1.
\end{split}
\end{equation}
Now taking $\partial_1\eqref{eqn7}-\partial_0\eqref{eqn8}$ yields
\be \label{diff3}
\begin{split}
-\partial_0\partial_1h_{11}-{\tfrac{1}{2}}(\partial_0+\partial_1)^2h_{00}+(\partial_1+\partial_0)X=O_2.
\end{split}
\end{equation}
Taking then $2\times$\eqref{diff3}+\eqref{diff1} gives 
\begin{equation*}
\begin{split}
-2\partial_0\partial_1h_{11}-(\partial_0+\partial_1)^2h_{00}+(\partial_0\partial_0-\partial_1\partial_1)(h_{00}-h_{11})=O_2.
\end{split}
\end{equation*}
which can be simplified as
\be \label{double}
\begin{split}
(-\partial_0\partial_0+\partial_1\partial_1-2\partial_0\partial_1)h_{11}-2\partial_1(\partial_0+\partial_1)h_{00}=O_2.
\end{split}
\end{equation}
Factorization gives $(-\partial_0\partial_0+\partial_1\partial_1-2\partial_0\partial_1)h_{11}=(\partial_1+(\sqrt 2-1)\partial_0)(\partial_1-(\sqrt 2+1)\partial_0)h_{11}$. The first factor, again of Sommerfeld type, leads to suitable energy estimates; the second factor however does not. Thus we seek a remaining boundary condition of the form 
\be \label{h1}
(\partial_0+\partial_1)h_{00}-\a\partial_0h_{11}-\b\partial_1h_{11}=O_1
\ee
for some real numbers $\a,\b$.
Based on the Dirichlet-to-Neumann estimate as discussed in the proof of Proposition \ref{energy1}, taking $\p_1$ of \eqref{h1} yields 
\be
\partial_1(\partial_0+\partial_1)h_{00}-\a\partial_0\partial_1h_{11}-\b\partial_1\partial_1h_{11}=O_2,
\ee
and adding this to \eqref{double} gives $(-\partial_0\partial_0+\partial_1\partial_1-2\partial_0\partial_1)h_{11}-2\a\partial_0\partial_1h_{11}-
2\b\partial_1\partial_1h_{11}=O_2$, i.e.
\be \label{double2}
[\partial_0\partial_0+2(\a+1)\partial_0\partial_1+(2\b-1)\partial_1\partial_1]h_{11}=O_2.
\end{equation}
For this to be a well-posed (Sommerfeld-type) boundary condition, one must have 
\be \label{AB}
\begin{split}
~\a+1\geq 0,~2\b-1\geq 0,~(\a+1)^2\geq 2\b-1.
\end{split}
\end{equation}
If the inequalities above are satisfied, \eqref{double2} can be taken as a``double Sommerfeld" type boundary condition on $h_{11}$. 

It remains to check what mean curvature quantity $H_{g_F}$ can lead to \eqref{h1}. Note that since we already have energy estimate for $f$ by Proposition \ref{energyF}, in the following we only consider variation of various mean curvature terms with respect to the deformation $h$. 
Let $K_{g|S_t}$ be the second fundamental form of the hypersurface $S_t$ in the ambient manifold $({\bf R},g)$, and $K_{g|\cC}$ be the second fundamental form of the timelike boundary $\{x^1=0\}\subset({\mathbf R},g)$. Let $\tr_{S_t}K_{g|S_t}$ denote the full trace of $K_{g|S_t}$ on the $t$-hypersurface $S_t$ and $\tr_{\Si_t}K_{g|S_t}$ the restricted trace of $K_{g|S_t}$ on $\Si_t$. Similarly, let $\tr_{\mathcal C}K_{g|\cC}$ be the full trace of $K_{g|\cC}$ on the timelike boundary, with $\tr_{\Si_t}K_{g|\cC}$ the restricted trace. All these trace operators are with respect to the metrics on $S_t$ and $\Si_t$ that are induced from $g$. The linearization of these terms is given by (cf. \S 6.2):
\begin{align}
\label{trK}
&2(\tr_{S_t}K_{g|S_t})'_h=\partial_0(h_{11}+h_{AA})-2\partial_1h_{01}-2\partial_Ah_{0A}=\partial_0h_{11}+\partial_1h_{00}+2\partial_0\tau_h-2X + O_1\\
\label{trtK}
&2(\tr_{\Si_t}K_{g|S_t})'_h=\partial_0h_{AA}-2\partial_Ah_{0A}=2(\partial_0\tau_h-X) + O_1\\
\label{trA}
&2(\tr_{\mathcal C}K_{g|\cC})'_h=\partial_1(-h_{00}+h_{AA})+2\partial_0h_{10}-2\partial_Ah_{1A}=-(\partial_1+\partial_0)h_{00}+2\partial_1\tau_h+2X + O_1\\
\label{trtA}
&2(\tr_{\Si_t}K_{g|\cC})'_h=\partial_1h_{AA}-2\partial_Ah_{1A}=2(\partial_1\tau_h+X) + O_1.
\end{align}
Here in \eqref{trK} and \eqref{trA} we have used the control on $\frac{1}{2}h_{00} + h_{01}$ as well as the control given by the Dirichlet-to-Neumann map, as in the proof of Proposition \ref{energy1} above. Substituting the relations \eqref{eqn7}-\eqref{eqn9} into these equations one easily obtains
\be \label{hm}
\begin{split}
&\p_0 h_{00} = -2(\tr_{S_t}K_{g|S_t})'_h,\\
&\p_1 h_{00} = 2(\tr_{\Si_t}K_{g|\cC})'_h - 2(\tr_{\cC}K_{g|\cC})'_h + 2(\tr_{S_t}K_{g|S_t})'_h,\\
&\p_0 h_{11} = - 2(\tr_{\Si_t}K_{g|S_t})'_h -2(\tr_{\Si_t}K_{g|\cC})'_h + 2(\tr_{\cC}K_{g|\cC})'_h,\\
&\p_1 h_{11} = 2(\tr_{\cC}K_{g|\cC})'_h.\\
\end{split}
\ee
Substituting these into \eqref{h1} transforms \eqref{h1}, after simple manipulations, into  
$$\a (\tr_{\Si_t}K_{g|S_t})'_h - (\a+\b+1)(\tr_{\cC}K_{g|\cC})'_h + (\a+1)(\tr_{\Si_t}K_{g|\cC})'_h = O_1.$$
Thus in the nonlinear system \eqref{main4}-\eqref{b4} we can set 
\be \label{h2n}
H_{g_F} = \a \tr_{\Si_\tau}K_{g_F|S_\tau} - (\a+\b+1)\tr_{\cC}K_{g_F|\cC} + (\a+1)\tr_{\Si_\tau}K_{g_F|\cC},
\ee
and require that $\a,\b$ satisfy \eqref{AB}. 

This leads to the following analog of Proposition \ref{energy1}.  

\begin{proposition}\label{energy2}
For the gauged system \eqref{main4}-\eqref{b4} where $H_{g_F}$ is given by \eqref{h2n} with constants $\a,\b$ satisfying \eqref{AB}, its blow-up linearization \eqref{linmain2}-\eqref{linb2} admits an $H^s$ energy estimate 
$$|||h||_{\bar H^s(B_t)}^2 + c||h||_{\bar H^s(\cC_t)}^2 \leq C[||h||_{\bar H^s(B_0)}^2 +||H'||_{H^{s-1}(\cC_t)}^2+ ||\underline{b}||_{H^s(\cC_t)}^2 + 
||\underline{c}||_{H^s(\cC_t)}^2],$$
where $\underline{b} = \{b_{\a}\}$, $\a = 0,2,3$, $\underline{c} = \{c_A\}$, $A = 2,3$.  
\end{proposition}

\begin{proof} The proof is the same as that of Proposition \ref{energy1}. Namely, if $\a,\b$ satisfy \eqref{AB}, then by \eqref{double2} one obtains an $H^s$ energy estimate for $h_{11}$. Via \eqref{h1} and the Dirichlet-to-Neumann estimate, this gives an $H^s$ energy estimate for $h_{00}$. Since the Dirichlet condition on $\frac{1}{2}h_{00} + h_{01}$ gives an $H^s$ energy estimate for this term, one has the $H^s$ energy estimate for $h_{01}$. Now equation \eqref{eqn9} yields $H^s$ control of $h_{0A}$ for $A=2,3$ and equation \eqref{eqn7} yields $H^s$ control of $\tau_h$. The Dirichlet boundary conditions give $H^s$ control on the remaining components of $h = h_{\a\b}$. 
\end{proof}

There are many other expressions for $H_{g_F}$ besides \eqref{h2n} for which Proposition \ref{energy2} remains valid; this arises from the fact that there are numerous other variants of the algebraic manipulations in \eqref{double}-\eqref{h1}. Similarly, other expressions for $H_{g_F}$ preserving the validity of Proposition \ref{energy2} may be obtained by changing the boundary condition $\frac{1}{2}h_{00} + h_{01}$ to $\l h_{00} + \mu h_{01}$, for arbitrary smooth $\l, \mu > 0$. We will not pursue this further in general here however. 

It is worth noting that when $\a=0$ and $\b=1$ in \eqref{h2n}, $H_{g_F}$ is the mean curvature boundary condition $L$ in \cite{Kreiss-Winicour:2014}. However, the method of proof of \cite{Kreiss-Winicour:2014}, relying on estimates with pure Neumann boundary data, is rather different than the proof above. 

Finally, it would be interesting to know if (with a suitable choice of $\l, \mu$ for instance), one can choose $H_{g_F} = \tr_{\cC}K_{g_F|\cC}$, the mean curvature of the boundary $\cC$, as in \cite{Friedrich-Nagy:1999}. 


\section{Local well-posedness and geometric uniqueness for the expanded IBVP} 

In this section, we first use the results above to prove well-posedness of the gauged IBVP's in \eqref{main3}-\eqref{b3} and \eqref{main4}-\eqref{b4}. Following this, we turn to ungauged systems \eqref{main}-\eqref{b1} and \eqref{main2}-\eqref{b2} and the issue of local geometric uniqueness. 

For the following results, recall that $\Upsilon=S\cup \cC$ denotes the initial-boundary surface of the ST-spacetime $M$; let $W$ be a local corner neighborhood around $p \in \Si$ in $M$ and (given a metric $g$) let {$\cD^+(\Upsilon\cap W,g)$ denote the future domain} of dependence of the initial-boundary surface $\Upsilon\cap W$ in $(W, g)$. 

\begin{theorem} (Local well-posedness I)\label{exist1}
The IBVP of the gauged system \eqref{main3}-\eqref{b3} with local extended initial data $I$ as in \eqref{i3} and boundary data $B$ as in \eqref{b3} is locally well-posed in 
$$H^{s} \times H^{s + 1},$$ 
for $s \geq 4$, $s \in \bN^+$. More precisely, suppose in a neighborhood $W$ of a corner point $p$, equipped with a standard corner chart $\chi$, one is given gauged $g$-initial data $(q,k) \in H^{s+\frac{1}{2}}(S\cap W)\times H^{s-\frac{1}{2}}(S\cap W)$ satisfying the constraint equations \eqref{constraint}, and $F$-initial data $(E_0,E_1) \in H^{s+\frac{3}{2}}(S\cap W)\times H^{s+\frac{1}{2}}(S\cap W)$ together with boundary data $(G, [\s], \Theta) \in H^{s+\frac{3}{2}}(\cC\cap W)\times H^{s+\frac{1}{2}}(\cC_0\cap G(W))\times H^{s + \frac{1}{2}}(\cC_0\cap G(W))$ as in \eqref{i3}-\eqref{b3} and satisfying the $C^{s-1}$ compatibility conditions at $\Si \cap W$. 

Then there exists a triple $(U, g, F)$ with a local corner neighborhood $U \subset W \subset M$, $p \in U$, satisfying the following properties: 
\begin{enumerate}
\item The pair $(g, F)$ is a solution of the system \eqref{main3} with  
$$(g,F) \in H^{s}(U)\times H^{s+1}(U).$$
The trace of $(g, F)$ on $\Upsilon\cap U$ is in $H^s(\Upsilon\cap U)\times H^{s+1}(\Upsilon\cap U)$ and realizes the initial and boundary conditions \eqref{i3}-\eqref{b3}. 
\item $U = \cD^+(\Upsilon\cap U,g)$. 
\item On the domain $U$, the solution $(g, F)$ is unique. 
\item The solution $(U, g, F)$ on the fixed domain depends continuously on the initial-boundary data $(I,B)$. 
\end{enumerate}

\end{theorem} 

\begin{proof} Proposition \ref{energy1} gives the existence of strong or boundary stable energy estimates for the frozen coefficient system, i.e.~the linearization of the system at a standard flat configuration. The proof of well-posedness then follows from the general theory of quasi-linear initial-boundary value problems. 

In more detail, consider the linearization of the system \eqref{main3}-\eqref{b3} at any smooth background configuration $(g, F)$. The bulk equations are then a system of linear wave equations, coupled only at lower order. As in Proposition \ref{energy1}, the boundary conditions are the 4 Dirichlet boundary conditions for $F$, 6 Dirichlet boundary conditions for $g$ and 4 gauge boundary conditions $V'_h = 0$, all satisfying the compatibility conditions. Given the existence of energy estimates for the frozen (constant) coefficient system, one obtains existence of energy estimates for the general linearized system by localization in a sufficiently small neighborhood of any corner point $p\in \Si$. This uses a partition of unity, giving local data of compact support, and rescaling, as discussed in \S 2. We refer for example to \cite[Theorem 9.1]{Benzoni-Serre:2007}, for details of this extension of energy estimates for the constant coefficient system to the general linear system. It follows that the general linearization of the system \eqref{main3}-\eqref{b3} at any given background has boundary stable energy estimates. 

The frozen coefficient system admits a reduction to a first order symmetric hyperbolic system, (i.e.~ there exists a Friedrichs symmetrizer) with non-characteristic boundary. The strong or boundary stable $H^s$ energy estimates are equivalent to the statement that the boundary conditions \eqref{b3} are strictly maximally dissipative, cf.~\cite{Benzoni-Serre:2007}, \cite{Sarbach-Tiglio:2012}. It then follows from \cite[Theorem 9.16]{Benzoni-Serre:2007}, that the system \eqref{main3}-\eqref{b3}, linearized at any smooth background $(g, F)$, is well-posed in $C^{r}([0,t], H^{s-r}(S)) \times C^{r}([0,t], H^{s+1-r}(S))$, $0 \leq r \leq s$. 

Finally, by a technically involved argument, the quasi-linear system is proved to be well-posed by a standard iteration or contraction mapping principle applied to a sequence of solutions of the linearized system, cf.~\cite{Rauch-Massey:1974}, \cite{Mokrane:1987}, \cite{Benzoni-Serre:2007}; the particular formulation given in Theorem \ref{exist1} is an application of \cite[Theorem 11.1]{Benzoni-Serre:2007}. 
\end{proof}

\medskip

By applying the same proof as that of Theorem \ref{exist1}, and using Proposition \ref{energy2} in place of Proposition \ref{energy1}, one proves the well-posedness of the system \eqref{main4}-\eqref{b4}. This gives the following result. 
\begin{theorem} (Local well-posedness II) \label{exist2}
The IBVP of the gauged system \eqref{main4}-\eqref{b4} with local extended initial data $I$ as in \eqref{i4} and boundary data $B_{\cC}$ as in \eqref{b4} is locally well-posed in 
$$H^{s} \times H^{s + 1},$$ 
for $s \geq 4$, $s \in \bN^+$.  More precisely, suppose in a neighborhood $W$ of a corner point $p$, equipped with a standard corner chart $\chi$, one is given gauged $g$-initial data $(q,k) \in H^{s+\frac{1}{2}}(S\cap W)\times H^{s-\frac{1}{2}}(S\cap W)$ satisfying the constraint equations \eqref{constraint}, and $F$-initial data $(E_0,E_1) \in H^{s+\frac{3}{2}}(S\cap W)\times H^{s+\frac{1}{2}}(S\cap W)$ as in \eqref{i4} together with boundary data $(G, [\s],H, \Theta_{\cC}) \in H^{s+\frac{3}{2}}(\cC\cap W)\times H^{s+\frac{1}{2}}(\cC_0\cap G(W))\times H^{s-\frac{1}{2}}(\cC_0\cap G(W))\times H^{s + \frac{1}{2}}(\cC_0\cap G(W))$ as in \eqref{b4}, and satisfying the $C^{s-1}$ compatibility conditions at $\Si \cap V$. 

Then there exists a triple $(U, g, F)$ with a local corner neighborhood $U \subset W \subset M$, $p \in U$, satisfying the following properties: 
\begin{enumerate}
\item The pair $(g, F)$ is a solution of the system \eqref{main4} with  
$$(g,F) \in H^{s}(U)\times H^{s+1}(U).$$
The trace of $(g, F)$ on $\Upsilon\cap U$ is in $H^s(\Upsilon\cap U)\times H^{s+1}(\Upsilon\cap U)$ and realizes the initial and 
boundary data \eqref{i4}-\eqref{b4}. 
\item $U = \cD^+(\Upsilon\cap U,g)$. 
\item On the domain $U$, the solution $(g, F)$ is unique. 
\item The solution $(U, g, F)$ on the fixed domain depends continuously on the initial-boundary data $(I,B_\cC)$. 
\end{enumerate}
\end{theorem} 

We note that the `size' of the domain $U$ in the theorems above depends on the prescribed initial-boundary data. However, the domain $U$ on which a solution $(g, F)$ exists is not unique; for example one may consider solutions on domains $U' \subset U$. It is only claimed that on the fixed point-set $U \subset M$, the solution $(g, F)$ is unique. It is well-known that such uniqueness fails on domains $\hat U$ which strictly contain the domain of dependence of their initial-boundary surface, i.e.~for $\hat U$ such that $\cD^+(\hat U\cap \Upsilon,g) \subset \subset \hat U$.

The regularity stated in Theorem \ref{exist1}-\ref{exist2} is likely not optimal in that there is a loss of half of derivative in the statement. This will not be pursued further here, cf.~also \cite[Ch. 11]{Benzoni-Serre:2007}. Note that these theorems also prove well-posedness in the space 
$$(g, F) \in C^{r}(I, H^{s-1-r}(S))\times C^r(I, H^{s-r}(S)),$$
for $0 \leq r \leq s-1$. 

\begin{remark}\label{dropF}
{\rm A simple inspection of the proofs shows that Theorems \ref{exist1} and \ref{exist2} remain valid when the wave map variable $F$ is dropped, specifying then initial-boundary data solely for the gauge-reduced Einstein equations for $g$. Moreover, this may be done with respect to an arbitrary fixed (smooth) foliation $\cF = \{t = const\}$ of $U$. However, as discussed further below, it does not appear possible to glue such local solutions together to obtain solutions on larger domains in general. 
}
\end{remark}

Theorem \ref{exist1} and \ref{exist2} prove local existence and uniqueness of solutions to the gauge-reduced IBVP with extended initial-boundary data $(I, B)$ or $(I, B_{\cC})$. {Local existence of solutions $(g,F)$ with $V_g=0$ to the (ungauged) IBVP \eqref{main}-\eqref{b1} (or \eqref{main2}-\eqref{b2}) follows easily from Lemma \ref{Hvac1} and Theorem \ref{exist1} (or Theorem \ref{exist2}).}
The next result, which uses Lemma \ref{Hvac2}, establish geometric uniqueness of general local solutions to the ungauged IBVP's.

Recall $\mathrm{Diff}(M)$ is the group of diffeomorphisms $\Psi: M \to M$, where $\Psi$ extends to a diffeomorphism of an open neighborhood of $M$ into itself and induces diffeomorphisms $\Psi_{\cC}: \cC \to \cC$, $\Psi_S: S \to S$ and $\Psi_{\Si}: \Si \to \Si$. Let $\mathrm{Diff}_0(M)$ be the group of diffeomorphisms $\Psi$ of $M$ equal to the identity on $\Upsilon = S \cup \cC$. For any local corner neighborhood $U\subset M$, let $\mathrm{Diff}(U)$ denote the space of maps $\Psi:U\to M$ which are diffeomorphisms onto its image $\Psi(U)\subset M$ with $S\cap\Psi(U)=S\cap U$ and $\cC\cap \Psi(U)=\cC\cap U$ and which induce  diffeomorphisms $\Psi|_{S\cap U}:S\cap U\to S\cap \Psi(U)$, $\Psi|_{\cC\cap U}:\cC\cap U\to\cC\cap \Psi(U)$ and $\Psi|_{\Si\cap U}:\Si\cap U\to\Si\cap\Psi(U)$. Also, let ${\rm Diff}_0(U)$ denote the subspace of ${\rm Diff}(U)$ where $\Psi|_{S\cap U}={\rm Id}_{S\cap U}$ and $\Psi|_{\cC\cap U}={\rm Id}_{\cC\cap U}$.
The space $\mathrm{Diff}(U)$ acts on solutions $(g, F)$ by pull-back: 
\be \label{diff00}
(\Psi, (g, F)) \to (\Psi^*g, \Psi^*F),
\ee
where $\Psi^*F = F\circ \Psi$. 
 
\begin{proposition}\label{unique1} (Local Geometric Uniqueness I) Fix an initial-boundary data $({\bf I},{\bf B})$ as in \eqref{I1},\eqref{B1}. A local solution $(U, g, F)$ to the IBVP \eqref{main}-\eqref{b1} with $U = \cD^+(\Upsilon\cap U,g)$ is locally unique up to the action of $\mathrm{Diff}_0(U)$, i.e.~if $U$ is a local corner neighborhood with a standard corner chart $\chi$, and $(g, F),~(\w g, \w F)$ both solve \eqref{main} in $U$, satisfy the initial conditions \eqref{i1} on $S\cap U$, and realize the boundary conditions \eqref{b1} on $\cC\cap U$, then there exists an open subset $U'\subset U$ covering $S\cap U$ and a diffeomorphism $\Psi \in \mathrm{Diff}_0(U')$ such that 
\be \label{locun}
(\Psi^*\w g, \Psi^*\w F) = (g, F).
\ee
In particular $\w g$ is isometric to $g$ in $U'$.
\end{proposition}

\begin{proof}  
Fix a standard corner chart $\chi=\{t,x^i\}$ on $U$. The metrics $g$ and $\w g$ induce the same Riemannian metric $g_S=\w g_S=\g$ and second fundamental form $K_{g|S}=K_{\w g|S}=\k$ on $S\cap U$; the compatibility conditions at the corner imply that they have the same corner geometry at $\Sigma\cap U$. Thus there exists a diffeomorphism $\Psi_0\in {\rm Diff}_0(U)$ such that, in $\chi$, the coordinate  components $(\Psi_0^*\w g)_{\a\b}$ and its time derivative $\p_t(\Psi_0^*\w g)_{\a\b}$ on $S\cap U$ agree with the initial value $(g_{\a\b}, \p_t g_{\a\b})$ of $g$. Meanwhile, since $\Psi_0|_{\Upsilon\cap U}={\rm Id}$, the new triple $(U,g_2=\Psi_0^*\w g,F_2=\Psi_0^*\w F)$ also solves the system \eqref{main}-\eqref{b1} with respect to the fixed initial-boundary data $({\bf I,B})$

By Lemma~\ref{Hvac2}, there is an open subset $U'\subset U$ covering $S\cap U$ and a diffeomorphism $\Psi_1\in {\rm Diff}_1(U')$ such that $\Psi_1^*g$ is in the harmonic gauge, i.e. $V_{\Psi_1^*g}(\chi)=0$. It then follows that the new triple $(U', g_1=\Psi_1^*g,F_1=\Psi_1^*F)$ solves the gauged IBVP \eqref{main3}-\eqref{b3} with respect to an extended local initial-boundary data $(I,B)$ raised from $({\bf I},{\bf B})$. Moreover since $\Psi_1$ equals to identity to the first order on $S\cap U'$, we have 
 \be\label{extend1}
(g_1)_{\a\b}=g_{\a\b},~\p_t(g_1)_{\a\b}=\p_t g_{\a\b}\ \ \mbox{on }S\cap U'.
 \ee
By applying the same argument to the pair $(g_2,F_2)$, we note there is a diffeomorphism $\Psi_2$ such that $(U',\Psi_2^*g_2,\Psi_2^*F_2)$ (shrink $U'$ if necessary) solves the gauged IBVP \eqref{main3}-\eqref{b3} with respect to an extended local initial-boundary data $(\w I,\w B)$ raised from $({\bf I},{\bf B})$ and 
 \be\label{extend2}
 (\Psi_2^*g_2)_{\a\b}=(g_2)_{\a\b}=g_{\a\b},~\p_t (\Psi_2^*g_2)_{\a\b}=\p_t(g_2)_{\a\b}=\p_t g_{\a\b}\ \ \mbox{on }S\cap U'.
 \ee
Equations \eqref{extend1}-\eqref{extend2} further imply that $(I,B)=(\w I,\w B)$, i.e. $(U',g_1,F_1)$ and $(U',\Psi_2^*g_2,\Psi_2^*F_2)$ are both solutions to the gauged system \eqref{main3}-\eqref{b3} with respect to the same initial-boundary data $(I,B)$. Therefore, $(\Psi_2^*g_2, \Psi_2^*F_2)= (g_1, F_1)$ by the uniqueness in Theorem \ref{exist1} above. 
This proves the result. 
\end{proof}

The same discussion and result holds for the IBVP \eqref{main2}-\eqref{b2}.
 
\begin{proposition}\label{unique2}(Local Geometric Uniqueness II)
Fix an initial-boundary data $({\bf I},{\bf B}_\cC)$ as in \eqref{I2},\eqref{B2}. A local solution $(U, g, F)$ to the IBVP \eqref{main2}-\eqref{b2} with $U = \cD^+(\Upsilon\cap U,g)$ is locally unique up to the action of $\mathrm{Diff}_0(U)$, i.e.~if $U$ is a local corner neighborhood with a standard corner chart $\chi$, and $(g, F),~(\w g, \w F)$ both solve \eqref{main2} in $U$, satisfy the initial conditions \eqref{i2} on $S\cap U$, and realize the boundary conditions \eqref{b2} on $\cC\cap U$, then there exists an open subset $U'\subset U$ covering $S\cap U$ and a diffeomorphism $\Psi \in \mathrm{Diff}_0(U')$ such that 
\be \label{locun2}
(\Psi^*\w g,\Psi^* \w F) = (g, F).
\ee
In particular $\w g$ is isometric to $g$ in $U'$. 
\end{proposition}
 
Of course the domain $U$ in the propositions above is not unique. For instance if $(U, g, F)$ is a solution, then so is $(U', g, F)$ for any open subset $U' \subset U$ with $\Upsilon \cap U' = \Upsilon\cap U$ and {$U' = \cD^+(\Upsilon\cap U',g)$. Nevertheless, the same proof as above shows that if $(U_1, g_1, F_1)$ and $(U_2, g_2, F_2)$ are two such solutions with the same initial-boundary data on $\Upsilon\cap U_1 = \Upsilon\cap U_2$, then there are {subdomains $U'_1 \subset U_1$, $U'_2 \subset U_2$ with $S \cap U'_i = S\cap U_i$}, $U'_i = \cD^+(\Upsilon \cap U'_i,g_i)$ and a diffeomorphism $\Psi: U'_1 \to U'_2$, equal to the identity on $\Upsilon\cap U'_1$, such that $\Psi^*(g_2, F_2) = (g_1, F_1)$.

\begin{remark}\label{alt_theta}
{\rm These uniqueness results in Propositions \ref{unique1} and \ref{unique2} no longer hold when the wave map $F$ is dropped from the system. The reason is that boundary quantities of $g$ alone are not invariant under $\mathrm{Diff}_1(U)$ or $\mathrm{Diff}_0(U)$. In particular, the normal component $\nu_g$ is not invariant under $\mathrm{Diff}_1(U)$. 

Thus one of the main reasons for introduction of the wave map $F$, and one of the main consequences of the results above, is that it is possible to establish, in a relatively simple way, the geometric uniqueness result, (and the related gluing procedure), for the expanded IBVP's.

Also, as noted in \eqref{altb}, there are a number of alternate boundary conditions one may impose in place of those for $\Theta$ or $\Theta_{\cC}$ in \eqref{b1} or \eqref{b2} to obtain existence results analogous to Theorems \ref{exist1} and \ref{exist2}. For this, one only requires the linearization of the $\Theta$-equations at flat data $(g_0,F_0)$ to have the same basic form as that analysed in \S 3. However, to preserve the uniqueness results above requires significant restrictions on the choice of boundary equations for $\Theta$; in particular the geometric quantities must be invariant under $\mathrm{Diff}_0(U)$. 
}
\end{remark}

Next we generalize the geometric uniqueness results above by working with the larger diffeomorphism group ${\rm Diff}(M)$.
Using the fact that the timelike normal vectors for $g$ and the pull-back metric $\Psi^*g$ transform as $T_{\Psi^*g} = (\Psi^{-1})_*T_g$ and similarly for the spacelike normal vector $\nu_g$, one easily verifies the following transformation rule for the initial and boundary geometry of the pairs $(g,F)$ and $(\Psi^*g,\Psi^*F)$ with $\Psi\in{\rm Diff}(M)$:
\bes
(\Psi^*g)_S=\Psi_S^*(g_S),~K_{\Psi^*g|S}=\Psi_S^*K_{g|S},~(\Psi^*F)_*(N_{\Psi^*g})=F_*(N_g)\ \ \mbox{on }S.
\ees
and
\bes
[\big((\Psi^*F)^{-1}\big)^*(\Psi^*g)^\intercal]=[(F^{-1})^*g^\intercal],~(\Psi^*F)_*(T_{\Psi^*g}+\nu_{\Psi^* g})=F_*(T_g+\nu_g)\ \ \mbox{on }\cC.
\ees
Thus for a general initial-boundary data $({\bf I,B})$ given in \eqref{I1}-\eqref{B1}, we define the action by the group ${\rm Diff}(\Upsilon)$ on it as
\be \label{dif}
\psi^*{\bf I}=\big( \psi^*\g,\psi^*\k, E_0\circ \psi, E_1\big),\ \ \psi^*{\bf B}=\big(G\circ\psi, [\s],\Theta\big).
\ee
Here $\psi$ belongs to the group ${\rm Diff}(\Upsilon)$ of diffeomorphisms $\psi:\Upsilon\to\Upsilon$ such that $\psi:S\to S$ and $\psi:\cC\to\cC$. It is crucial for the uniqueness results discussed below that the boundary data $([\s], \Theta)$ are invariant under the action of the gauge group $\mathrm{Diff}(\Upsilon)$.

Define two collections of initial-boundary data ${\bf I}_i=\big(\g_i,\k_i,(E_0)_i,(E_1)_i\big)$ and ${\bf B}_i= \big(G_i,[\s_i],\Theta_i\big)$, $(i=1,2)$ to be \textit{equivalent} if there is a diffeomorphism $\psi\in\mathrm{Diff}(\Upsilon)$ such that
\be \label{equiv}
\psi^*({\bf I}_2, {\bf B}_2) = ({\bf I}_1, {\bf B}_1).
\end{equation}
The same relation holds for the $({\bf I}, {\bf B}_{\cC})$ initial-boundary data, i.e. ~$({\bf I}_1, ({\bf B}_{\cC})_1) \sim ({\bf I}_2, ({\bf B}_{\cC})_2)$ if and only if there is a diffeomorphism $\psi \in \mathrm{Diff}(\Upsilon)$ such that
\be \label{equiv2}
\psi^*({\bf I}_2, ({\bf B}_{\cC})_2) = ({\bf I}_1, ({\bf B}_{\cC})_1).
\end{equation}

Now apply these equivalence relations locally in a local corner neighborhood $U$. 

\begin{corollary}\label{unique3}(Local Geometric Uniqueness III)
If $(U, g_1,F_1)$ and $(U, g_2,F_2)$ are two solutions to the IBVP \eqref{main}-\eqref{b1} (or \eqref{main2}-\eqref{b2}) with respect to initial-boundary data that are equivalent in the sense of \eqref{equiv} (or \eqref{equiv2}) on $\Upsilon\cap U$, then there exists an open subset $U'\subset U$ covering $S\cap U$ and a diffeomorphism $\Psi \in \mathrm{Diff}(U')$ such that 
\be \label{locun3}
\Psi^*g_2=g_1\ {\rm and} \ \Psi^{*}F_2 = F_1\ \ \text{in }U'.
\ee
\end{corollary}

\begin{proof} 
The proof is the same for both systems, so we work with the system \eqref{main}-\eqref{b1}. Let $(g_1, F_1)$ and $(g_2, F_2)$ 
be solutions in $U$ satisfying \eqref{main}-\eqref{b1} with respect to the initial-boundary data $({\bf I}_1, {\bf B}_1)$ and $({\bf I}_2, {\bf B}_2)$ respectively. 
Suppose $({\bf I}_1, {\bf B}_1)$ and $({\bf I}_2, {\bf B}_2)$ are related as in \eqref{equiv}, so that there is a diffeomorphism $\psi \in \mathrm{Diff}(\Upsilon\cap U)$ such that $\psi^*({\bf I}_2, {\bf B}_2)$ = $({\bf I}_1, {\bf B}_1)$. Extend $\psi$ to a diffeomorphism $\Psi_1 \in \mathrm{Diff}(U)$. Then $\Psi_1^*(g_2,F_2)$ becomes a solution with initial-boundary data $({\bf I}_1, {\bf B}_1)$. By Proposition \ref{unique1}, there is a diffeomorphism $\Psi_2 \in \mathrm{Diff}_0(U')$ for some $U'\subset U$ covering $S\cap U$ such that $\Psi_2^*\Psi_1^*(g_2, F_2) = (g_1, F_1)$ in $U'$. Thus \eqref{locun3} holds with $\Psi = \Psi_1\circ\Psi_2$. In particular we have $\Psi|_{\Upsilon\cap U'}=\psi$.
\end{proof}

The local geometric uniqueness above for the IBVP of the expanded system of $(g,F)$ is the same as that for the Cauchy problem of the Einstein equations. Namely, for the Cauchy problem recall that if two solutions have equivalent initial data $(S, \g,\k)$, then there exists a local $4$-diffeomorphism, i.e.~isometry, relating the two solutions, restricting to suitable domains if necessary. Similarly here, to check whether two sets of initial-boundary data generate isometric solutions, it is sufficient to examine the equivalence relation \eqref{equiv} or \eqref{equiv2} and one does not need to solve the IBVP explicitly. 



\section{Gluing and geometric uniqueness for the IBVP.} 

Up until this point, all the discussion has been local, in a sufficiently small neighborhood $U$ of a corner point $p \in \Si$ and for the expanded system of $(g, F)$. We now turn to the global (in space) issue of existence and geometric uniqueness of the IBVP in a full ST-neighborhood of the initial surface $S$ in $M$. This is obtained by gluing local solutions together, using local geometric uniqueness. This gluing is first carried out in a full {ST-corner neighborhood} $\cU$ for the expanded system in Theorem \ref{corner_exist1} below. {By applying these results for the pair $(g,F)$, we then obtain in Theorem \ref{corner_exist2} similar results for the IBVP of the Einstein metric $g$ alone, as discussed in \S 1. }
 
To pass to the interior, away from the corner $\Si$, recall that given initial data $(\g,\k)$ on $S$, it is proved in \cite{Choquet-Geroch:1969} that there is a maximal solution of the Cauchy problem of vacuum Einstein equations, i.e.~a maximal globally hyperbolic vacuum spacetime $(M_{\mathring{S}},g)$ where $\mathring{S}=S\setminus\Si$ is embedded in $M_{\mathring{S}}$ as a Cauchy surface with induced metric and second fundamental form $(g_{\mathring S},K_{g|\mathring{S}})$ equal to $\bI = (\g,\k)$. To obtain vacuum developments of the full initial-boundary data $(\Upsilon, \bI, \bB)$ in Theorem \ref{exist}, we will show that the solution $(\cU, g)$ around the corner (referred as a ``boundary vacuum development" in Theorem \ref{corner_exist2}) may be smoothly patched with the interior solution $(M_{\mathring S},g)$, giving then a global development. A similar analysis applies well to the expanded system for $(g, F)$, cf.~Remark~\ref{G_gF}.

As noted above, (up to isometry) solutions $g$ of the Cauchy problem for the vacuum Einstein equations do not depend on any choice of gauge or local coordinates while the solution near the boundary $\cC$ are gauge (i.e.~$\f_g$ or $F$) dependent. This is another 
reason that the analysis needs to be separated into the (pure) Cauchy problem (without boundary) and the IBVP in a neighborhood $\cU$ 
of the boundary $\cC$. 

\subsection{Vacuum Developments}
We first make a couple of definitions. The arguments to follow regarding pairs $(g, F)$ do not depend on the choice of ${\bf B}$ or ${\bf B}_{\cC}$ boundary data, so we will not distinguish ${\bf B}$ and ${\bf B}_{\cC}$ and use $\cB$ to denote either one of them. 
Recall that $\Upsilon=S\cup\cC$ is the initial-boundary surface of $M$. In the following, we also identify $\Upsilon$ with $S_0 \cup \cC_0 \subset M_0$ via a fixed, but arbitrary diffeomorphism identifying $S\to S_0$ and $\cC\to \cC_0$, (unrelated to the wave map $F$).
An {\it initial-boundary data set} $(\Upsilon, {\bf I}, {\cB})$ consists of an initial-boundary surface $\Upsilon$, an initial data ${\bf I}$ as in \eqref{I1} or \eqref{I2} defined on $S$, and a boundary data ${\bf B}$ as in \eqref{B1} or ${\bf B}_\cC$ as in \eqref{B2}, (where we understand $G:\cC\to \cC_0$ as a diffeomorphism and $([\s],\Theta)$ in \eqref{B1} or $([\s],H,\Theta_\cC)$ in \eqref{B2} defined on $G(\cC)\subset \cC_0$). Both $({\bf I, B})$ and $({\bf I, B}_\cC)$ satisfy the $C^{s-1}$ compatibility conditions on $\Si$. As previously, we fix a time function $t_0$ and a complete Riemannian metric $g_R$ on the target space $M_0$.

Let $\rho_0$ denote the distance function to $\Si$ on $S$ with respect to the background Riemannian metric $g_R$. A {\it partial initial-boundary 
surface} $P \subset \Upsilon$ is an initial-boundary surface of the form 
$$P = P_{\rho_0}\cup (\cC_G)_{\tau}$$
where $P_{\rho_0} \subset S$ is the $\rho_0$-tubular neighborhood of $\Si$ in $S$ with respect to $g_R$ and $(\cC_G)_{\tau} = \{p \in \cC: t_0\circ G(p)\in [0,\tau)\}$ with $G$ the Dirichlet boundary data of $F$ given in $\cB$. We will allow $\tau = \infty$ but assume $\rho_0$ is small, so that $P_{\rho_0} \cong I\times \Si$.  A partial initial-boundary data set $(P, {\bf I}, {\cB})$ is defined as the restriction of $(\Upsilon,{\bf I}, {\cB})$ to the subset $P$.

Recall that an ST-corner neighborhood $\cU$ is an open neighborhood of the corner $\Sigma$, which contains the entire corner $\Si$ and admits spacelike initial surface $S\cap\cU$ and timelike boundary $\cC\cap\cU$.

\begin{definition}\label{vacf_dep}
A \textit{boundary vacuum development with gauge fields} for the partial data set $(P,\mathbf{I},\cB)$ is an {ST-corner neighborhood} $\cU \subset M$, equipped with a pair $(g,F)$ such that:
\begin{enumerate}
\item $(g,F)$ solves \eqref{main} in $\cU$.
\item $\cU\cong [0,\rho_0)\times (\cC_G)_{\tau}$ is diffeomorphic to a product neighborhood of $(\cC_G)_{\tau}$, with its initial-boundary surface identified with $P$ in a natural way. 
\item  $P_{\rho_0}\cap\cU$ is spacelike and $(\cC_G)_{\tau}\cap\cU$ is timelike in $(\cU, g)$.
\item $F$ is a diffeomorphism from $\cU$ onto a domain $\cU_0\subset M_0$.
\item $(\cU,g,F)$ satisfies the conditions \eqref{i1}-\eqref{b1} or \eqref{i2}-\eqref{b2} with the given initial-boundary data $({\bf I},{\bf B})$ or $({\bf I}, {\bf B}_\cC)$ on $P$.
\item By choosing a smaller neighborhood if necessary, we require that 
$\cU=\cD^+(\Upsilon\cap \cU,g)$, i.e.~$\cU$ is equal to the future domain of dependence of its initial-boundary surface $\Upsilon\cap \cU$ in $(\cU,g)$.
\end{enumerate}
\end{definition}

Recall from Remark \ref{loc_diff} (ii) that for any local solution $(U, g, F)$ near $\Si$, the map $F$ is a diffeomorphism 
from its domain onto its image $F(U) \subset M_0$ in a neighborhood of $\Si$.

We first prove the semi-global existence result for the expanded IBVP's \eqref{main}-\eqref{b1} and \eqref{main2}-\eqref{b2}.

\begin{theorem}\label{corner_exist1}
{Let $(P, {\bf I}, {\cB})$ be a partial initial-boundary data set on $P = P_{\rho_0} \cup (\cC_G)_{\tau}$}, with $g$-initial data 
$(\g,\k) \in H^{s+\frac{1}{2}}(P_{r_0})\times H^{s-\frac{1}{2}}(P_{r_0})$ 
satisfying the vacuum constraint equations \eqref{constraint}, and $F$-initial data 
$(E_0,E_1) \in H^{s+\frac{3}{2}}(P_{r_0})\times H^{s+\frac{1}{2}}(P_{r_0})$ 
as in \eqref{I1}, together with boundary data 
$(G, [\s], \Theta) \in H^{s+\frac{3}{2}}((\cC_G)_{\tau})\times H^{s+\frac{1}{2}}(\cC_{\tau})\times H^{s + \frac{1}{2}}(\cC_{\tau})$ 
as in \eqref{B1} (or  
$(G, [\s], H, \Theta_{\cC}) \in H^{s+\frac{3}{2}}((\cC_G)_{\tau})\times H^{s+\frac{1}{2}}(\cC_{\tau})\times H^{s - \frac{1}{2}}(\cC_{\tau})\times H^{s+\frac{1}{2}}(\cC_{\tau})$ 
as in \eqref{b4}), satisfying the $C^{s-1}$ compatibility conditions. 
Then there exists $\tau' > 0$, $\rho_0' > 0$ so that, for the subset $P'=P_{\rho_0'}\cup (\cC_G)_{\tau'}$, $(P',{\bf I},\cB)$ admits a boundary vacuum development with gauge fields, i.e.~a triple $(\cU,g,F)$ with 
\be \label{regs1}
(g,F) \in H^{s}(\cU)\times H^{s+1}(\cU),
\ee
and $(g,F)$ has trace on $\Upsilon\cap\cU$ in $H^s(\Upsilon\cap\cU)\times H^{s+1}(P')$ realizing all the conditions in Definition \ref{vacf_dep}. 

Moreover, two boundary vacuum developments with gauge fields for the same partial initial-boundary data $(P,{\bf I},\cB)$ are isometric in a neighborhood of a subset $P'\subset P$. 
\end{theorem} 

\begin{proof} Here we give the proof in the case $\cB={\bf B}$. The same proof works for $\cB={\bf B}_{\cC}$.
By the local existence theorem, Theorem \ref{exist1}, for any point $p\in\Sigma$, there exists an open neighborhood $V\ni p$ in $P$ admitting a local vacuum development, i.e. in a local corner neighborhood $U$ with $U\cap P=V$, $(U,g,F)$ solves the system \eqref{main}-\eqref{b1} with respect to the initial-boundary data $({\bf I},{\bf B})$ (restricted on $V$). Choose then a finite collection of open subsets $\{V_n\}_{n=1}^m$ of $P$ covering the corner $\Sigma$. Each $V_n$ is equipped with initial-boundary data $(\mathbf I_n,\mathbf B_n)$ obtained by restricting $(\bf I,B)$ to $V_n$ and each $(V_n,\mathbf I_n,\mathbf B_n)$ admits a local vacuum development $(U_n,g_n,F_n)$. When two subsets $V_n, V_m$ overlap, their vacuum developments can be patched together in the following way.

Let $\mathring{U}_n$ denote the image of $U_n$ under $F_n$ i.e. $\mathring{U}_n=F_n(U_n)\subset M_0$ and let $g_{F_n}$ denote the pull-back metric $(F_n^{-1})^*g_n$ on $\mathring{U}_n$. Then $(g_{F_n},{\rm Id}_{\mathring{U}_n})$ is a solution to \eqref{main}-\eqref{b1} on $\mathring{U}_n$ with the initial-boundary data $\mathring{\bf I}_n$ on $S_0\cap \mathring{U}$ and $\mathring{\bf B}_n$ on $\cC_0\cap\mathring{U}$ given by
\begin{equation*}
\mathring{\bf I}_n=\big((E_0^{-1})^*\g,(E_0^{-1})^*\k,{\rm Id}_{S_0\cap \mathring{U}_n},E_1\big),\ \ \mathring{\bf B}_n=\big({\rm Id}_{\cC_0\cap \mathring{U}_n},[\s],\Theta\big).
\end{equation*}
The same applies to $(U_m,g_m,F_m)$, so we obtain a triple $(\mathring{U}_m,g_{F_m},{\rm Id}_{\mathring{U}_m})$.
Observe that on the common overlapping initial-boundary surface $\Upsilon\cap \mathring{U}_n\cap \mathring{U}_m$, the pairs $(g_{F_n},{\rm Id}_{\mathring{U}_{n}})$ and $(g_{F_m},{\rm Id}_{\mathring{U}_{m}})$ have the same initial-boundary data. By geometric uniqueness Proposition \ref{unique1}, there is a subdomain $\mathring{U}_{nm}$ covering $\Upsilon\cap \mathring{U}_n\cap \mathring{U}_m$ and a diffeomorphism $\psi\in\mathrm{Diff}_0 (\mathring{U}_{mn})$ such that $g_{F_n}=\psi^*g_{F_m}$ and ${\rm Id}_{\mathring{U}_{n}}={\rm Id}_{\mathring{U}_{m}}\circ \psi$ on $\mathring{U}_{mn}$. Obviously from the latter equation, $\f={\rm Id}_{\mathring{U}_{nm}}$. Hence 
$$g_{F_n}=g_{F_m}\mbox{ in }\mathring{U_{mn}}.$$
It follows by induction that the local metrics $g_{F_n}$ can be trivially glued together to obtain a solution $(\mathring{g},{\rm Id}_{\mathring{\cU}})$ on some ST-corner neighborhood $\mathring{\cU}$ of $\Si_0$ in $M_0$ satisfying \eqref{main}-\eqref{b1} with initial-boundary data given by
\begin{equation*}
\mathring{\bf I}=\big( (E_0^{-1})^*\g,(E_0^{-1})^*\k,{\rm Id}_{S_0\cap \mathring{\cU}},E_1\big),\ \ \mathring{\bf B}=\big({\rm Id}_{\cC_0\cap \mathring{\cU}},[\s],\Theta\big).
\end{equation*}
Since $\mathring{\cU}$ is patched up by finite local solutions, it is easy to adjust the domain so that $\cC_0\cap\mathring{\cU}=\{t_0\in[0,\tau')\}$ for some $\tau'>0$ and $S_0\cap\mathring{\cU}=E_0(P_{\rho_0'})$ for some $\rho_0'>0$. Now construct a diffeomorphism $F:M\to M_0$ such that $F|_S=E_0$ and $F|_{\cC}=G$. Let $g=F^*\mathring{g}$ and $\cU=F^{-1}(\mathring{\cU})$. Then it is easy to check that $(\cU,g,F)$ is a boundary vacuum development with gauge fields of some sub-data $(P',{\bf I},{\bf B})$ of $(P,{\bf I},{\bf B})$.

Next let $(\cU_1,g_1,F_1)$ and $(\cU_2,g_2,F_2)$ be a pair of boundary vacuum developments with gauge fields of the same $(P,{\bf I},\cB)$. By local uniqueness Proposition \ref{unique1} or \ref{unique2}, at every corner point $p\in\Si$ there is an open neighborhood $U$ and a diffeomorphism $\psi\in{\rm Diff}_0(U)$ such that $\psi^*g_1=g_2$ and $F_1\circ\psi=F_2$ on $U$. Since $F_i~(i=1,2)$ is a local diffeomorphism, the second equation uniquely determines $\psi=F_1^{-1}\circ F_2$. Patching up naturally such local neighborhoods, we obtain an ST-corner neighborhood $\cU$ covering a subset $P'\subset P$ in which $\Psi^*g_1=g_2$ and $F_1\circ\Psi=F_2$ for the unique $\Psi\in{\rm Diff}_0(\cU)$ determined by $\Psi=F_1^{-1}\circ F_2$.
\end{proof}

\begin{remark}
{\rm As noted in Remark \ref{dropF}, with regard to local existence one may drop the wave map $F$ and locally solve the IBVP for the metric $g$ with ${\bf B}$ or ${\bf B}_{\cC}$ data. This is done with respect to a local chart $\chi: U \to {\bf R}$ in which the coordinate functions are $g$-harmonic. Suppose $\chi': U' \to {\bf R}$ is another local chart with $U \cap U' \neq \emptyset$ giving rise to a solution $g'$ in $U'$. If the chart $\chi$ is affinely related to the $\chi$ chart on $U \cap U'$, then the coordinates of $\chi'$ are also harmonic with respect to $g$, and so the uniqueness in Theorem \ref{exist1} implies that $g' = g$ on $U \cap U'$.  

In this very special case, where the domain of $g$ has an atlas of affinely related charts preserving the manifold-with-corner structure, (so the domain has an affine-flat structure) with corresponding affine-related initial-boundary data, one may patch together local solutions to obtain a larger solution $g$. However, there appears to be no method to prove such solutions are unique. 
}
\end{remark}

We proceed to discuss the analogs of these results for the IBVP of Einstein equations \eqref{maing3}-\eqref{bg3} which involves the associated gauge $\f_g$ determined uniquely and implicitly by $g$ as in \eqref{mainf3}-\eqref{bf3}. { Here and in the following, we will always assume a fixed but arbitrary choice of boundary gauge $\Theta_{\cC}$. Moreover, as discussed in \eqref{cibc3}, we assume $\Theta_\cC$ is chosen so that
\be\label{corner1}
\Theta_\cC=\ell T_{g_R} \mbox{ on }\Si_0
\ee
for some function $\ell $ on $\Si_0$ such that {$0<\ell<\tfrac{1}{\sqrt{2}}$}.
}

Note first that given a fixed metric $g$ satisfying compatibility conditions at the corner, there is a unique solution $\f_g$ of the system \eqref{mainf3}-\eqref{bf3} in an ST-corner neighborhood $\cU$. Since the boundary conditions for $\f_g$ are a simple combination of Sommerfeld and Dirichlet boundary conditions, this existence and uniqueness follows by standard results for IBVP's of systems of semi-linear wave equations. The conditions on the initial data $(E_{g_S}, T_{g_R})$ imply that $\f_g: \cU \to \cU_0 \subset M_0$ is a diffeomorphism onto its image $\cU_0$. The uniqueness also gives the equivariance property \eqref{equif}. 
  
Recall that ${\mathbb I}=(\g,\k),~{\bB}=([\s],H)$ denote the initial-boundary data in the system \eqref{maing3}-\eqref{bg3}. In the following the initial-boundary data set $(\Upsilon,{\mathbb I},\bB)$ consists of a initial-boundary surface $\Upsilon$, an initial data ${\bI}$ defined on $S$ and a boundary data ${\bB}$ defined on $\cC_0$. A partial initial-boundary data set $(P,{\mathbb I},\bB)$ is a subset $P=P_{\rho_0}\cup \cC_\tau$ of $\Upsilon$ with $(\bI,\bB)$ restricted on it, where $P_{\rho_0}\subset S$ is the tubular neighborhood of $\Si$ defined as above, and $\cC_\tau=\{p\in\cC_0:t_0(p)\in[0,\tau)\}$.

\begin{definition}\label{vac_dep}
A \textit{boundary vacuum development} for the partial initial-boundary data set $(P,\mathbb{I},\bB)$ is an {ST-corner neighborhood}  $\cU\subset M$, equipped a Lorentz metric  $g$ on $\cU$ such that:
\begin{enumerate}
\item $g$ is Ricci-flat on $\cU$.
\item The unique associated gauge $\f_g$ for  $g$ via \eqref{mainf3}-\eqref{bf3} is a diffeomorphism from $\cU$ onto its image in $M_0$.
\item $\cU\cong [0,\rho_0)\times (\cC_{\phi_g})_{\tau}$ is diffeomorphic to a product neighborhood of $(\cC_{\phi_g})_{\tau}=\{p\in\cC\cap\cU: t_0\circ\phi_g(p)\in[0,\tau)\}$ with its initial and boundary surface identified with $P$ naturally.
\item  $P_{\rho_0}\cap\cU$ is spacelike and $(\cC_{\phi_g})_{\tau}\cap\cU$ is timelike in $(\cU, g)$.
\item $(\cU,g)$ satisfies the conditions \eqref{ig3}-\eqref{bg3} with the given initial-boundary data $(\mathbb I,\bB)$ on $P$.
\item By choosing a smaller neighborhood if necessary, we require that $\cU=\cD^+(\Upsilon\cap \cU,g)$.
\end{enumerate}
\end{definition}

The semi-global analog of Theorem \ref{exist2} is:

\begin{theorem}\label{corner_exist2}
Let $(P, {\mathbb I}, {\bB})$ be a partial initial-boundary data set on $P = P_{\rho_0} \cup \cC_{\tau}$, with $g$-initial data $(\g,\k) \in H^{s+\frac{1}{2}}(P_{r_0})\times H^{s-\frac{1}{2}}(P_{r_0})$ satisfying the constraint equations \eqref{constraint}, together with boundary data $( [\s], H) \in H^{s+\frac{1}{2}}(\cC_{\tau})\times H^{s-\frac{1}{2}}(\cC_{\tau})$ as in \eqref{ibg3} satisfying the $C^{s-1}$ compatibility conditions. 
Then there exists $\tau' > 0$, $\rho_0' > 0$ so that there is a boundary vacuum development $(\cU,g)$ for the subset $P'=P_{\rho_0'}\cup \cC_{\tau'}\subset P$ with
\be \label{regs11}
g \in H^{s}(\cU),
\ee
and $g$ has trace on $\Upsilon\cap\cU$ in $H^s(\Upsilon\cap\cU)$ realizing the conditions in Definition \ref{vac_dep}. 

Two boundary vacuum developments of the same partial initial-boundary data $(P,\mathbb I,\bB)$ are isometric, in an ST-corner neighborhood in $M$. 
\end{theorem} 

\begin{proof}
To show existence of a vacuum development $(\cU,g)$, we expand the initial-boundary data $\big(\mathbb I=(\g,\k),\bB=([\s],H)\big)$ to 
\begin{align}\label{eib}
{\bf I}=\big(\g,\k,E_0=E_{\g}, E_1=T_{g_R}\big),~{\bf B}_\cC=\big(G,[\s],H,\Theta_{\cC}\big),
\end{align}
which can be considered as the initial-boundary data for the coupled system \eqref{main2}-\eqref{b2}. {Here $E_\g$ is the map generated by $\g$ as in \eqref{c1};} $G$ is an arbitrary diffeomorphism $G:\cC\to\cC_0$ satisfying the compatibility conditions. Then Theorem 5.2 shows that there is a boundary vacuum development $(g,F)$ for some sub-data $(P',{\bf I,B}_\cC)$. Observe here, based on how it is constructed, the wave map $F$ must be equal to the unique diffeomorphism $\f_g$ determined by \eqref{mainf3}-\eqref{bf3}. It then follows that the so obtained $(\cU,g)$ is a boundary vacuum development of $(P',{\bI,\bB})$.

Suppose $(\cU_1,g_1)$ and $(\cU_2,g_2)$ are two boundary vacuum developments of the same partial initial-boundary data $(P,{\mathbb I}, \bB)$. We can consider the pull-back metrics $g_{\f_1}=(\f_{g_1}^{-1})^*g_1$ and $g_{\f_2}=(\f_{g_2}^{-1})^*g_2$. Let $\mathring{\cU}=\f_{g_1}(\cU_1)\cap\f_{g_2}(\cU_2)\subset M_0$. Then the triples $(\mathring{\cU},g_{\f_1},{\rm Id}_{\mathring{\cU}})$ and $(\mathring{\cU},g_{\f_2},{\rm Id}_{\mathring{\cU}})$ are both vacuum developments with gauge fields of some common subset $P'\subset (P,\mathring{\bf I},\mathring{\bf B}_\cC)$ where 
\begin{align*}
\mathring{\bf I}=\big( (E_{\g}^{-1})^*\g,(E_{\g}^{-1})^*\k,E_0={\rm Id}_{S_0}, E_1=T_{g_R}\big),~\mathring{\bf B}_\cC=\big(G={\rm Id}_{\cC_0},[\s],H,\Theta_{\cC}\big).
\end{align*}
By the uniqueness result in Theorem \ref{corner_exist1}, there exists a diffeomorphism $\Psi\in {\rm Diff}_0(\mathring{\cU})$ (shrinking $\mathring{\cU}$ if necessary) such that $\Psi^*g_{\f_1}=g_{\f_2}$ and ${\rm Id}_{\mathring{\cU}}\circ\Psi={\rm Id}_{\mathring{\cU}}$. Therefore, $g_1$ and $g_2$ are equivalent -- in fact they are related by the unique diffeomorphisms determined by \eqref{mainf3}-\eqref{bf3}, i.e. $(\f_{g_1}^{-1}\circ\f_{g_2})^*g_1=g_2$ on $\phi_{g_2}^{-1}(\mathring{\cU})$.
\end{proof}


Next we define the vacuum development of global initial-boundary data $(\Upsilon,\mathbb{I},\bB)$ on $M$.
\begin{definition}\label{vac_dev}
A \textit{vacuum development} of the initial-boundary data $(\Upsilon,\mathbb{I},\bB)$ is an {ST-neighborhood}  $\cT\subset M$ such that $\{p\in M: t(p)<\tau\}\subset \cT$ for some time function $t$ and some $\tau>0$, equipped with a Lorentz metric $g$ such that:
\begin{enumerate}
\item $g$ is Ricci-flat in $\cT$.
\item The initial surface $S$ is spacelike and the boundary surface $\cC\cap \cT$ is timelike with respect to $g$. In addition, $\cT=\cD^+(\Upsilon\cap\cT,g)$ i.e. $\cT$ is the future domain of dependence of the initial-boundary surface $\Upsilon\cap\cT$ in $(\cT,g)$.
\item The unique associated gauge $\f_g$ for $g$ via \eqref{mainf3}-\eqref{bf3} is a diffeomorphism from an ST-corner neighborhood $\cU$ covering $\cC\cap\cT$ onto its image in $M_0$.
\item $(\cT,g)$ satisfies the initial and boundary conditions in \eqref{ig3}-\eqref{bg3} with the given initial-boundary data $\mathbb I$ on $S$ and $\bB$ restricted to $\phi_g(\cC\cap\cT)\cap \cC_0$.
\end{enumerate}
\end{definition}
We note that based on the definition above, the boundary $\cC\cap\cT$ must contain a subset $(\cC_{\phi_g})_{\tau}=\{p\in\cC:t_0\circ\phi_g (p)\in [0,\tau)\}$ for some $\tau>0$. 

Combining Theorem \ref{corner_exist2} with the solution of the Cauchy problem gives the following result, which is a more precise version of Theorems \ref{exist},\ref{geom_unique}.

\begin{theorem}\label{exist_unique}
Let $(\Upsilon, {\mathbb I}, {\bB})$ be an initial-boundary data set satisfying the assumptions of Theorem \ref{corner_exist2}. Then $(\Upsilon, {\mathbb I}, {\bB})$ admits a vacuum development, i.e.~there exists a pair $(\cT, g)$ such that 
\be \label{regs2}
g \in H^{s}(\cT),
\ee
and with trace on $\Upsilon\cap\cT$ in $H^s(\Upsilon\cap\cT)$ realizing the conditions in Definition 5.6. 

Moreover, vacuum developments of equivalent initial-boundary data are equivalent. In detail, if $(\cT_1,g_1)$ and $(\cT_2,g_2)$ are vacuum developments of $(\Upsilon, {\mathbb I}_1, {\bB}_1)$ and $(\Upsilon, {\mathbb I}_2, {\bB}_2)$ respectively, and $( {\mathbb I}_1, {\bB}_1),( {\mathbb I}_2, {\bB}_2)$ are equivalent as in Definition \ref{equiv_IB}, i.e. there exists a diffeomorphism $\psi\in{\rm Diff}'(S)$ such that 
\begin{equation*}
(\g_1,\k_1)=(\psi^*\g_2,\psi^*\k_2)\ \ \mbox{on }S,
\end{equation*}
and a subdomain $\cC_{0\tau}=\{x\in\cC_0:t_0(x)<\tau\}~(\tau>0)$ in $\cC_0$ such that 
\begin{equation*}
([\s_1],H_1)=([\s_2],H_2)\ \ \mbox{on }\cC_{0\tau},
\end{equation*}
then there are subdomains $\cT'_1 \subset \cT_1$ and $\cT'_2 \subset \cT_2$, with $\cT'_i\supset \{p\in\cT_i: t(p)<\tau_i\}$ for some $\tau_i>0$ ($i=1,2$), such that 
$$\Psi^*g_2=g_1$$
for some diffeomorphism $\Psi: \cT'_1 \to \cT'_2$. In addition, $\Psi|_{S}=\psi$ and $\Psi|_{\cU} = \f_{g_2}^{-1}\circ\f_{g_1}|_{\cU}$ where $\cU$ is a neighborhood of $\cC\cap\cT'_1$ and $\f_{g_1},\f_{g_2}$ are the associated gauge for $g_1,g_2$ (in the boundary gauge $\Theta_{\cC}$) on $\cU$ and $\Psi(\cU)$ respectively.
\end{theorem} 

\begin{proof} By Theorem \ref{corner_exist2} there is a partial initial-boundary data set $(P,{\mathbb I},\bB)$ of $(\Upsilon,{\mathbb I},\bB)$ admitting a boundary vacuum development $(\cU, g)$ defined in an ST-corner neighborhood $\cU$ and unique up to diffeomorphisms equal to the identity on $P\cap \cU$. On the other hand, by the solution to the Cauchy problem for the vacuum Einstein equations, the interior initial data $(S, {\mathbb I}) = (S, \g,\k)$ also admits a vacuum development $(\cU_{\rm int}, g_{\rm int})$, unique up to diffeomorphism in $\mathrm{Diff}_0(\cU_{\rm int})$. {For convenience, we choose $\cU_{\rm int}$ to be the maximal Cauchy development of the initial data and view $\cU_{\rm int} \subset M$.}

By construction ${\cU_{\rm int}}\cap\cU$ is an open neighborhood of $S\cap\cU_{\rm int}$ in $M$. Since $g$ and $g_{\rm int}$ both solve the Ricci-flat equation \eqref{main} in ${\cU_{\rm int}}\cap\cU$ and satisfy the same geometric initial condition with ${\mathbb I}$ on $S\cap\cU_{\rm int}$, there is an open neighborhood $\cV\subset {\cU_{\rm int}}\cap\cU$ covering $S\cap\cU_{\rm int}$ and a diffeomorphism $\psi:\cV\to\cV$ fixing $S\cap\cU_{\rm int}$ such that $g=\psi^*g_{\rm int}$ on $\cV$. By shrinking the open sets $\cU$ and $\cU_{\rm int}$, we can assume $\cU\cap\cU_{\rm int}=\cV$  and then extend $\psi$ to be a diffeomorphism $\cU_{\rm int}\to\cU_{\rm int}$ which fixes $S$. We can then glue $(\cU,g)$ with $\psi^*(\cU_{\rm int},g_{\rm int})$ naturally to obtain $(\cT,g)$ which forms a vacuum development as in Definition \ref{vac_dev}.

Now suppose $(\cT_i,g_i)$, $(i=1,2)$ are two vacuum developments for $(\Upsilon,{\bI}_i, {\bB}_i)$ with ${\bI}_i=(\g_i,\k_i),~\bB_i=([\s_i],H_i)$; and $(\bI_1,\bB_1),(\bI_2,\bB_2)$ are equivalent. After shrinking the domain $\cT_i$ in time, we can assume $[\s_1]=[\s_2],~H_1=H_2$ on $\cC_0\cap \phi_{g_1}(\cT_1)\cap\phi_{g_2}(\cT_2)$. First extend $\psi$ to a diffeomorphism $\Psi_1$ on $M$ and set $\w g_2 = \Psi_1^*g_2$ (well-defined on $\Psi_1^{-1}(\cT_2)$). Then according to \eqref{equif} the unique associated gauge for $\w g_2$ is given by $\f_{\w g_2}=\f_{g_2}\circ\Psi_1$ in some ST-corner neighborhood $\cU$ covering the boundary $\cC\cap\Psi_1^{-1}(\cT_2)$. It is then easy to verify that $(\w \cT_2=\Psi_1^{-1}(\cT_2),~\w g_2)$ is a vacuum development for $(\w\bI_2,\w\bB_2)$ with $\w\bI_2=(\psi^*\g_2,\psi^*\k_2)$ and $\w\bB_2=([\s_2],H_2)$. Based on the equalities above, we obtain $\w\bI_2=\bI_1$ on $S$ and $\w\bB_2=\bB_1$ on $\cC_0\cap \phi_{g_1}(\cT_1)\cap \phi_{\w g_2}(\w\cT_2)$.

Now , $(\cT_1,g_1)$ and $(\w\cT_2,\w g_2)$ are two vacuum developments of the same initial-boundary data. Then by Theorem \ref{corner_exist2} there is an ST-corner neighborhood $\cU$ so that $(\f_{\w g_2}^{-1}\circ \f_{g_1})^*\w g_2=g_1$ in $\cU$. By standard uniqueness results in the solution of the Cauchy problem on $(S, {\mathbb I})$, there is a neighborhood $\cU_{S}$ of the initial surface $S$ and a diffeomorphism $\Psi_2$ fixing $S$ such that $\Psi_2^*\w g_2=g_1$ on $\cU_S$. Observe that in the overlap $\cU\cap\cU_{S}$, the maps $(\f_{\w g_2}^{-1}\circ \f_{g_1})$ and $\Psi_2$ both equal the identity on $S\cap (U\cap\cU_{S})$ and push forward the unit timelike normal vector $T_{g_1}$ to $T_{\w g_2}$. In addition they both pull back the metric $\w g_2$ to $g_1$. It follows that $\f_{\w g_2}^{-1}\circ \f_{g_1}=\Psi_2$ in $\cU\cap\cU_{S}$. The map $\Psi_2$ may thus be naturally extended to a map $\Psi_2: \cT'_1=\cU\cup\cU_S \to M$ which is a diffeomorphism onto its image $\cT'=\Psi_2(\cT_1')$ and which fixes the initial surface $S$. Hence $g_1$ and $\w g_2$ are related by $\Psi_2:\cT'_1\to\cT'$, and it follows naturally that $g_1$ and $g_2$ are related by $\Psi_1\circ\Psi_2:\cT'_1\to\cT'_2$. Moreover, in the neighborhood $\cU$ we have $\Psi_1\circ\Psi_2=\Psi_1\circ\phi_{\w g_2}^{-1}\circ\phi_{g_1}=\phi_{g_1}^{-1}\circ\phi_{g_1}$.
\end{proof}

\begin{remark}\label{G_gF}
{\rm The proof of Theorem 5.7 also holds for the systems $(g, F)$ in \eqref{main}-\eqref{b1} and \eqref{main2}-\eqref{b2}. To see this, note that it is straightforward to extend the existence of a maximal Cauchy development $(M_S, g)$ of initial data $(S, \bI)$ to existence of a maximal Cauchy development $(M_S, g, F)$ where $F$ is a wave map as in \eqref{wave} satisfying initial conditions ${\bf I}$ as in \eqref{I1}. The proof of Theorem 5.7 for triples $(\cT, g, F)$ then proceeds in the same way.
}
\end{remark} 

\medskip

Next we turn to the proof of Theorem \ref{phase_space}; for simplicity, we work only in the $C^{\infty}$ setting. Using the notation of the Introduction, we begin with the first part of Theorem \ref{phase_space}.

\begin{proposition}\label{coner_cmpt} Associated to each smooth section $\Lambda$ of the fibration $\pi:\chi(\cC_0) \to \cJ({\Si_0})$, there is a bijection 
\be \label{mod4}
D_{\Lambda}: \cM \to (\cI \times_c \cB) \times \cJ({\Si_0}).
\ee
\end{proposition}
\begin{proof}
Note that given a spacetime $(\cT,g)$, the defining system \eqref{mainf3}-\eqref{bf3} for the associated gauge $\phi_g$ is a well-defined IBVP which will admit a unique solution if and only if the initial-boundary data satisfy compatibility conditions along the corner. These conditions define the space $\cJ({\Si_0})$ which we first discuss in more detail. Given this, the map $D_{\Lambda}$ is then  essentially simply an evaluation or restriction map. 

To begin, it is obvious that, in the system \eqref{mainf3}-\eqref{bf3}, the $C^0$ compatibility condition is always satisfied since $E_{g_S}$ maps $\Si$ to $\Si_0$ on the initial surface and $r_0\circ\phi_g=0$ on $\cC$. 

For the $C^1$ compatibility condition, first observe that the boundary condition $r_0\circ\phi_g=0$ is consistent with the initial condition $(\phi_g)_*(N_g)=T_{g_R}$ along the corner, because $\phi_g$ mapping $\cC\cap\cT$ to $\cC_0$ implies that $(\phi_g)_*$ must push forward vectors tangent to $\cC\cap\cT$ to vectors tangent to $\cC_0$. The $C^1$ compatibility condition requires that \eqref{corner1} holds, where $\ell$ is given by (cf. \S 6.3)
\be\label{corner2} 
\ell=1-\frac{g(N_g,n_{g_S})}{\sqrt{1+g(N_g,n_{g_S})^2}}\ \ \mbox{ on }\Si_0.
\ee
Observe that the right side of equation above is a diffeomorphism invariant value, i.e. the value of $\ell$ on $\Si_0$ only depends on the equivalence class of $g$. 

The $C^2$ compatibility condition is given by the wave equation \eqref{mainf3} along the corner. If we choose a local chart $\{x^\a\}~(\a=0,1,2,3)$ near a corner point $p\in\Si_0$ with $x^0=t_0$ and $\nu_{g_R}=\partial_{x^1}$, and choose a local chart $\{y^\mu\}~(\mu=0,1,2,3)$ near the corner point $E_{g_\Si}^{-1}(p)\in\Si$, then from the wave equation 
$\Box_g \f_g + \Gamma(\f_g)g(\nabla \f_g, \nabla \f_g)= 0$
we can derive that 
$$T_g^c(T_g^c(\phi_g^\a))=P^\a(\partial g_{\mu\nu})~\mbox{ on }\Si_0~(\a=0,2,3)$$
where $P^\a(\partial g_{\mu\nu})$ is a linear function in $\partial g_{\mu\nu}$ whose coefficients are given by functions of $g_{\mu\nu},E_{g_S},g_R$.
On the other hand, let $T^c_{g_R}$ be the field of unit normal vectors to level sets $\Si_\tau$ in the ambient manifold $(\cC_0,(g_R)|_{\cC_0})$. Given $\Theta_\cC$ on $\cC_0$, we can calculate the Lie derivative $\cL_{T^c_{g_R}}\Theta_\cC$ and obtain (notice that by definition $\phi_g(T_g^c)=T^c_{g_R}$ on $\Si$):
$$(\cL_{T^c_{g_R}}\Theta_\cC)^\a=\ell T_g^c(T_g^c(\phi_g^\a))+Q^\a(\partial g_{\mu\nu})~\mbox{ on }\Si_0~(\a=0,2,3)$$
where $Q^\a(\partial g_{\mu\nu})$ is a linear function in $\partial g_{\mu\nu}$ with coefficients given by functions of $g_{\mu\nu},E_{g_S},g_R$.
Thus $C^2$ compatibility condition implies:
\be\label{corner4}
(\cL_{T^c_{g_R}}\Theta_\cC)^\a=\ell P^\a(\partial g_{\mu\nu})+Q^\a(\partial g_{\mu\nu})~\mbox{ on }\Si_0~(\a=0,2,3).
\ee
Although expressed in local charts, this is a geometric (i.e. diffeomorphism invariant) corner condition on $g$, because both the wave equation \eqref{mainf3} and the Lie derivative $\cL_{T^c_{g_R}}\Theta_\cC$ are geometric. 
By taking higher order Lie derivatives of $\Theta_\cC$ along $T_{g_R}^c$ and derivatives of the wave equation \eqref{mainf3}, we can obtain higher order derivatives of $\Theta_\cC$ determined by $g$ based on $C^k~(k\geq 2)$ compatibility conditions. For simplicity, we write them in a generalized form of equation \eqref{corner4}:
\be\label{corner}
(\cL^{k}_{T^c_{g_R}}\Theta_\cC)^\a={\bf P}^\a_k(\partial^{k} g,\partial^{k-1} g,...,\partial g)~(\a=0,2,3, k\geq 1)
\ \ \mbox{ on } \Si_0,
\ee
where $\cL^{k}_{T^c_{g_R}}$ is the $k$-fold Lie derivative with respect to $T^c_{g_R}$ and ${\bf P}^\a_k$ is a polynomial whose coefficients are determined by $g_{\mu\nu},E_{g_S},g_R$. This describes the space $\cJ({\Si_0})$. 

Next, let $\Lambda: \cJ({\Si_0}) \to \chi(\cC_0)$ be a section of the fibration $\pi$ so that $\Lambda$ assigns to each jet $J\in\cJ({\Si_0})$ a smooth vector field $\Theta_\cC$ on $\cC_0$ extending $J$. 
Then given a smooth spacetime $(\cT,g)$, we can first identify the $C^{\infty}$ jet $J=J(g)$ in $\cJ({\Si_0})$ whose information at $\Si_0$ is fully determined by the equations \eqref{corner2} and \eqref{corner}. 
Next the map $\Lambda$ is used to obtain a vector field $\Theta_\cC=\Lambda\big(J(g)\big)$ which is then used to construct the {associated gauge} $\phi_g$ for $g$ in the boundary gauge $\Theta_\cC$. The map $D_{\Lambda}$ in \eqref{mod4} is then simply defined by 
$$D_{\Lambda}:~\{(\cT,g)\} \to \big( \{ (g_S,K_{g|_S}),([(\phi_g^{-1})^*g^\intercal],H_{(\phi_g^{-1})^*g}) \},~J(g)\big)$$
where the boundary data  $([(\phi_g^{-1})^*g^\intercal],H_{(\phi_g^{-1})^*g})$ is read off from $(\cT,g)$ via the associated gauge $\phi_g$ constructed as above. Based on the geometric uniqueness result in Theorem \eqref{exist_unique}, $D_\Lambda$ is well-defined.

To define the inverse map $D^{-1}_\Lambda$, for a given element $\big(\{(\bI,\bB)\},J\big)\in  (\cI\times_c\cB)\times \cJ({\Si_0})$, we first use the map $\Lambda$ to obtain a vector field $\Theta_\cC=\Lambda(J)$. This gives then a vacuum development $(\cT,g)$ from $(\bI,\bB,\Theta_\cC)$ via the existence result in Theorem \ref{exist_unique}. Now let $D^{-1}_\Lambda (\{(\bI,\bB)\},J)$ be the equivalence class of $(\cT,g)$ in $\cM$. Again, according to the geometric uniqueness result in Theorem \ref{exist_unique}, $D^{-1}_\Lambda$ is well-defined.
\end{proof}

{Note the the corner angle along $\Si$ in the solution space $(\cT,g)\in D^{-1}_\Lambda (\{(\bI,\bB)\},J)$ is determined by the first component in the jet $J$. In more detail, based on equation \eqref{corner2}, we see that $g(N_g, n_{g_S})=\tfrac{\ell^2}{1-\ell^2}$. So as long as $\ell^2<\frac{1}{2}$, the boundary $\cC\cap\cT$ will be timelike.}

Given a representative element $(\bI,\bB,J)$, the existence of a maximal vacuum development for $\big(\bI,\bB,\Theta_\cC=\Lambda(J)\big)$ will be discussed in detail below. 

Finally, we note that analogs of Theorems \ref{exist} and \ref{phase_space}, with the same proofs, also hold with respect to ${\bf B}$ boundary data, where one replaces the boundary data $([\s], H)$ by $([\s], \eta)$, where $\eta$ is to prescribe the normal component 
$$g_R((\f_g)_*(T_g + \nu_g), \nu_{g_R})=\eta.$$
The proof is the same. 

\subsection{Maximal vacuum developments}
Theorem \ref{exist_unique} is an analog of the situation for vacuum developments of Cauchy data $(S, {\mathbb I})$ on an initial time surface, modulo the non-local nature of the boundary conditions. In particular, it is now natural to consider the existence and uniqueness of a maximal development. We concentrate below on maximal vacuum developments of initial-boundary data $({\mathbb I},\bB)$, {with a fixed choice of $\Theta_\cC$}; the proof in the case of the expanded system $(g,F)$ is essentially the same and slightly simpler. The proof will proceed along the same lines as in \cite{Choquet-Geroch:1969}, following the exposition in \cite{Hawking-Ellis:1973}, cf.~also especially \cite{Ringstrom:2009}, 
(as well as \cite{Sbierski:2016}, \cite{Wong:2013} for related but distinct approaches). 

To begin, as with the Cauchy problem, we pass to the abstract setting and will include both future and past developments. 
In the following discussion, we fix $M_0= \bR\times S$ with a time function $t_0$, and a Riemmanian metric $g_R$ on it.
The boundary data $\bB$ are now defined on $\cC_0 = \bR \times \Si$ (and not $[0, \infty)\times \Si$ as before). As previously, we use 
$S_{\tau}$ and $\Si_{\tau}$ to denote the level sets of $t_0$ on $M_0$ and $\cC_0$. The initial-boundary surface $\Upsilon$ is now given by $S_0 \cup \cC_0$, 
with corner $\Si_0=\{0\}\times \Si$. For simplicity, we work in the $C^{\infty}$ setting in the analysis below. From now on, we use $(M,g)$ to denote abstract vacuum developments defined in the following. 

\begin{definition} \label{avd}
An \textit{(abstract) vacuum development} for the initial-boundary data $(\Upsilon,\mathbb{I},\bB)$ is a manifold $M\cong (-1,1)\times S$ with boundary $M_\cC\cong(-1,1)\times\Si$, equipped with a Lorentz metric $g$ on $M$ such that:
\begin{enumerate}
{\item $g$ is Ricci-flat on $M$.}
\item $(M, g)$ is a globally hyperbolic spacetime with timelike boundary admitting {a Cauchy hypersurface-with-boundary $S$}. 
\item There is an embedding $\iota: \Upsilon' \to M$ of some subdomain $\Upsilon'\subset \Upsilon$ such that $\iota(S_0) = S$, $\iota(\cC_0\cap \Upsilon') =  M_\cC$.
\item $g$ induces $g_S=\g,~K_{g|S}=\k$ with the initial data $(\g,\k)$ on $S$ induced from $\iota(\bI)$; 
\item There is an open neighborhood $\cU$ of $ M_\cC\subset M$ so that the unique associated gauge $\f_g$ for $g$ on $\cU$ is a diffeomorphism onto its image in $M_0$;
and $\big([\big((\phi_g^{-1})^*g\big)^\intercal],H_{(\phi_g^{-1})^*g}\big)=\bB$ on $\cC_0\cap\phi_g(M_\cC)$.
\item $\phi_g^*(t_0)$ is a time function of $(M,g)$ near $M_\cC$ and the timelike boundary $M_\cC =\f_g^{-1}(\{p\in\cC_0:t_0(p)\in(\tau_1,\tau_2)\})$ for some $\tau_1<0,\tau_2>0$.  

\end{enumerate}
\end{definition}

Here we regard $\f_g$ as a map from a neighborhood of $M_\cC$ in the (abstract) manifold $M$ to a neighborhood of $\cC_0$ in the fixed manifold $(M_0,t_0,g_R)$, determined by $g$ via \eqref{mainf3}-\eqref{bf3}. In the following we use $\Upsilon\cap M$ to denote both the subset $\Upsilon'\subset \Upsilon$ and the image $\iota(\Upsilon')\subset M$ which can be identified via the embedding $\iota$. 

Globally hyperbolic manifolds with timelike boundary are defined in the same way as globally hyperbolic manifolds (without boundary) and have the same essential properties, cf.~\cite{Ake-Flores-Sanchez:2021} for a recent analysis. In particular, inextendible timelike curves intersect the Cauchy surface $S$ exactly once and $M$ is diffeomorphic to {$M\cong(-1,1)\times S$.}
Let ${\rm cl}(M)$ denote the closure of the manifold $M$. Note that the boundary $\p [{\rm cl}(M)]$ consists of the timelike boundary $M_\cC\cong (-1,1)\times\Si$ and the \textit{edge} $M_\cE\cong (\{1\}\times S)\cup (\{-1\}\times S)$. By definition $M$ includes its timelike boundary $ M_\cC \subset M$, but not the edge $M_\cE$. Further, one has 
\be \label{dev}  
M = \cD(M_\cC\cup S), 
\ee
where $\cD$ is the full (future and past) domain of dependence. 
Taking the union of both future and past vacuum developments in Theorem \ref{exist_unique} shows that any initial-boundary data set $(\Upsilon, {\mathbb I}, \bB)$ admits an (abstract) vacuum development. 

We now turn to the existence of maximal developments. First we give a precise definition of extension.
\begin{definition}
An \textit{extension} of the vacuum development $(M,g)$ of the initial-boundary data set $(\Upsilon,\mathbb{I},\bB)$ is a development $(M',g')$ of the same initial-boundary data set such that there exists an isometric embedding $\Psi:(M,g)\to (M',g')$ with $\Psi|_{S\cap M}={\rm Id}_{S\cap M'},~\Psi|_{ M_\cC}=\f_{g'}^{-1}\circ\f_g|_{ M_\cC}$.
\end{definition}

\begin{lemma}\label{uni_ext}
The isometric embedding from a vacuum development to its extension is unique.
\end{lemma}
\begin{proof}
The proof is the standard one from \cite{Choquet-Geroch:1969}. Suppose $(M',g')$ is an extension of $(M,g)$ with embedding $\Psi:(M,g)\to (M',g')$. 
Take any point $p \in M$ and let $\s=\s(s)$ be an inextendible timelike geodesic in $M$ starting from $p$. By \eqref{dev}, $\s$ must hit the hypersurface $\Upsilon\cap M$ at a unique point $q$, for which the image $\Psi(q)$ is uniquely determined by $\Psi|_{\Upsilon\cap M}$. The length or proper time $\ell(\s)$ and angle $\a$ between $\tfrac{d}{ds}\s$ and the tangent space of $\Upsilon$ at $q$ uniquely determine $\s$. This data is preserved under an isometric embedding. Since the point $\Psi(p)$ is uniquely determined by this data and $\Psi(q)$, the embedding $\Psi$ is unique. 
\end{proof}

\begin{theorem}\label{max_dev}
Given an initial-boundary data set $(\Upsilon,{\mathbb I},\bB)$, up to isometry there exists a unique maximal vacuum development $(\w M,\w g)$. The development  $(\w M,\w g)$ is an extension of any other vacuum development of $(\Upsilon,{\mathbb I},\bB)$.
\end{theorem}
\begin{proof} The proof follows closely that in \cite{Choquet-Geroch:1969}, cf.~also \cite{Hawking-Ellis:1973} and \cite{Ringstrom:2013-2}. 

Let $\cM(\Upsilon, {\mathbb I}, \bB)$ be the set of all vacuum developments of a given initial-boundary data set $(\Upsilon, {\mathbb I}, \bB)$. By Theorem \ref{exist_unique}, $\cM(\Upsilon, {\mathbb I}, \bB)$ is nonempty. This set is partially ordered by the extension relation; $M_1 \leq M_2$ if $M_2$ is an extension of $M_1$. 
If $\{M_{n}\}$ is a totally ordered subset, then the uniqueness from Lemma \ref{uni_ext} implies that the union $\cup M_{n}$ is also a vacuum development which is clearly an upper bound for $\{M_{n}\}$. It follows from Zorn's Lemma that $\cM(\Upsilon, {\mathbb I}, \bB)$ has a maximal element $(\w M, \w g)$. Any extension of $(\w M,\w g)$ thus equals $(\w M,\w g)$. 

The main issue is to prove uniqueness. Suppose $(M', g')$ is another vacuum development of $(\Upsilon,{\mathbb I},\bB)$; we need to prove $\w M$ is an extension of $M'$. 

By Theorem \ref{exist_unique}, $\w M$ and $M'$ must be extensions of some common vacuum development which satisfies all the conditions in Definition \ref{avd} and can be embedded into $\w M$ (and $M'$) as an open subset. Consider then the set $C(\w M, M')$ of all such common sub-developments of $\w M$ and $M'$. This set is again partially ordered by the extension relation and hence {\it it} has a maximal element $(\hat M, \hat g)$, with isometric embeddings into $(\w M, \w g)$ and $(M',g')$. Without loss of generality, we can assume $(\hat M,\hat g)$ is an open subset of $(\w M,\w g)$ and  $\psi: (\hat M, \hat g) \to (M', g')$ is an isometric embedding. 

One then forms the disjoint union $\w M \cup M'$ and divides by the equivalence relation 
$$M^{+} = \w M \cup  M' / \sim,$$
where $p\sim p'$ if $ p \in \hat M\subset \w M$ and $p'=\psi(p)$. Thus one is gluing the spaces $\w M$ and $M'$ together along their common isometrically embedded subspace $\hat M$. There is a well-defined vacuum spacetime metric on $M^+$ formed by $\w g$ and $g'$.

The main claim is that $M^{+}$ is Hausdorff. Given this, the spacetime $M^+$ is then a vacuum development of $(\Upsilon, {\mathbb I}, \bB)$ which is an extension of both $\w M$ and $M'$. Since the only extension of $\w M$ is $\w M$ itself, it follows that $\w M=M^+$ and hence $\w M$ is an extension of $M'$. This proves the uniqueness and the fact that any vacuum development has an extension to the maximal development $\w M$. 

The proof of the Hausdorff property is by contradiction. If $M^+$ is not Hausdorff, then non-Hausdorff points of $M^+$ must come from points on the \textit{edge} of $\hat M$, i.e. $\hat M_\cE=\p[{\rm cl}( \hat M)]\setminus \hat M_\cC$, {with $\hat M_\cC$ denoting the timelike boundary $\hat M_\cC\cong (-1,1)\times\Sigma$ of $\hat M$. }
A point $\w p$ is a non-Hausdorff point if $\w p\in \hat M_\cE\cap\w M$ and there is a corresponding point $p'\in\p [{\rm cl}\big(\psi(\hat M)\big)]\cap M'$ such that every neighborhood $U\subset \w M$ of $\w p$ has the property that the closure ${\rm cl}\big(\psi(U)\big)\subset M'$ contains $p'$. Following the same argument as in \cite{Hawking-Ellis:1973}, one sees that for a non-Hausdorff point $\w p\in\w M$ the associated point, denoted as $H(\w p)\in M'$, is unique and the set $\cH$ of all non-Hausdorff points is open in $\hat M_\cE$. Furthermore, if $\w p\in\cH$, then there exist neighborhoods $\w U$ of $\w p$ in $\w M$ and $U'$ of $H(\w p)$ in $M'$ such that $\psi$ maps $\w U\cap \hat M$ to $U'\cap\psi(\hat M)$ and it can be extended to a diffeomorphism $\Psi:\w U\to U'$.

Let $\w M_{\rm int}$ denote the interior of $\w M$. Based on \cite{Choquet-Geroch:1969}, if one can find a point $\w p\in\hat M_\cE\cap\w M_{\rm int}$ and a spacelike hypersurface $S_p$ passing through $\w p$ such that $S_p-\{\w p\}\subset\hat M$, then the maximal common sub-development $\hat M$ can be extended larger by a further common Cauchy development starting from $S_p\subset\w M$ and $\Psi(S_p)\subset M'$, and hence one arrives at a contradiction. This argument will be used frequently below. 

Note that based on the last condition in Definition \ref{vac_dev}, one has  
$\w M_\cC=\f_{\w g}^{-1}(\{p\in\cC_0:t_0(p)\in(\w \tau_1,\w \tau_2)\})$,
$\hat M_\cC=\f_{\w g}^{-1}(\{p\in\cC_0:t_0(p)\in(\hat\tau_1,\hat\tau_2)\})$, and
$ M'_\cC=\f_{g'}^{-1}(\{p\in\cC_0:t_0(p)\in(\tau'_1,\tau'_2)\})$, for some $\w\tau_1,\hat\tau_1,\tau'_1<0$ and $\w\tau_2,\hat\tau_2,\tau'_2>0$.
We first show that $\hat\tau_2=\min(\w\tau_2,\tau'_2)$; the same argument yields $\hat\tau_1=\max(\w\tau_1,\tau'_1)$. This is equivalent to the fact that any point on the edge of $\hat M_\cC$, i.e. $\cE(\hat M_\cC)=\hat M_\cE\cap {\rm cl}(\hat M_\cC)$, cannot be non-Hausdorff. 

If the above is not true, without loss of generality, we can assume $\hat\tau_2<\w\tau_2\leq \tau'_2$. Then obviously every point of $\cE^{+}(\hat M_{\cC})=\phi_{\w g}^{-1}(\{p\in\cC_0:t_0(p)=\hat\tau_2\})$ is non-Hausdorff. Since non-Hausdorff is an open condition, there is an open neighborhood $V_\cH$ of $\cE^{+}(\hat M_\cC)$ in $\hat M_\cE$ where all points are non-Hausdorff, i.e. $V_\cH\subset\cH$. Let $\w t$ denote the time function $\w t=\phi_{\w g}^*(t_0)$. We claim that any point $p\in \big(V_\cH\setminus\cE^+(\hat M_\cC)\big)$ must have $\w t(p)>\hat\tau_2$. If not, there is a point $\w p\in \big(V_\cH\setminus\cE^+(\hat M_\cC)\big)$ with $\w t(\w p)\leq\hat \tau_2$, then consider an inextendible past-directed null geodesic $\w \s$ starting at $\w p$.
It follows from the globally hyperbolic property (cf. \cite{Ringstrom:2013-2}) that there is $a_0>0$ such that $\w \s([0,a_0])\in \cH\cap\hat M_\cE$ and $\w\s$ leaves $\hat M_\cE$ at the end point $\w q=\w\s(a_0)$. Since $\w t(\w q)<\hat \tau_2$, one has  $\w q\in\hat M_\cE\cap \w M_{\rm int}$. Then based on the analysis in \cite{Ringstrom:2013-2}, there is a point $q\in\hat M_\cE\cap\w M_{\rm int}$ close to $\w q$ and a spacelike hypersurface $S_{q}$ containing $q$ such that $S_q-\{q\}\subset\hat M$. This gives the same contradiction as in the argument above.

Hence all points $p\in \big(V_\cH\setminus\cE^+(\hat M_\cC)\big)$ satisfy $\w t(p)>\hat\tau_2$. It follows that a tubular neighborhood $\w U$ of the corner $\w\Si_{\hat \tau_2}=\phi_{\w g}^{-1}(\{p\in \cC_0: t_0(p)=\hat\tau_2\})$ in the spacelike hypersurface $\w S_{\hat\tau_2}=\phi_{\w g}^{-1}(\{p\in M_0: t_0(p)=\hat\tau_2\})$ is contained in $\hat M$, i.e. $\w U\setminus\w\Si_{\hat\tau_2}\subset \hat M$. Using the correspondence between non-Hausdorff points $\w p\in\w M$ and $H(\w p)\in M'$, there is a tubular neighborhood $U'$ of the corner $\Si'_{\hat \tau_2}=\phi_{g'}^{-1}(\{p\in \cC_0: t_0(p)=\hat\tau_2\})$ in the spacelike hypersurface $S'_{\hat\tau_2}=\phi_{g'}^{-1}(\{p\in M_0: t_0(p)=\hat\tau_2\})$ such that $U'\setminus\Si'_{\hat\tau_2}\subset\psi(\hat M)$.
Thus, near the corner $\w\Si_{\hat\tau_2}\subset \w M$ and $\Si'_{\hat\tau_2}\subset M'$, $g$ and $g'$ have equivalent initial data on $\w U$ and $U'$ as well as equivalent boundary data in a neighborhood of $\w\Si_{\hat\tau_2}\subset \w M_\cC$ and $\Si'_{\hat\tau_2}\subset M'_\cC$. By the uniqueness in Theorem \ref{corner_exist2}, there is a common boundary vacuum development $M^*$ of $(\w M, g)$ and $(M',g')$ near the corner $\w\Si_{\hat\tau_2}$ and $\Si'_{\hat\tau_2}$. 
Joining $M^*$ with $\hat M$ gives an extension of $\hat M$ of which both $\w M$ and $M'$ are extensions. 
This contradicts the maximality of $\hat M \in C(\w M, M')$. 

The analysis above proves that all non-Hausdorff points must be in $\hat M_\cE\cap \w M_{\rm int}$. One can now take $\w p \in \cH$ and consider an inextendible past-directed null geodesic starting at $\w p$. Based on the globally hyperbolic property again, $\w p$ must leave $\hat M_\cE$ at some point $\w q$ and moreover, according to \cite{Ringstrom:2013-2}, $\w q \in \cH$ and $\w q\in\hat M_\cE\cap \w M_{\rm int}$. Further analysis in \cite{Ringstrom:2013-2} implies that there is a point $q\in\hat M_\cE\cap\w M$ close to $\w q$ and a spacelike hypersurface $S_q$ passing through $q$ with $S_q-\{q\}\subset\hat M$. This leads to a contradiction again as mentioned at the beginning, which implies that $M^+$ must be Hausdorff. 
\end{proof}

We conclude the paper with a few final remarks.

For initial surface $S \subset \Upsilon$, let $\w M_{S}$ be the unique maximal Cauchy development of $S$. Clearly $\w M_{S} \subset \w M$ for the maximal development $\w M$ constructed in the proof above. The existence of boundary vacuum developments $\cU \subset \w M$ as in Theorem \ref{corner_exist2} implies that in small neighborhoods $\cU$ of $\Si$, $\w M_S \cap \cU$ has a Cauchy horizon in $\cU$. In general, the ``boundary" of $\w M_S$ may be very complicated, consisting of regions where the solution $g$ has (curvature) singularities and is in general not well understood. Thus, the presence of boundary data near $\Si$ has the effect of regulating the metric near the boundary $\Si$. 

Theorem \ref{max_dev} differs from the situation of the maximal Cauchy development $\w M_S$ in that the maximal solution $(\w M, \w g)$ is maximal with respect to the choice of the preferred gauge $\f_g$. To discuss this}, let {$\tau_0 = \sup\{\tau: t_0\circ\phi_{\w g}(p)=\tau \mbox{ for some } p\in \w M_\cC\}$}; $\tau_0$ is the maximal time of existence of the solution $\w M$ at the boundary, measured in the time coordinate $t_0$. The solution $(\w M, \w g)$ may break down or degenerate at $\tau_0$ in two ways. One way is that the metric $\w g$ becomes degenerate so we cannot extend the solution further. 
On the other hand, it may happen that the solution $(\w M, \w g)$ breaks down only because the associated gauge $\f_{\w g}$ becomes degenerate (i.e.~is no longer a diffeomorphism) at $\tau_0$, {possibly causing a breakdown in the foliation $\phi_{\w g}^{-1}(S_\tau)$ for $\w g$ as $\tau \to \tau_0$. In this latter case it may be possible to extend $(\w M,\w g)$ to a larger domain by defining new initial data $E_{\w g_{S'}}$ on a new partial initial data set $S' \subset \w M$ with $\partial S'$ near $\w M_\cC$ and solving, near $S'$, for the vacuum metric $\w g'$ with a new preferred gauge $\f'$ determined by $\w g$ and $E_{\w g_{S'}}$ on $S'$, i.e.~by changing the associated map for $\w g$.  Of course this also requires appropriate changes on the boundary conditions \eqref{bg3} (which remains open). In addition, it may be possible that such an extension need only be done near certain regions of  $ \w M_\cC$. }

\section{Appendix} 
In this section we collect a number of results and formulas (mostly standard) used in the main text. 
\subsection{Boundary Conditions and Energy Estimates}
In this subsection, we summarize the energy estimates for the IBVP for a scalar wave equation on a Minkowski half-space with Sommerfeld, Dirichlet and also Neumann boundary data. These estimates are basically well-known and included for completeness.     

Consider the scalar wave equation 
\be \label{toy2} 
\Box_{g_0}u = \f.
\ee 
on the region ${\bf R} = [0,\infty)\times (\bR^3)^+$ of Minkowski spacetime with standard coordinates $(t, x^i)$. The stress-energy tensor $S$ of $u$ is given by  
$$S = du^{2} - {\tfrac{1}{2}}|du|^{2}g,$$
As is well-known, the symmetric bilinear form $S$ is conserved on-shell, i.e.~if $u$ solves the equation of motion \eqref{toy2}, then 
$$\d S = -\Box_{g_{0}} u du = -\f du,$$
(cf.~\cite{Hawking-Ellis:1973} for example). For any smooth vector field $Z$, one then has  
$$\d(S(Z)) = (\d S)(Z) +\<S, \d^{*}Z\> = - \f Z(u) + \<S, \d^{*}Z\>.$$
Let $U$ be any open domain in ${\bf R}$ with compact closure and piecewise smooth boundary $\p U$. Applying the divergence theorem to the left side then gives 
\be \label{en}
\int_{\p U}S(Z, N) = \int_{U}\<S, \d^{*}Z\> - \f Z(u),
\ee
where $N$ is the outward $g$-unit normal at the boundary. The equation \eqref{en} leads to the basic energy estimates. 

Let 
$$E_{S_t}(u) = {\tfrac{1}{2}}\int_{S_t}u_t^2 + |du|^2,$$
where $du$ is the full spatial derivative. Here and below, integration is with respect to the standard measures. As in the main text, let $S_t$ be the level set of $t$, $\cC_s = \{x^1 = 0\} \cap \{t \in [0,s]\}$, $\Si_t = \cC \cap S_t$ and $M_{s} = \{t \in [0,s]\}$. Also for this section, let $x = x^1$,  and $(x^2, x^3) = (y,z)$.   
  
Consider first $Z = \p_t$. Then $\d^{*}Z = 0$ and one obtains from \eqref{en} 
\be \label{enT}
\frac{d}{dt}E_{S_t}(u) + \int_{S_t}\f u_t = \int_{\Si_t}u_x u_t .
\ee 
For $\f = 0$, this immediately gives the relation 
$$E_{S_t}(u) =  E_{S_0}(u) + \int_{\cC_t}u_xu_t.$$
For general $\f$, since $|\f u_t| \leq \frac{1}{2}(u_t^2 + \f^2)$, one has 
$$E_{S_t}(u) \leq E_{S_0}(u) + \int_{0}^{t}E_{S_s}(u) ds + {\tfrac{1}{2}}\int_{M_t}\f^2 + \int_{\cC_t}u_xu_t.$$
The integral form of the standard Gronwall inequality then gives the bound 
\be \label{gron}
E_{S_t}(u) \leq E_{S_0}(u) + Ce^t[\int_{M_t}\f^2 +  \int_{\cC_t}u_xu_t].
\ee
For data $\f$ of compact support, the factor $e^t$ may be absorbed into the constant $C$. 

Next for $Z = \p_x$ again first with $\f = 0$, \eqref{en} gives  
\be \label{enX}
\frac{d}{dt}\int_{S_t}u_x u_t  =\int_{\Si_t}u_x^2 - {\tfrac{1}{2}}|du|^{2} = {\tfrac{1}{2}}\int_{\Si_t}u_x^2 + u_t^2 - |d_A u|^2.
\ee
Thus  
\be \label{tang}
{\tfrac{1}{2}}\int_{\cC_t}|d_A u|^2 = {\tfrac{1}{2}}\int_{\cC_t} u_x^2 + u_t^2 - \int_{S_t}u_xu_t \leq 
{\tfrac{1}{2}}\int_{\cC_t} u_x^2 + u_t^2 + E_{S_t}(u) + E_{S_0}(u) .
\ee
For the inhomogeneous equation, using \eqref{gron} one obtains in the same way that 
\be \label{tang2}
{\tfrac{1}{2}}\int_{\cC_t}|d_A u|^2 \leq 
{\tfrac{1}{2}}\int_{\cC_t} u_x^2 + u_t^2 + C(E_{S_t}(u) + E_{S_0}(u) + \int_{M_t}\f^2),
\ee
with again $C$ depending only on $t$. 
 
\medskip 

{\bf Sommerfeld Boundary data:} This is boundary data of the form 
\be \label{somm}
u_{t} + u_{x} = b,
\ee
where $b$ is a given function on the boundary cylinder $\cC$. 
Then $u_{x} = b - u_{t}$ so that 
$$\int_{\Si_t}u_{t}u_{x} = -\int_{\Si_t}u_{t}^{2} + \int_{\Si_t}bu_{t}.$$
Since $|bu_t| \leq \frac{1}{2}(u_t^2 + b^2)$, we obtain from \eqref{enT}
\be \label{e1}
\frac{d}{dt}\int_{S_t}u_{t}^{2} + |du|^{2} + \int_{\Si_t}u_{t}^{2} \leq \int_{\Si_t}b^{2},
\ee
giving the basic energy estimate 
$$E_{S_t}(u) + \int_{\cC_t}u_t^2 \leq E_{S_0}(u) + \int_{\cC_t}b^2.$$
To extend this to a strong or boundary stable estimate, note that $u_x^2 \leq 2(u_t^2 + b^2)$, so that 
$u_t^2 +u_x^2 \leq 3u_t^2 + 2b^2$. Substituting this in \eqref{e1} gives 
$$\frac{d}{dt}\int_{S_t}u_{t}^{2} + |du|^{2} + {\tfrac{1}{3}}\int_{\Si_t}u_{t}^{2} + u_x^2 \leq 2\int_{\Si_t}b^{2}.$$
Using the relation \eqref{tang}, one easily derives that 
$$\int_{S_t}u_{t}^{2} + |du|^{2} + {\tfrac{1}{4}}\int_{\cC_t}u_{t}^{2} + u_x^2 + |d_A u|^2 \leq E_{S_0}(t) + 3\int_{\cC_t}b^{2},$$
for solutions $u$ of \eqref{toy2} with $\f = 0$. When $\f \neq 0$, using \eqref{tang2}, the same arguments give 
$$\int_{S_t}u_{t}^{2} + |du|^{2} + {\tfrac{1}{4}}\int_{\cC_t}u_{t}^{2} + u_x^2 + |d_A u|^2 \leq E_{S_0}(u) + C[\int_{\cC_t}b^{2} + 
\int_{M_t}\f^2].$$
As is well-known, this estimate can be promoted to a full energy estimate, i.e.~including the $L^2$ norm of $u$, by noting that if $u$ satisfies \eqref{toy2}, then $v = e^{ct}u$ satisfies 
$$(\Box + c^2)v = \f v + 2cv_t.$$
The same arguments then give an energy estimate for $v$ including the $L^2$ norm, which then translates to a similar energy estimate for $u$. In sum and in the notation of \S 3, this gives the strong or boundary stable $H^1$ estimate
\be \label{e2}
\cE_{S_t}(u) + {\tfrac{1}{2}}\cE_{\cC_t}(u) \leq E_{S_0}(u) + C[\int_{\cC_t}b^2 + \int_{M_t}\f^2],
\ee
for solutions $u$ of \eqref{toy2} with Sommerfeld boundary condition. 

One obtains higher order $H^s$ energy estimates by simple differentiation. Thus, for $i = 0, 2,3$, so $\p_i$ is tangent to the boundary $\cC$, one has, for $u_i = \p_i u$, 
$$\Box_{g_{0}}u_i = \p_i \f,$$
and the boundary condition \eqref{somm} becomes 
$$(u_i)_t + (u_i)_x = \p_i b.$$
It follows that one has the $H^1$ energy estimate for each $u_i$, given $H^1$ control on $b$ and $\f$. For the normal derivative $u_x$, the bulk equation \eqref{toy2} gives $u_{xx} = \Box_{\cC} u - \f$. The term $\Box_{\cC} u$ is bounded in $L^2$ by the estimate above giving then an $L^2$ bound on $u_{xx}$, which gives then a full $H^2$ energy estimate. One continues in this way inductively for each $s$.

\medskip 

{\bf Dirichlet Boundary Data:} Here 
$$u = b$$ on 
$\cC$. In this context one has 
$$\int_{\Si_t}u_x u_t \leq \e \int_{\Si_t}u_x^2 + \e^{-1}\int_{\Si_t}u_{t}^2 = \e \int_{\Si_t}u_x^2 + \e^{-1}\int_{\Si_t}b_{t}^2,$$
so that to control $E_{S_t}(u)$, it suffices to control the Neumann derivative $u_x$. 
Also, as in \eqref{enX}, we have 
\be \label{tang3}
{\tfrac{1}{2}}\int_{\cC_t} u_x^2 + u_t^2 \leq E_{S_t}(u) + E_{S_0}(u) + {\tfrac{1}{2}}\int_{\cC_t}|d_A u|^2, 
\ee
so that 
\be \label{dn}
{\tfrac{1}{2}}\int_{\cC_t} u_x^2 \leq E_{S_t}(u) + E_{S_0}(u) + {\tfrac{1}{2}}\int_{\cC_t}|d_A b|^2.
\ee
The estimate \eqref{dn} shows that one can control Neumann boundary data of $u$ at $\cC$ in terms of Dirichlet control of $u$ on $\cC$ (and the energy). In other words, consider the Dirichlet-to-Neumann map $\cN(b) = u_x$, where $u$ is the unique solution to the IBVP \eqref{toy2} with Dirichlet boundary data $b$ and zero initial data. Then \eqref{dn} gives an $L^2(\cC)$ bound for $\cN$. This estimate is important for boundary stable energy estimates.  

The same arguments as above then give the energy estimate \eqref{e2} with Dirichlet boundary value $b$, with constants suitably adjusted. Similarly, in the same way as above, one obtains higher order $H^s$ energy estimates.  

\begin{remark}
{\rm An estimate of the form \eqref{dn} with Dirichlet and Neumann data interchanged, i.e.~an estimate of the form
$$\int_{\cC_t} |du|^2 \leq C[E_{S_t}(u) + E_{S_0}(u) + \int_{\cC_t} u_x^2]$$
does not hold, i.e.~Dirichlet data cannot be controlled by Neumann data at the same level of differentiability. There is a definite loss of regularity or diffentiability, cf.~\cite{Tataru:1998} for a detailed analysis. 

We note that in proving the well-posedness of the IBVP of quasi-linear systems such as those in \eqref{main3}-\eqref{b3}, it is important to have boundary stable energy estimates as in \eqref{e2}. 

}
\end{remark}

\subsection{Linearization Formulas}
In this subsection, for convenience we derive the formulas \eqref{trK}-\eqref{trtA} and \eqref{hamlin1}-\eqref{hamlin2}. Considering linearizations at the standard configuration $({\bf R}, \eta)$, the linearization of the normal vectors $T$ (normal to $S$) and $\nu$ (normal to $\cC$) are 
\begin{equation*}
\begin{split}
T'_h={\tfrac{1}{2}}h_{00}\partial_0-h_{0i}\partial_i, \ \ \nu'_h=h_{10}\partial_0-{\tfrac{1}{2}}h_{11}\partial_1-h_{1A}\partial_A.
\end{split}
\end{equation*}

For the second fundamental form $K=K_{g|S}$, one has $2K = \cL_T g$, so that $2K'_h = \cL_T h + \cL_{T'}g$. This gives 
$$2K'_h = \nabla_T h + dh_{00}\cdot dt_0 - 2dh_{0i}dx_0^i.$$
Taking the trace with respect to $\eta$ then gives \eqref{trK} as well as \eqref{trtK}. Replacing $T$ by $\nu$, similar computation gives \eqref{trA}-\eqref{trtA}. 

Next, the Hamiltonian constraint (Gauss equation) for a vacuum solution $Ric_g = 0$ on the timelike boundary $\cC$ is 
\be \label{Ham1}
\begin{split}
R_{\mathcal C}-(\tr_{\cC}K_{g|\cC})^2+|K_{g|\cC}|^2=0, 
\end{split}
\ee
where $K_{g|\cC}$ is the second fundamental form of $\cC\subset (U,g)$. For a linearization $h$ at the flat metric $\eta$ with $Ric'_h = 0$, it follows that 
\be \label{R=0}
(R_{\cC})'_h=0. 
\ee
It is standard that $(R_{\cC})'_h=-\Box_{\cC}(tr_{\cC}h)+\delta_{\cC}\delta_{\cC} h_{\cC}-g_{\cC}(Ric_{\cC}, h)$, which is computed as follows: 
\begin{equation*}
\begin{split}
(R_{\mathcal C})'_h&=-\Box_{\mathcal C}(-h_{00}+2\tau_h)+\partial_0\partial_0h_{00}-2\partial_0\partial_Ah_{0A}+\partial_A\partial_Bh_{AB}\\
&=\Delta_{\Si_t}h_{00}+\partial_0\partial_0(2\tau_h)-\Delta_{\Si_t}2\tau_h-2\partial_0\partial_Ah_{0A}+\partial_A\partial_Ah_{AA}+O_2\\
&=(\partial_0\partial_0-\partial_1\partial_1)h_{00}+\partial_0\partial_0(2\tau_h)-\Delta_{\Si_t}2\tau_h-2\partial_0X+\Delta_{\Si_t}\tau_h+O_2\\
&=(\partial_0\partial_0-\partial_1\partial_1)h_{00}+(\partial_0\partial_0+\partial_1\partial_1)\tau_h-2\partial_0X+O_2
\end{split}
\end{equation*}
Recall that $\tau_h=\tfrac{1}{2}(h_{22}+h_{33})$. In the above, we have used the facts that $h_{23}=O$, $h_{22} = \tau_h + O$, $h_{33} = \tau_h + O$, $\Box h_{00} = 0$, $\Box \tau_h = 0$, so that for instance 
$\Delta_{\Si_t}\tau_h=\partial_0\partial_0 \tau_h-\partial_1\partial_1 \tau_h$. Thus from \eqref{R=0} we obtain 
\be \label{Hamlin1}
(\partial_0\partial_0-\partial_1\partial_1)h_{00}+(\partial_0\partial_0+\partial_1\partial_1)\tau_h-2\partial_0X=O_2
\ee

Similarly, the Hamiltonian constraint or Gauss equation on the hypersurfaces $S_t = \{t=\text{constant}\}$ gives:
\be \label{Ham2}
R_{S_t}+(\tr_{S_t}K_{g|S_t})^2-|K_{g|S_t}|^2=0
\ee
The same analysis as above then gives 
\be \label{Hamlin2}
-\partial_1\partial_1\tau_h-\partial_0\partial_0\tau_h-\Delta_{\Si_t}h_{11}-2\partial_1\partial_Ah_{0A}=O_2.
\ee

\subsection{Geometry at the corner.}
In this last section, we derive the corner equations in the discussion of compatibility conditions. Let $g$ be a spacetime metric on $M$ with boundary hypersurfaces $S\cup\cC$ and corner $\Si$. Let $p\in\Si$ be an arbitrary corner point. Recall that, at $p$, $T_g^0$ denote the unit timelike normal vector to the initial surface $S\subset (M,g)$; $n_{g_S}$ denote the inward spacelike unit normal vector to $\Si\subset (S,g_S)$; $\nu_g$ is the spacelike outward unit normal to $\cC\subset (M,g)$; and $N_g$ (or $T_g^c$) is the timelike unit normal to $\Si\subset (\cC,g_\cC)$.

Notice that $T_g^0,\nu_g\in{\rm span}\{N_g,n_{g_S}\}$ at $p\in\Si$. Thus 
$$T_g^0=a_1 N_g+b_1 (-n_{g_S}),~\nu_g=a_2N_g+b_2(-n_{g_S})$$ 
for some constants $a_i,b_i~(i=1,2)$ at $p$. Applying the geometric identities $g(T_g^0,n_{g_S})=0,~g(T_g^0,T_g^0)=-1$, we obtain $a_1g(N_g,n_{g_S})-b_1=0$ and $-a_1^2+b_1^2-2a_1b_1g(N_g,n_{g_S})=-1$, which further imply that
$$a_1=\tfrac{1}{\sqrt{1+g(N_g,n_{g_S})^2}},~b_1=\tfrac{g(N_g,n_{g_S})}{\sqrt{1+g(N_g,n_{g_S})^2}}.$$
Applying the same calculation, we can also obtain 
$$a_2=-\tfrac{g(N_g,n_{g_S})}{\sqrt{1+g(N_g,n_{g_S})^2}},~b_2=\tfrac{1}{\sqrt{1+g(N_g,n_{g_S})^2}}.$$
Thus $T_g^0+\nu_g=\l_1N_g+\l_2(-n_{g_S})$ with $\l_1=\tfrac{1-g(N_g,n_{g_S})}{\sqrt{1+g(N_g,n_{g_S})^2}},~\l_2=\tfrac{1+g(N_g,n_{g_S})}{\sqrt{1+g(N_g,n_{g_S})^2}}$. Such linear relation is preserved by push forward via $F_*$, i.e. 
$$F_*(T_g+\nu_g)=\l_1 F_*(N_g)+\l_2F_*(-n_{g_S})~\mbox{ at }p\in\Si,$$ which further results in the compatibility equation \eqref{cib3} and the second equation in \eqref{cib4}.

In addition, since $T_g^c=N_g$ at $p$, we also have $T_g^c+\nu_g=\ell N_g +b_2(-n_{g_S})$ with $\ell=1-\tfrac{g(N_g,n_{g_S})}{\sqrt{1+g(N_g,n_{g_S})^2}}$. Applying the push forward map $F_*$ and taking projection, we obtain 
$$F_*(T_g^c+\nu_g)^T=\ell F_*(N_g)~\mbox{ at }p\in\Si.$$ 
Here we use the assumption that $F_*(n_{g_S})\in{\rm span}\{\nu_{g_R}\}$. The equation above further yields the compatibility equation \eqref{cibc3} and the second equation in \eqref{cibc4}.

\bibliographystyle{amsplain}
\bibliography{Bib2021}

\end{document}